\documentclass[a4paper,12pt,intlimits,oneside]{amsart}
\usepackage{amsmath}
\usepackage{amsthm}
\usepackage{latexsym}
\usepackage{amssymb}
\usepackage{xcolor}
\usepackage{graphicx}
\usepackage[colorlinks=true]{hyperref}
\numberwithin{figure}{section}
\def\R{{\mathbb R}}
\def\C{{\mathbb C}}

\def\T{{\mathbb T}}
\def\Z{{\mathbb Z}}

\def\N{{\mathbb N}}

\def\e{\varepsilon}

\def\build#1_#2^#3{\mathrel{
\mathop{\kern 0pt#1}\limits_{#2}^{#3}}}
\def\td_#1,#2{\mathrel{\mathop{\build\longrightarrow_{#1\rightarrow #2}^{}}}}
\DeclareFontFamily{U}{MnSymbolC}{}
\DeclareSymbolFont{MnSyC}{U}{MnSymbolC}{m}{n}
\DeclareFontShape{U}{MnSymbolC}{m}{n}{
    <-6>  MnSymbolC5
   <6-7>  MnSymbolC6
   <7-8>  MnSymbolC7
   <8-9>  MnSymbolC8
   <9-10> MnSymbolC9
  <10-12> MnSymbolC10
  <12->   MnSymbolC12}{}
\DeclareMathSymbol{\intprod}{\mathbin}{MnSyC}{'270}
\newtheorem{theorem}{Theorem}
\newtheorem{corollary}{Corollary}[section]
\newtheorem{proposition}{Proposition}[section]
\newtheorem{lemma}{Lemma}[section]
\newtheorem{remark}{Remark}[section]
\newtheorem{definition}{Definition}[section]
\numberwithin{equation}{section}
\begin{document}
\title[On the integrability of the BO equation]{On the integrability of the Benjamin--Ono equation on the torus}
\author[P. G\'erard]{Patrick G\'erard}
\address{Laboratoire de Math\'ematiques d'Orsay, Univ. Paris-Sud, CNRS, Universit\'e Paris--Saclay, 91405 Orsay, France} \email{{\tt patrick.gerard@math.u-psud.fr}}
\author[T. Kappeler]{Thomas Kappeler}
\address{Institut f\"ur Mathematik, Universit\"at Z\"urich, Winterthurerstrasse 190, 8057 Zurich, Switzerland} \email{{\tt thomas.kappeler@math.uzh.ch}}

\subjclass[2010]{ 37K15 primary, 47B35 secondary}

\date{October 4, 2019}

\begin{abstract}
In this paper we prove that the Benjamin-Ono equation, when considered on the torus,
is an integrable (pseudo)differential equation in the strongest possible sense: it admits 
global Birkhoff coordinates on the space $L^2(\T)$. These are coordinates which allow 
to integrate it by quadrature and hence are also referred to as nonlinear Fourier 
coefficients. As a consequence, all the $L^2(\T)$ solutions of the Benjamin--Ono equation are almost periodic functions of the time variable. The construction of such coordinates relies on the spectral study of the Lax operator in the Lax pair formulation of the Benjamin--Ono equation and on the use of a generating functional, which encodes the entire Benjamin--Ono hierarchy.
\end{abstract}

\keywords{Benjamin--Ono equation, Lax pair,  Hardy space, Toeplitz operators}

\thanks{This paper was started while both authors were visiting the Pauli Institute for a workshop organised by J.C. Saut, whom we warmly thank. The last part  was written while the first author was visiting the Institut Mittag--Leffler. We are grateful to these institutions for providing favourable working conditions.}

\maketitle

\medskip

\tableofcontents

\medskip

\section{Introduction}\label{Introduction}
In this paper we consider the Benjamin-Ono (BO) equation on the torus,
\begin{equation}\label{BO}
\partial_t u = H\partial^2_x u - \partial_x (u^2)\,, \qquad x \in \T:= \R/2\pi\Z\,, \,\,\, t \in \R,
\end{equation}
where $u\equiv u(t, x)$ is real valued and $H$ denotes the Hilbert transform, defined for $f = \sum_{n \in \mathbb Z} \widehat f(n) e^{ i n x} $,
$\widehat f(n) = \frac{1}{2\pi}\int_0^{2\pi} f(x) e^{-  i n x} dx$,  by
$$
H f(x) := \sum_{n \in \Z} -i \ \text{sign}(n) \widehat f(n) \ e^{inx}
$$
with $\text{sign}(\pm n):= \pm 1$ for any $n \ge 1$, whereas $\text{sign}(0) := 0$. This pseudodifferential equation ($\Psi$DE)
in one space dimension has been introduced by Benjamin \cite{Benj} and Ono \cite{Ono} to model long, one-way internal gravity waves 
in a two-layer fluid.
It has been extensively studied, in particular the wellposedness problem on the torus as well as on the real line. 
On appropriate Sobolev spaces, the BO equation \eqref{BO} can be written in Hamiltonian form  
$$
\partial_t u = \partial_x (\nabla \mathcal H (u))\,, \qquad  \mathcal H (u):= \frac{1}{2\pi}\int_0^{2 \pi} \left [\frac{1}{2} (|\partial_x|^{1/2} u)^2 - \frac{1}{3} u^3 \right ]dx
$$
where $|\partial_x|^{1/2}$ is the square root of the Fourier multiplier operator $|\partial_x|$ 
given by
$$
|\partial_x| f(x) = \sum_{n \in \Z} |n| \widehat f(n) e^{inx}\,.
$$
Note that the $L^2-$gradient $\nabla \mathcal H$ of $\mathcal H$ can be computed to be $|\partial_x| u - u^2$ and that $\partial_x \nabla \mathcal H$ is the
Hamiltonian vector field
corresponding to the Gardner bracket, defined for any two functionals $F, G : L^2 \to \C$ with sufficiently regular $L^2-$gradients by
$$
 \{F, G \} := \frac{1}{2 \pi} \int_0^{2\pi} (\partial_x \nabla F) \nabla G dx\ .
$$
The main result of this paper says that the BO equation \eqref{BO}
is an integrable $\Psi$DE. In fact, we show that it admits global Birkhoff coordinates
and hence is an integrable $\Psi$DE in the strongest possible sense. 
To state our results more precisely, we first need to introduce some notation.
Denote by $L^2$ the $\C-$Hilbert space $L^2(\T, \C)$ of $L^2-$integrable, complex valued functions with the
standard inner product 
\begin{equation}\label{L2 inner product}
\langle f | g \rangle :=\frac{1}{2\pi}\int_0^{2\pi}f(x)\overline{g(x)}\, dx\ 
\end{equation}
and the corresponding norm $\| f\| := \langle f | f \rangle ^{1/2}$.
Furthermore, we denote by $L^2_r$ the $\R-$Hilbert space $L^2(\T, \R)$, consisting of elements $u \in L^2$ which are real valued
and let
$$
L^2_{r, 0} := \{  u \in L^2_r\, | \, \langle u | 1 \rangle = 0 \}  \  .
$$
For any subset $J \subset \mathbb Z_{\ge 0}$ and any $s \in \mathbb R$, 
$h^s(J) \equiv h^s (J, \mathbb C)$ denotes the weighted $\ell^2-$sequence space
$$h^s(J) = \{ (z_n)_{n \in J} \subset \mathbb C \, : \, \| (z_n)_{ n \in J} \|_s < \infty  \}
$$
where
$$
\| (z_n)_{ n \in J} \|_s : = \big( \sum_{n \in J} \langle n \rangle^{2s} |z_n|^2 \big)^{1/2} \ , \quad  
\langle n \rangle := \text{ max} \{ 1, |n| \} \, .
$$
In case where $J = \mathbb N := \{ n \in \mathbb Z \, : \, n \ge 1 \}$ we write
$h^s_+$ instead of $h^s(\mathbb N)$.
In the sequel, we view $h^s_+$ as the $\R-$Hilbert space $h^s(\N, \R) \oplus h^s(\N, \R)$
by identifying a sequence $(z_n)_{n \in \N} \in h^s_+$ with the pair of sequences
$\big( ({\rm Re} \, z_n)_{n \in \N}, ({\rm Im} \, z_n)_{n \in \N} \big)$ in $h^s(\N, \R) \oplus h^s(\N, \R)$.
\bigskip

The main result of this paper then reads as follows.

\begin{theorem}\label{main result} 
There exists a homeomorphism 
$$
\Phi : L^2_{r,0}\to h^{1/2}_+ \,, \, u \mapsto (\zeta_n(u))_{n \ge 1}
$$
so that the following holds:\\
(B1) For any $n \ge 1$, $\zeta_n:L^2_{r,0}\to \C $ is real analytic.\\
(B2) The Poisson brackets between the coordinate functions $\zeta_n$ are well defined and for any $n, k \ge 1,$
$$
\{\zeta_n , \overline{\zeta_k} \} = - i \delta_{nk}\,, \qquad  \{\zeta_n , \zeta_k \} = 0\,.
$$ 
(B3) On its domain of definition, $\mathcal H \circ \Phi^{-1}$ is a (real analytic) function, which only depends on the actions $|\zeta_n|^2,$ $n \ge 1$.\\
The coordinates $\zeta_n$ are referred to as complex Birkhoff coordinates.
\end{theorem}
\begin{remark} (i) The Birkhoff map $\Phi$ is bounded, meaning that for any bounded subset $B$ of $L^2_{r,0}$, the image $\Phi(B)$ is bounded.
Indeed, this is a direct consequence of the trace formula of Proposition \ref{formulae}, saying that for any $u \in L^2_{r,0}$, 
$$
\|u \|^2 = 2 \sum_{n =1}^\infty n |\zeta_n|^2 \ .
$$  
In analogy to Parseval's identity in Fourier analysis, we refer to this trace formula as {\em Parseval's identity for the nonlinear map $\Phi$}.\\
(ii) When restricted to submanifolds of finite gap potentials (cf. Definition \ref{definition finite gap}),
the map $\Phi$ is a canonical, real analytic diffeomorphism
onto corresponding Euclidean spaces -- see 
Theorem \ref{Ngap} in Section \ref{finite gap potentials} for details.
\end{remark}

\bigskip

\noindent
In subsequent work \cite{GKT} we plan to further study
the regularity of the Birkhoff map of
Theorem \ref{main result}
and its restrictions to the scale of Sobolev spaces $H^s_{r,0}$, $s \ge 0$, where
$$
H^s_{r,0} := \{ u \in L^2_{r,0}\, | \, \|u\|_s < \infty \}\, , \qquad    \|u\|_s:= \big( \sum_{n \in \Z} \langle n \rangle^{2s} |\widehat u(n) |^2 \big)^{1/2}
$$
and to make a detailed analysis of the solution map of the Benjamin--Ono equation, including qualitative properties 
of solutions of \eqref{BO} such as  their long time behaviour. 
As an immediate application of Theorem \ref{main result} in this direction we mention the following result on the 
solutions of the BO equation for initial data in $L^2_{r,0}$ obtained by quadrature,
when  equation \eqref{BO} is expressed in Birkhoff coordinates. 
\begin{theorem}\label{almost periodic} For every initial data $u(0)$ in $L^2_{r,0},$ the solution of the BO equation
with initial data $u(0)$,
$$
t\in \R \mapsto u(t)\in L^2_{r,0} \, ,
$$ 
is almost periodic. Its orbit is relatively compact in 
$L^2_{r,0}$.
\end{theorem}
\begin{remark}
(i) The solutions of the BO equation of Theorem \ref{almost periodic} coincide with the ones
obtained by Molinet in \cite{Mol} (cf. also \cite{MP}). Since the average is conserved by \eqref{BO},
the results of Theorem \ref{almost periodic} easily extend to solutions with initial data in 
$L^2_r$.\\
(ii)
Expressing a given solution $t \in \R \mapsto u(t) \in L^2_{r,0}$ of \eqref{BO}  in Birkhoff coordinates one immediately sees that Corollary 1.3 in 
\cite{TV} on the recurrence of solutions extends in the sense that for any initial datum $u(0) \in L^2_{r,0}$ 
there exists a sequence
$0 < t_1 < t_2 < \cdots$ of times with $t_n \to \infty$ so that 
$\lim_{n \to \infty}\| u(t_n) - u(0)\| = 0$.\\
(iii) As another immediate application of Theorem \ref{main result}, we present a new proof of the result due to 
Amick$\&$Toland \cite{AT}, characterizing the traveling wave solutions of the BO equation as solutions with initial data given by one gap potentials -- see Proposition \ref{traveling} in 
Appendix \ref{general one gap potentials} for details.
\end{remark}

\bigskip

\noindent
{\em Outline of the proof of Theorem \ref{main result}:}  At the heart of the proof is the Lax pair formulation 
$\partial_t L_u = [B_u, L_u]$ of the BO equation, discovered by Nakamura \cite{Nak} and Bock\&Kruskal \cite{BK}.
It is reviewed at the beginning of Section \ref{Lax operator} -- see also Appendix \ref{verification Lax pair equation}.
For any $u \in L^2_r,$ $L_u$ is a selfadjoint pseudodifferential operator of order one whose spectrum
is conserved along the flow of the BO equation.
In contrast to integrable PDEs with a Lax pair formulation such as the Korteweg-de Vries or the nonlinear Schr\"odinger equation,
the Lax operator $L_u$ for the $BO$ equation is not a differential operator.
When considered on the Hardy space $L^2_+ := \{ h \in L^2 \, : \, \widehat h(n) = 0 \,\, \forall n < 0 \}$,
the operator $L_u$ is given by $L_u = -i \partial_x - T_u$ where $T_u$ is the Toeplitz operator $T_uh =  \Pi(uh)$
and $\Pi$ denotes the Szeg\H{o} projector $\Pi : L^2_r \to L^2_+$.  Its spectrum consists of real
eigenvalues which when listed in increasing order and with their multiplicities take the form
$\lambda_0 \le \lambda_1 \le \lambda_2 \le \ldots$ with $\lambda_n \to \infty$ as $n \to \infty$.
In our analysis of the spectrum of $L_u$, the shift operator $S: L^2_+ \to L^2_+, h \mapsto e^{ix}h$
plays an important role. We point out that this operator played also a major role in the study of the Szeg\H{o} equation (cf. \cite{GG0}, \cite{GG1}).
Using commutator relations between $L_u$ and $S$ (cf. Proposition \ref{simple}) we infer that 
for any $u \in L^2_r$ 
$$
\gamma_n(u) := \lambda_n(u) - \lambda_{n-1}(u) -1 \ge 0 \,, \quad \forall \, n \ge 1 \, .
$$
Since all eigenvalues are simple, $\gamma_n \equiv \gamma_n(u)$, $n \ge 1,$ can be shown
to depend real analytically on $u$. Note that the eigenvalues $\lambda_n$, $n \ge 1$, can then be expressed
in terms of the $\gamma_k$'s by
$\lambda_n = \lambda_0 + n + \sum_{k=1}^{n} \gamma_k.$
Furthermore, we provide an orthonormal basis of eigenfunctions $f_n$, where each $f_n$ depends 
analytically on $u$ (Definition \ref{on basis}).

A key ingredient of the proof of Theorem \ref{main result} is the generating function $\mathcal H_\lambda (u)$, which we study in Section \ref{trace}.
For any $u \in L^2_r$,  $\lambda \mapsto  \mathcal H_\lambda (u)$ is the meromorphic function, defined by
$$
\mathcal H_\lambda (u):=\langle (L_u+\lambda Id )^{-1}1\, | \, 1\rangle \ .
$$
Note that $\mathcal H_\lambda (u)$ is holomorphic on $\mathbb C \setminus 
\{ - \lambda_n(u) \, : \, n \ge 0\}$ and might have simple poles at $\lambda =-\lambda_n(u), n\ge 0$.
By Proposition \ref{formulae}, $\mathcal H_\lambda (u)$ can be expressed as a function of $\lambda_0$ and the $\gamma_n$'s, 
and hence
is a family of integrals of the BO equation, parametrized by the spectral parameter $\lambda$.
For our study of the BO equation it plays a role comparable to the one of the discriminant in the normal form theory of the KdV equation.
In particular, $\mathcal H_\lambda$ admits a product representation (Proposition \ref{formulae}) and an expansion at $\lambda = \infty$ 
whose coefficients constitute the BO hierarchy. The functionals $\langle u | 1 \rangle$, $\|u\|^2$ as well as the BO Hamiltonian $\mathcal H$ are elements in this hierarchy.
As an application we derive a trace formula for $\|u\|^2$ (Proposition \ref{formulae}), implying that for any $u \in L^2_{r,0}$ the sequence $(\sqrt{\gamma_n})_{n \ge 1}$ 
is in $h^{1/2}_+$. Furthermore, we show that for any 
$n \ge 0,$ the absolute value $| \langle 1 | f_n \rangle | $ is a function of $\lambda_0$ and the $\gamma_n$'s alone  (Corollary \ref{1,f}).

Our candidates for the Birkhoff coordinates $(\zeta_n)_{n \ge 1}$ are then obtained by normalizing $(\langle 1 | f_n \rangle)_{n \ge 1}$ so that 
$| \zeta_n | = \sqrt{\gamma_n}$ for any $n \ge 1$ (cf. Section \ref{Birkhoff coordinates}) and the corresponding action angle coordinates 
are $(\gamma_n)_{n \ge 1}$, $({\rm arg} \langle 1 | f_n \rangle)_{n \ge 1}$. These variables are very reminiscent of the canonical variables
constructed for the Toda system with Dirichlet boundary condition in \cite{M}. In fact, the Toda system admits a Lax pair formulation where
the Lax operator $L$ is a Jacobi matrix. In the case of Dirichlet boundary conditions, all eigenvalues of $L$ are simple. 
Canonical variables are then given by these eigenvalues and the first components of appropriately normalized eigenvectors -- see \cite{M} for more details.

In the sequel, the map $\Phi: u \mapsto (\zeta_n(u))_{n \ge 1}$, defined on $L^2_{r,0}$ and with values in $h^{1/2}_+$, is studied in detail. 
A key observation is that any $u \in L^2_{r,0}$ can be reconstructed from $\langle 1 | f_n \rangle,$ $n \ge 0$, allowing to prove that
$\Phi$ is one to one (formula \eqref{recover}, Proposition \ref{injectivity}). 
The generating function also plays a major role for proving that $(\zeta_n)_{n \ge 1}$, $(\overline \zeta_n)_{n \ge 1}$ satisfy the canonical relations
stated in Theorem \ref{main result}.
Since for any $n \ge1,$ the $L^2-$gradient of $\nabla \zeta_n$ is in $H^1$ (Proposition \ref{zeta_n in H^1}), it follows 
that the Poisson brackets among these functionals are well defined. A key ingredient for proving that they satisfy canonical relations, 
is the observation that, for $\lambda$ sufficiently large, the $\Psi DE$ $\partial_t u = \partial_x \nabla \mathcal H_\lambda(u)$ admits 
the Lax pair formulation $\frac{dL_u}{dt}=[B_u^\lambda , L_u]$ where $B_u^\lambda$ is a certain skew adjoint operator
(Proposition \ref{laxpairHlambda}). As a consequence, the eigenvalues $\lambda_p$ commute with $\mathcal H_{\lambda}$, implying that
$\{ \gamma_p, \gamma_n \}$ vanishes on $L^2_{r,0}$ for any $p, n \ge 1$ (Corollary \ref{bracketlambda}). In a similar fashion 
one computes $\{ \mathcal H_\lambda,  \langle 1, f_n \rangle \}$ and then derives that for any $p\ge 1, n\ge 0,$ 
$\{ \gamma_p, \langle 1, f_n \rangle\} = i\, \langle 1,  f_n \rangle \delta_{pn}$ on $ L^2_{r,0}$ (Proposition \ref{gamma,1,f}).

As a last key ingredient into the proof of the claimed canonical relations for
$(\zeta_n)_{n \ge 1}$, $(\overline \zeta_n)_{n \ge 1}$ we consider finite gap potentials:
we say that $u \in L^2_r$ is a {\em finite gap potential} if the set $\{ n \ge 1 \, : \, \gamma_n(u) > 0 \}$ is finite.
We then show that for any $N \ge 1,$ the restriction of $\Phi$ to the submanifold $\mathcal U_N$ of special finite gap potentials, defined by
$$\mathcal U_N := \{ u \in L^2_{r,0} \, : \, \gamma_N(u) > 0, \, \gamma_n(u) = 0 \,\, \forall \, n > N \}\ , $$
is a symplectic diffeomorphism onto $\C^{N-1} \times \C^\ast$ 
(cf. Theorem \ref{Ngap} in Section \ref{finite gap potentials}).
\vskip 0.25cm
 In the course of the proof we 
derive a formula for an arbitrary element in $\mathcal U_N$, showing in particular that $\Pi u$ is of the form $\Pi u(x) = R(e^{ix})$ where $R$ is a rational function,
$R(z) = - z \frac{Q'(z)}{Q(z)}$, and $Q(z)$ a polynomial of degree $N$ with roots outside the unit disc -- see \eqref{formula Pi u}, Remark \ref{formula Ngap}. We note that $\bigcup_{N \ge 0} \mathcal U_N$ is the set of all
finite gap potentials in $L^2_{r,0}$ and is dense in $L^2_{r,0}$.
In Section \ref{proof of Theorem 1}  we make a synopsis of the proof of Theorem \ref{main result}
and as an application of Theorem \ref{main result} show Theorem \ref{almost periodic}.

\bigskip

\noindent
{\em Related work:}
The Benjamin-Ono equation has been extensively studied. 
For an excellent account we refer to the recent survey by 
J.C. Saut \cite{Sa}.
Besides the foundational work of Benjamin \cite{Benj} and Ono \cite{Ono}, let us point out a few highlights,
relevant in the context of the present paper.
The first results about \eqref{BO} by PDE methods
can be found in \cite{Sa0}, \cite{ABFS} and the most recent ones in \cite{Mol}, \cite{MP}, \cite{IT} -- we refer
to \cite{MP} and \cite{Sa} for numerous other references.\\
Concerning the investigation of the integrability of \eqref{BO}, besides the discovery of the Lax pair formulation of this equation
already referred to above (cf. \cite{Nak}, \cite{BK}), we mention the pioneering work of
Fokas$\&$Ablowitz \cite{FA}, Coifman$\&$Wickerhauser
\cite{CW}, Kaup$\&$Matsuno \cite{KM}, and the more recent contributions by Wu \cite{Wu}, \cite{Wu2},  on the scattering transform for the BO equation, which concerns the integrability of the BO equation when considered on the real line. Surprisingly, the integrability of the BO equation on the torus has not been studied in detail so far. We point out the work by Satsuma$\&$Ishimori \cite{SI} on the construction of multi-phase solutions of \eqref{BO} by Hirota's bilinear method,
the one by Amick$\&$Toland \cite{AT} on the characterisation of the solitary wave solutions as well as the work
by Dobrokhotov$\&$Krichever \cite{DK} where multi-phase solutions are constructed by the method of finite zone integration. We refer to these solutions as finite gap solutions and they are treated in detail in Section \ref{finite gap potentials}. 
In their work on the quantum BO equation,
Nazarov\&Sklyanin \cite{NS} introduced a generating function for the classical BO hierarchy of the type introduced in Section \ref{trace}, which admits a quantum analogue. 
The approach of Nazarov\&Sklyanin has been further developed by Moll \cite{Moll} in the context of the classical BO equation.
Finally, we mention the work
by Tzvetkov$\&$Visciglia on the construction of invariant Gaussian measures and applications to the long time behaviour of solutions of \eqref{BO} 
(cf. \cite{TV} and references therein).\\
Our analysis of the integrability of \eqref{BO}
borrows from the one carried out for the Korteweg-de Vries and the defocusing nonlinear Schr\"odinger equations
as well as the Szeg\H{o} equation. As the BO equation, 
the Korteweg-de Vries and the defocusing nonlinear Schr\"odinger equations are integrable in the sense that they
admit global Birkhoff coordinates -- see \cite{KP1}, \cite{GK1} and references therein.
Furthermore, it turned out that the Szeg\H{o} equation, an integrable $\Psi$DE introduced and studied in detail by G\'erard and Grellier,
is closely related to the BO equation and tools developed for the study of the integrability of this equation 
were of great use in the present paper -- cf. \cite{GG0}, \cite{GG1}, \cite{GG2} and references therein.

\vskip0.5cm

\noindent
{\em Notation:} We have already introduced the Hilbert spaces $L^2,$ $L^2_r$,$L^2_{r,0}$, and the Hardy space $L^2_+$
as well as the Szeg\H{o} projector $\Pi : L^2 \to L^2_+$.
 More generally,  for any $1 \le p \le \infty, $
 $L^p \equiv L^p(\T, \C)$ denotes the standard complex $L^p-$space  and the $L^p-$norm of an element $f \in L^p$
 is denoted by $\|f\|_{L^p}$. The subspace  of real valued functions in $L^p$ is denoted by $L^p_r \equiv L^p(\T, \R)$.
 Furthermore, we have already introduced the Sobolev spaces $H^s_{r,0}$. More generally, we denote by $H^s$ the standard
 Sobolev space $H^s(\T, \C)$ and by $H^s_r$ the corresponding real subspace. 
 In order to define the $L^2-$ gradient of a differentiable functional $\mu: f \mapsto \mu(f) = \text{Re} \mu(f) + i \text{Im} \mu(f)  \in \C$ , 
 defined on $L^2_r$, introduce the complex bilinear form,
 $$
\langle f , g \rangle :=\frac{1}{2\pi}\int_0^{2\pi}f(x) g(x)\, dx\ , \qquad \forall \,\, f, g \in L^2\ .
 $$
Note that the function $g$ in the latter integral is not complex conjugated. The $L^2-$gradient $\nabla \mu (f)$ of $\mu$ at $f \in L^2_r$ 
is then defined as the element in $L^2$, uniquely determined by
$$
\langle \nabla \mu (f) , g \rangle = d \mu (f)[g]\ , \qquad \forall \, g \in L^2_r \ .
$$
The $L^2-$gradient of a functional, defined on a subspace of $L^2_r$ such as a Sobolev space, is defined in a similar fashion
by extending the bilinear form $\langle f , g \rangle$ as a dual paring between the subspace and its dual.
In case the functional $\mu$ is defined on a complex neighborhood of $L^2_r$ in $L^2$ and analytic there, the $L^2-$gradient
of $\mu$ is defined in the same way. 

Finally, we recall that we have already introduced the sequence spaces $h^s(J)$ where $J \subset \Z$.
In case $s=0$, we write $\ell^2(J)$ for $h^0(J)$ and endow it with the standard inner product
$$
\langle (z_n)_{ n \in J} | (w_n)_{ n \in J} \rangle_{\ell^2} = \sum_{n \in J} z_n \overline{w_n}\, .
$$
More generally, for any $1 \le p < \infty$ and $s \in \R,$ denote by $\ell^{p, s}(J) \equiv \ell^{p, s}(J,\C)$ the weighted $\ell^p-$sequence space,
consisting of sequences $\xi = ( \xi_n)_{n \in J} \subset \C$ so that 
$$
 \| \xi \|_{\ell^{p,s}} := \big( \sum_{n \in J} \langle n \rangle^{ps} |\xi_n|^p \big)^{1/p} < \infty\ .
$$
Correspondingly we define the sequence spaces $\ell^{p, s}(J, \R)$. 
The spaces $\ell^{\infty, s}(J) \equiv 
\ell^{\infty, s}(J,\C)$ and $\ell^{\infty, s}(J,\R)$ 
are defined in a similar fashion.
In the sequel we will often use the notation
 $\ell^p_n$ to denote the $n$--th element of a sequence in $\ell^p$.

\section{The Lax operator}\label{Lax operator}

Nakamura \cite{Nak} and Bock\&Kruskal \cite{BK} discovered that the BO equation admits a Lax pair. 
In the periodic setup it can be described as follows.
For any $u \in L^2_r$, introduce the unbounded operators $L_u$ and $B_u$ on $L^2_+$. 
\begin{equation}\label{Lax pair for BO}
L_u  = D  - T_u \,, \qquad B_u  :=  i D^2  + 2i T_{ D( \Pi u)}  - 2i D T_u  \,.
\end{equation}
where $D$ is the unbounded operator on $L^2_+$, defined by
$D:=-i\partial_x$,
$T_u$ the Toeplitz operator on $L^2_+$  with symbol $u$,
$$
T_u :  L^2_+ \to L^2_+, \, h \, \mapsto \, T_uh = \Pi (uh)
$$ 
 and $\Pi : L^2 \to L^2_+$ the Szeg\H{o} projector.
 For any $u \in L^2_r$, $T_u$ is an unbounded operator, satisfying the estimate
$$\| T_u(h)\|\leq \|u\|\, \|h\|_{L^\infty}\ .$$
Hence $L_u= D - T_u$ is an unbounded operator with domain $H^1_+ := H^1 \cap L^2_+$ 
and $B_u$ is such an operator with domain 
$H^2_+ := H^2 \cap L^2_+$.
Here
$H^m\equiv H^m(\mathbb T, \mathbb C)$ denotes the standard Sobolev space of order $m \ge 1$. 
Furthermore, it follows that $T_u$ is $D-$compact.
Arguing as in Appendix \ref{L_u on the line} 
(cf. Lemma \ref{relative bound}) one verifies that $L_u$ is selfadjoint and $B_u$ skew-adjoint. The Lax pair formulation of the BO equation then reads
\begin{equation}\label{Lax pair}
\frac{d}{dt} L_u = [B_u, L_u]\,.
\end{equation}
\begin{remark}\label{Btilde}
In order to make the paper self-contained, we verify the Lax pair formulation of the BO equation in Appendix A.
Clearly,  the operator  $B_u$ is not uniquely determined. Instead of $B_u$ one might take any operator of the form $B_u + P_u$
where $P_u$ is a skew adjoint operator commuting with $L_u.$ In particular, one might choose instead of $B_u$ the operator 
$\tilde B_u:=B_u - i L_u^2.$ By a straightforward computation one gets 
$$\tilde B_u=  i(T_{ |\partial_x| u}  -  T_u^2)\ ,
$$ 
which is a pseudodifferential operator of order $0$. This Lax pair formulation will be also a consequence of Proposition \ref{laxpairHlambda}.
\end{remark}
The Lax pair formulation of the BO equation implies that at least formally, the spectrum of the operator $L_u$ is left invariant by the solution map of the BO equation.
For this reason, we want to analyze it in more detail.
In addition of being selfadjoint,  $L_u$ has a compact resolvent. 
Furthermore note that for $u=0$, $L_0 = D$ is a nonnegative operator and its spectrum consists of simple eigenvalues,
which we list in  increasing order, $\lambda_n = n$, $n \ge 0$. Since $T_u$ is $D$--compact
it then follows that for any $u \in L^2_r,$ $L_u$ is bounded from below and its spectrum consists of eigenvalues
which are bounded from below and have finite multiplicities. 
We list these eigenvalues in increasing order 
and with their multiplicities so that
$$ \lambda_n(u)  \le \lambda_{n+1}(u) , \quad \forall n \ge 0.$$
Our first result says that all eigenvalues of $L_u$ are simple. More precisely, we have the following

\begin{proposition}\label{simple} For any $u \in L^2_r$ and $n \ge 0$, the eigenvalues $\lambda_n \equiv \lambda_n(u)$ satisfy
$$\lambda_{n+1} \geq \lambda_n + 1\ .$$
As a consequence, all eigenvalues of $L_u$ are simple.
\end{proposition}
\begin{remark}\label{constant potentials} Since for any $u \in L^2_r$, 
$ \langle  L_u 1 | 1 \rangle =  -  \langle u | 1 \rangle$, 
 it follows from the variational characterization of $\lambda_0$
that $ \lambda_0 \le -  \langle  u | 1 \rangle$.  Furthermore, if $\lambda_0 = -  \langle  u | 1 \rangle$, then $f_0 = 1$ is a normalized eigenfunction corresponding to $\lambda_0$.
 Since $L_u 1 = - \Pi u$, one has $\Pi u =  \langle  u | 1 \rangle 1$, implying that $u =  \langle  u | 1 \rangle$.
\end{remark}

The proof of Proposition \ref{simple} is based on properties of the shift operator $S$ and its commutator relations with  $L_u$. More precisely,
$S$ is the operator on $L^2_+$, defined by
$$S : L^2_+ \to L^2_+, \, h \, \mapsto \, Sh(x)={\rm e}^{ix}h(x)\ .$$
Its adjoint with respect to the $L^2-$inner product \eqref{L2 inner product} is given by
$$S^*=T_{{\rm e}^{-ix}}\ ,$$
implying that the following identities hold
\begin{equation}\label{SS*}
S^*S=Id\ ,\quad \ SS^*=Id - \langle \, \cdot \, | 1\rangle 1\ .
\end{equation}
Furthermore, since for any $f \in L^2$, 
$\Pi (e^{ix} f)=S\Pi (f)+\langle e^{ix} f | 1\rangle $
one has for any $h\in H^1_+$,
\begin{equation}\label{T_uS}
T_u(Sh)=ST_u h+ \langle uS h | 1\rangle 1\,.
\end{equation}
Using that  $S^*1=0$, one then gets
\begin{equation}\label{conjugation of T_u}
S^*T_uS=T_u 
\end{equation}
and consequently, since $DS=SD+S$ and hence $S^*DS = D + Id$,
\begin{equation}\label{S*LS}
S^*L_uS=L_u+ Id\,.
\end{equation}
Moreover, combining  $DS=SD+S$ with \eqref{T_uS}, one sees that
\begin{equation}\label{L_uS}
L_u S = S L_u + S  - \langle \,  uS  \cdot \, | 1 \rangle 1\ .
\end{equation}
With these preparations we are now ready to prove Proposition \ref{simple}.
\begin{proof}
Apply the max-min formula for $\lambda_n$,
$$\lambda_n=\max_{\dim F=n}m(F)\ ,\ m(F):=\min \{  \langle  L_u h | h \rangle \, : \,  h\in H^1_+\cap F^\perp \ ,\ \|h\|=1\}\ ,$$
where, in the above maximum, $F$ describes all  $\mathbb C-$subspaces of $L^2_+$ of dimension $n$. If $F$ is such a vector subspace, observe that
$G(F):=\C 1\oplus S(F)$
is a vector subspace of dimension $n+1$, and that $G(F)^\perp \cap H^1_+$ is precisely 
$S(F^\perp \cap H^1_+)$.
Therefore, we have
$$m(G(F))=\min \{ \langle  L_u Sh | S h \rangle\, : \,  h\in H^1_+\cap F^\perp \ ,\ \|h\|=1\}\ .$$
From \eqref{S*LS}, we infer
$\langle L_uSh,Sh \rangle = \langle  L_u h | h \rangle+\|h\|^2$
and consequently
$$
m(G(F)) = m(F)+1\ .$$
Substituting this identity into the above formula for $\lambda_{n+1}$, we conclude that
$
\lambda_{n+1}\geq  \max_{\dim F=n}m(G(F))=\lambda_n+1$.
\end{proof}
We shall prove in Section \ref{gradients} that the eigenvalues $\lambda_n$  are real analytic. At this stage, we establish that they are Lipschitz continuous.
\begin{proposition}\label{lipschitz}
For every $n\ge 0$, the functional $\lambda_n$ is uniformly Lipschitz continuous on bounded subsets of $L^2_r(\T )$. 
\end{proposition}
\begin{proof} For any $h\in H^1_+$,
\begin{equation}\label{comparison}
\vert \langle L_uh,h\rangle -\langle L_vh,h\rangle\vert =\vert \langle (u-v), |h|^2\rangle \vert\leq \| u-v\|\| h\|_{L^4}^2\
\end{equation}
and if in addition $\| h \| = 1$, one has by the Sobolev embedding theorem
\begin{equation}\label{Sobolev}
\| h\|_{L^4}^2\leq C \langle (1+D)^{1/2}h\vert h\rangle \leq  C\langle (1+D)h\vert h\rangle ^{1/2} \ .
\end{equation}
Applying the max-min formula and \eqref{comparison} for $v=0$, we obtain
\begin{eqnarray*}
\lambda_n(u)&\leq &\max_{\dim F=n}\min \{ \langle Dh+C\| u\| \langle (1+D)^{1/2}h \, \vert \, h\rangle \, : 
\, h\in F^\bot \cap H^1_+, \| h\|=1\}\\
&\leq &n+C\| u\|(1+n)^{1/2}\ .
\end{eqnarray*}
Given a subspace $F \subset L^2_+$ with $\dim F=n$,  select 
$h\in H^1_+\cap \oplus_{k=0}^n \ker (L_u-\lambda_k(u) Id)\cap F^\perp $ with $\| h\|=1$. From \eqref{comparison} and \eqref{Sobolev} one then infers that
$$\langle L_vh\vert h\rangle \leq \langle  L_uh\vert h\rangle +C \| u-v\| \langle (1+D)h\vert h\rangle^{1/2}\ .$$
On the other hand, $\langle L_uh\vert h\rangle \leq \lambda_n(u)$, implying that
$$\langle Dh\vert h\rangle =\langle uh\vert h\rangle +\lambda_n(u)\leq C\| u\| \langle (1+D)h\vert h\rangle ^{1/2}+n +C\| u\|(1+n)^{1/2}\ ,$$
yielding
$$\langle Dh\vert h\rangle \leq \tilde C(n+1+\| u\|^2)\ .$$
We then conclude by the max-min formula that
$$\lambda_n(v)\leq \lambda_n(u)+ K\| u-v\| (n+1+\| u\|^2)^{1/2}$$
where $K > 0$ is an absolute constant.
\end{proof}
In view of Proposition \ref{simple} we introduce for any $n \ge 1$ and $u \in L^2_r$,
\begin{equation}\label{gamma}
\gamma_n(u):=\lambda_n(u)-\lambda_{n-1}(u)-1\ ,
\end{equation}
and refer to it as the $n$--th gap length --
see Proposition \ref{Rspectrum} in Appendix
\ref{L_u on the line} for an explanation of this terminology.
It then follows that for any $n \ge 1$
\begin{equation}\label{lower bound ev}
\lambda_n(u) = n + \lambda_0(u) + \sum_{k=1}^n \gamma_k(u) \ge n + \lambda_0(u)\ .
\end{equation}
Next we obtain a useful characterisation for the vanishing of $\gamma_n.$
For any $u \in L^2_r$ and $n \ge 0$, we denote by $f_n$ a $L^2$-normalized eigenfunction of $L_u$ for the eigenvalue $\lambda_n$. Hence
$f_n \in H^1_+$ and $\| f_n\| = 1.$
In view of Proposition \ref{simple}, $f_n$ is determined uniquely up to a phase factor for any $n \ge 0$. 
Later we will show that this indeterminacy can be removed in such a way that the eigenfunctions $f_n$ depend real analytically on $u$.
\begin{lemma}\label{vanishing of gamma_n}
For any $u \in L^2_r$, $\langle f_0 | 1  \rangle \ne 0$. Furthermore, for any $n \ge 1$ and $u \in L^2_r$,  the following statements are equivalent:  
$$(1) \langle f_n | 1  \rangle=0\ ,\ 
(2) L_u(S f_{n-1})=\lambda_nSf_{n-1}\ ,\  
(3) \langle Sf_{n-1}\vert u\rangle =0\ ,\ 
(4) \gamma_n=0\ .
$$
\end{lemma}
\begin{proof} Assume that, for some $n\ge 0$, we have 
$\langle f_n | 1  \rangle= 0$. Then $f_n=Sg_n$ for some $g_n$, and, by \eqref{L_uS}, 
$$
\lambda_nSg_n =L_uSg_n=S(L_ug_n+g_n)-\langle Sg_n\vert u\rangle 1\ .$$
Applying $S^*$ to both sides, we infer, in view of \eqref{SS*}, that 
$$L_ug_n=(\lambda_n-1)g_n\ .$$
In the case $n=0$, the latter identity contradicts the fact that 
$\lambda_0$ is the smallest eigenvalue of $L_u$
and hence $\langle f_0 | 1  \rangle \ne 0$.
In the case $n\ge 1$, the identity implies that $\lambda_n-1=\lambda_{n-1}$ and that $g_n$ is collinear to $f_{n-1}$, 
showing that (1) implies (2).\\
Applying \eqref{L_uS} to $f_{n-1}$ immediately yields the equivalence of (2) and (3), and that (3) implies (4). \\
Finally, assume that $\gamma_n = 0$ for some $n\ge 1$. Then, again by applying \eqref{L_uS} to $f_{n-1}$,
$$
L_uSf_{n-1}  =(\lambda_{n-1}+1)Sf_{n-1} - \langle Sf_{n-1} | u \rangle 1
=  \lambda_nSf_{n-1} - \langle Sf_{n-1} | u\rangle 1 \ .
$$
Taking the inner product with $f_n$ of both sides of the latter identity, one gets
$$
\langle uSf_{n-1} | 1\rangle \langle 1 | f_n \rangle=0\ ,
$$
which means that either (1) holds, or (3) holds, hence (2) holds, so that $Sf_{n-1}$ is collinear to $f_n$, implying  (1). 
\end{proof} 
In a next step we  show that the eigenfunctions $f_n$ can be normalized in such a way that they depend continuously in $u$.
First we need to make some preliminary considerations.
\begin{lemma}\label{formula with langle f_p | 1 rangle}
For any $u \in L^2_r$ and $k, p \ge 0$
$$
(\lambda_p-\lambda_k-1) \langle f_p | Sf_k \rangle = - \langle f_p | 1\rangle \langle u | Sf_k\rangle .
$$
\end{lemma}
\begin{proof}
Recall from \eqref{L_uS} that for any $k \ge 0$,
\begin{equation}\label{LuSf}
L_u(Sf_k)= ( \lambda_k+1) Sf_k- \langle Sf_k\, | \, u \rangle 1\ .
\end{equation}
Computing the inner product of $f_p$ with both sides of \eqref{LuSf} leads to the claimed formula.
\end{proof}

\begin{lemma}\label{nonvanishing lemma}
For any $u \in L^2_r$ and  $n\ge 1$,
$\langle f_n | Sf_{n-1}  \rangle\ne 0\ .$
\end{lemma}
\begin{proof}
Assume that for some $u \in L^2_r$ and $n \ge 1$, $\langle f_n | Sf_{n-1}  \rangle= 0$ . From Lemma \ref{formula with langle f_p | 1 rangle}, we infer
$$
\langle f_n | 1  \rangle\langle u | Sf_{n-1}  \rangle=0\ .
$$
By Lemma 2.1, this implies 
that $Sf_{n-1}$ and $f_n$ are collinear, contradicting the 
assumption $\langle f_n | Sf_{n-1}  \rangle=0$.
\end{proof}

Lemma \ref{vanishing of gamma_n}(i) and Lemma \ref{nonvanishing lemma} allow to define an orthonormal basis of eigenfunctions of $L_u$ as follows.

\begin{definition}\label{on basis} For any $u \in L^2_r$,
$f_0$ is defined to be the $L^2-$normalized eigenfunction of $L_u$ corresponding to the eigenvalue $\lambda_0$, uniquely determined by the condition
$$
\langle 1\vert f_0 \rangle>0\,.
$$ 
The $L^2-$normalized eigenfunction $f_1$ of $L_u$, corresponding to the eigenvalue $\lambda_1$, is then defined to be the eigenfunction uniquely determined by the condition
$$ 
\langle f_1 | Sf_{0}  \rangle >0 \ .
$$
Assume that for any $0 \le k \le n$, the eigenfunctions  $f_k$ have been defined. Then the $L^2-$normalized eigenfunction of $L_u$, 
corresponding to the eigenvalue $\lambda_{n+1}$, is defined by the condition
$$
\langle f_{n+1} | Sf_{n}  \rangle >0 \ .
$$
\end{definition}
In the remaining part of the paper, for any $u \in L^2_r$, $f_n(\cdot\, , u)$, $n \ge 0$, will always denote the orthonormal basis of $L^2_+$ introduced in Definition \ref{on basis}. 
\begin{remark}\label{f_n continuous} It is easy to see that the map $u\mapsto f_n(\cdot\, ,u)$ is continuous from $L^2_r$ to $H^1_+$. Indeed, if $u^{(k)}\to u$ in $L^2_r$, 
$f_n^{(k)}:=f_n(\cdot ,u^{(k)})$ satisfies
$$\| Df_n^{(k)}\|^2=\lambda_n(u^{(k)})\langle f_n^{(k)}\, \vert \,  Df_n^{(k)}\rangle +\langle u^{(k)}f_n^{(k)}| Df_n^{(k)}\rangle $$
which, by Proposition \ref{lipschitz} and Sobolev inequalities, implies that $f_n^{(k)}$ is bounded in $H^1_+$ as $k\to \infty$. By Rellich's theorem and Proposition \ref{lipschitz}, we infer that any weak * limit point  $g_n$ of $(f_n^{(k)})_{k\ge 1}$ in $H^1_+$ satisfies $L_u g_n=\lambda_n(u)g_n$,
$\| g_n\|=1$ and $\| Dg_n\|^2=\lim_{k\to \infty}\| Df_n^{(k)}\|^2$. Furthermore, the conditions
$\langle 1, f_0 \rangle\geq 0 \ ,\ \langle f_{n+1} | Sf_{n}  \rangle \geq 0 $ are also stable by 
weak * convergence in $H^1_+$, implying that
$g_0 = f_0$ (Lemma \ref{vanishing of gamma_n}) 
and $g_n = f_n$ for any $n \ge 1$
(Lemma \ref{nonvanishing lemma}).
In Section 5, we will prove that  $u\mapsto f_n(\cdot\, ,u)$ 
is a real analytic map from $L^2_r$ to $H^1_+$. 
\end{remark}
 The complex numbers 
$\langle 1 | f_n \rangle$, $n \ge 1$, appropriately rescaled, are our candidates for the complex Birkhoff coordinates.
Note that for any $u \in L^2_r,$ one can express $\Pi u$ in terms of $\langle 1 | f_n \rangle$, $n \ge 0$, as follows. Since 
 $f_n$, $n\ge 0$, is an orthonormal basis of $L^2_+$, one has $\Pi u = \sum_{n= 0}^\infty \langle \Pi u | f_n \rangle f_n$. When combined with the identity
\begin{equation}\label{identity in  f_n and u}
\lambda_n \langle 1 | f_n \rangle = \langle 1| L_u f_n\rangle = \langle L_u 1| f_n\rangle = - \langle \Pi u | f_n\rangle = - \langle u | f_n\rangle
\end{equation}
one is led to the following trace formula
\begin{equation}\label{u in eigenbasis}
\Pi u = - \sum_{n= 0}^\infty \lambda_n \langle1 | f_n \rangle f_n 
\end{equation}
and hence
\begin{equation}\label{L 2 norm}
\| \Pi u \|^2 = \sum_{n=0} ^\infty \lambda_n^2  \, | \langle 1 | f_n \rangle |^2
\end{equation}
and
\begin{equation}\label{average}
\langle u \, | \, 1 \rangle 
=   - \sum_{n =0} ^\infty \lambda_n  | \langle 1 | f_n \rangle |^2\ .
\end{equation}
Similarly, the expansion
$1 = \sum_{n =0}^\infty  \langle 1 | f_n  \rangle f_n $ leads to 
$\sum_{n =0}^\infty | \langle 1  |   f_n \rangle|^2 = 1\ .$  

We record the following property of the eigenfunctions $f_n$.
\begin{lemma}\label{f_n for gamma_n = 0}
For any $u \in L^2_r$ with $\gamma_n= 0$ for a given $n \ge 1$, $f_n = S f_{n-1}.$
\end{lemma}
\begin{proof}
By Lemma \ref{vanishing of gamma_n}(iii), for any $n \ge 1$ with $\gamma_n= 0$, the function $S f_{n-1}$ is an $L^2-$normalized eigenfunction of $L_u$, 
corresponding to the eigenvalue $\lambda_n$ and satisfies the condition $\langle f_{n} | Sf_{n-1}  \rangle >0$.
\end{proof}

In analogy with the notion of a finite gap potential,
introduced in the context of the Korteweg-de Vries equation
and the nonlinear Schr\"odinger equation, we define
the corresponding notion in the context of the BO equation.
\begin{definition} \label{definition finite gap}
A potential $u \in L^2_{r}$ is said to be a {\em finite gap} potential if 
 $J(u)= \{ n \ge 1\, : \, \gamma_n(u) > 0 \} \subset \N$ is finite. In case $J(u)$
 consists of one element only, $u$ is referred to as a {\em one gap} potential.
\end{definition}
Note that for any finite gap potential $u$, the expansion \eqref{u in eigenbasis},
is a finite sum. Finite gap potentials play a similar role as trigonometric polynomials in harmonic analysis.
For any given subset $J \subset \N$ introduce
$$
G_J := \{ u \in L^2_{r} : \gamma_j (u)> 0 \,\, \forall j \in J; \,\, \gamma_j(u) = 0 \,\, \forall j  \in \N \setminus J \}\ .
$$
Elements in the set $G_J$ are referred to as $J-$gap potentials.  We refer to Section \ref{finite gap potentials}
for a detailed studied of $J-$gap potentials with $J=\{1, \cdots, N \}$ and to Appendix \ref{general one gap potentials}
for a description of one gap potentials.

\begin{remark} One might ask if there exists a potential $u\in L^2_r$ with the property that $\gamma_n(u)>0$ for every $n\ge 1$. We claim that $u(x)=2\alpha \cos x$ is such a potential for every $\alpha \in \R \setminus
 \{ 0\}$. Indeed, we have for any $n\ge 1$,
$$
\langle Sf_{n-1} | u \rangle =\alpha  \langle Sf_{n-1} | {\rm e}^{ix} \rangle = \alpha \langle f_{n-1} |  1\rangle \,.
$$
Hence, by Lemma \ref{vanishing of gamma_n}, $\langle f_n\vert 1\rangle =0$ if and only if 
$\langle f_{n-1}\vert u\rangle =0$. 
 Since 
$\langle f_0 | 1 \rangle \ne 0$, we conclude that $\langle f_n | 1  \rangle \ne 0$ for every $n$.\\
\end{remark}

We finish this section with a brief discussion of the symmetry induced by translation.
For any $\tau \in \R$, denote by $Q_\tau$ the linear isometry on $L^2_r,$ given by the translation by $\tau,$
$$
Q_\tau: L^2_r \to L^2_r, \ ,\  Q_\tau u (x) := u( x+ \tau)\,.
$$
\begin{lemma} For any $\tau \in \R$, $u \in L^2_r$, and $n \ge 0$, one has
$$
\lambda_n(Q_\tau u) = \lambda_n(u)\,, \qquad
f_n( x, Q_\tau u) = e^{- in\tau}  f_n( x+ \tau, u)\,.
$$
As a consequence,
\begin{equation}\label{identity for (f_n , 1)}
\langle 1\, |\, f_n(\cdot, Q_\tau u) \rangle  = e^{i n\tau}\langle 1\, | \, f_n(\cdot, u) \rangle \,, \quad \forall \, n \ge 0\,.
\end{equation}
\end{lemma}
\begin{proof} For any $\tau \in \R$, $u \in L^2_r$, and $n \ge 0,$ note that $f_n \equiv f_n(\cdot, u)$ satisfies
$$ 
L_{Q_\tau u} (Q_\tau f_n) = Q_\tau (L_u f_n) = \lambda_n(u) Q_\tau f_n\,.
$$
It implies that
$\lambda_n(Q_\tau u) = \lambda_n(u)$
and that $Q_\tau f_n(\cdot, u)$ is an $L^2-$normalized eigenfunction of $L_{Q_\tau u}$, corresponding to the eigenvalue $\lambda_n$.
Furthermore, 
$$
\langle 1 \, | \, f_0( \cdot + \tau, u) \rangle = \langle 1 \, | \, f_0( \cdot , u) \rangle > 0
$$
implying that $f_0(\cdot , Q_\tau u) = f_0 ( \cdot + \tau, u)$. Similarly, for any $n \ge 1,$
$$
\langle e^{- i n\tau} f_{n} ( \cdot + \tau, u) \, | \, S e^{- i (n -1)\tau} f_{n-1} ( \cdot + \tau, u) \rangle 
= \langle f_{n} ( \cdot + \tau, u) \, | \, e^{ i \tau} S f_{n-1} ( \cdot + \tau, u) \rangle
$$
which equals $\langle f_{n} ( \cdot , u) \, | \,  S f_{n-1} ( \cdot , u) \rangle > 0.$
Arguing by induction one then concludes that $f_{n} ( \cdot , Q_\tau u) = e^{- i n \tau} f_{n}(\cdot + \tau, u)$
as claimed.
\end{proof}


\section{Generating function and trace formulae}\label{trace}

One of the main results of this section are trace formulas for $| \langle 1 | f_n \rangle |^2$, stated in Corollary \ref{1,f}.
They will be used to define our candidates of (complex) Birkhoff coordinates by appropriately scaling $\langle 1 | f_n \rangle$. 
A key ingredient for the proof of these trace formulas is the generating function $\mathcal H_\lambda$. We remark that $\mathcal H_\lambda$
plays a role for the normal form theory of the BO equation comparable to the one of the discriminant in the normal form theory of the KdV equation.

For any $u \in L^2_r$, the generating function is the meromorphic function $\lambda \mapsto  \mathcal H_\lambda (u)$, defined by
\begin{equation}\label{Hlambda}
\mathcal H_\lambda (u):=\langle (L_u+\lambda Id )^{-1}1\, | \, 1\rangle \ .
\end{equation}
Note that $\mathcal H_\lambda (u)$ is holomorphic on $\mathbb C \setminus \{ - \lambda_n(u) \, : \, n \ge 0\}$ and might have simple poles at $\lambda =-\lambda_n(u), n\ge 0$. 
Substituting the expansion $1 = \sum_{n = 0}^\infty \langle 1 \, | \, f_n \rangle f_n$ into the expression for  $\mathcal H_\lambda (u)$ one obtains
\begin{equation}\label{expandHlambda}
\mathcal H_\lambda (u)=\sum_{n=0}^\infty \frac{|\langle 1 | f_n \rangle|^2}{\lambda_n+\lambda }\ .
\end{equation}
\begin{proposition}\label{formulae}
For any $u \in L^2_r$, the following identities hold:
$$
(i) \,\,\, \mathcal H_\lambda (u)=\frac{1}{\lambda_0+\lambda}\prod_{n=1}^\infty \big(1-\frac{\gamma_n}{\lambda_n+\lambda}   \big),  \qquad \qquad 
$$
$$
(ii) \,\,\, \| u\|^2- \langle  u | 1 \rangle^2 =  2\sum_{n=1}^\infty n\gamma_n\, , \qquad
{ \langle  u | 1 \rangle}=-\lambda_0-\sum_{n=1}^\infty \gamma_n\ . 
$$
Hence  $\gamma_n(u) = \frac{1}{n} \ell^1_n$ and both
$ \langle  u | 1 \rangle$ and $\| u\|^2$ are spectral invariants.
\end{proposition}
\begin{proof} (i) The proof makes use of the shift operator $S$ and in particular identity \eqref{S*LS}.
Assume that $\lambda $ is real and $\lambda \ge   -\lambda_0 +1$, so that 
$L_u+\lambda Id : H^1_+ \to L^2_+$ is invertible.  
By \eqref{S*LS}, one has
\begin{equation}\label{formula for S^* (L + lambda)S}
S^*(L_u+\lambda Id)S=L_u+(\lambda +1)Id \ ,
\end{equation}
implying that the operator $S^*(L_u+\lambda Id)S : H^1_+ \to L^2_+$ is invertible. To compute its inverse we need the following
\begin{lemma}\label{inverse of S^*AS}
Assume that $A : H^1_+ \to L^2_+$ is positive and selfadjoint so that $A$ and $S^*AS$ are invertible and $A^{-1}$ is positive on $L^2_+$. Then
$$(S^*AS)^{-1}=S^*A^{-1}S - \frac{\langle \, \cdot \, | \, S^*A^{-1}1\rangle }{\langle A^{-1}1\, | \, 1 \rangle} S^*A^{-1}1\ .$$
\end{lemma}
\noindent
{\em Proof of Lemma \ref{inverse of S^*AS}.} Consider $f\in L^2_+$ and let $h:=(S^*AS)^{-1}f$. Applying $S$ to both sides of the equation
$$S^*ASh=f\ ,$$
we infer from \eqref{SS*} that
$$ASh=Sf+\langle ASh | 1\rangle 1 \ ,$$
hence 
\begin{equation}\label{formula Sh}
Sh=A^{-1}Sf+ \langle ASh | 1\rangle A^{-1}1\ .
\end{equation}
In particular, we have
$$\langle A^{-1}Sf | 1 \rangle =- \langle ASh | 1 \rangle \langle A^{-1}1 | 1\rangle \ .$$
Since $A^{-1}$ is positive, $\langle A^{-1}1 | 1 \rangle > 0$ and hence we obtain
$$\langle ASh | 1\rangle =-\frac{\langle A^{-1}Sf | 1\rangle}{\langle A^{-1}1 | 1 \rangle }=-\frac{\langle f | S^*A^{-1}1 \rangle }{\langle A^{-1}1,1\rangle }\ .$$
Substituting this identity into the equation \eqref{formula Sh} for $Sh$ and applying $S^*$ to both sides, we obtain the claimed statement.
 \hfill $\square$
 
\vskip 0.25cm

\noindent
Let us now go back to the proof of item(i).
Since $L_u+\lambda Id$ is selfadjoint and positive 
we can apply Lemma \ref{inverse of S^*AS} to $A:=L_u+\lambda Id$ to get, in view of \eqref{formula for S^* (L + lambda)S},
\begin{align}\label{formula for difference}
(L_u+ & (\lambda +1)Id)^{-1}  - S^*(L_u+\lambda Id)^{-1}S \nonumber \\
& = - \frac{\langle \, \cdot \, | S^*(L_u+\lambda Id )^{-1}1\rangle }{\langle (L_u+\lambda Id)^{-1}1 |  1 \rangle }S^*(L_u+\lambda Id)^{-1}1.
\end{align}
Since the latter operator is of rank 1, it is of trace class and hence so is the operator $(L_u+(\lambda +1)Id)^{-1}-S^*(L_u+\lambda Id)^{-1}S$. 
To compute its trace we write it as a sum of two operators,
$$ (L_u+(\lambda +1)Id)^{-1}-S^*(L_u+\lambda Id)^{-1}S= A_1 + A_2$$
where 
$$ A_1 := (L_u+(\lambda +1)Id)^{-1}-(L_u+\lambda Id)^{-1}\,,$$
$$ A_2 := (L_u+\lambda Id)^{-1}-S^*(L_u+\lambda Id)^{-1}S\,.$$
In view of Proposition \ref{simple}, $A_1$ is of trace class and 
$$
 {\rm Tr} A_1=\sum_{n=0}^\infty \Big( \frac{1}{\lambda_n+\lambda +1}-\frac{1}{\lambda_n+\lambda}\Big)\ .
$$
Since by \eqref{formula for difference}, $A_1 + A_2$ has rank $1$ and hence is of trace class, the operator $A_2$ is also of trace class. 
Computing its trace with respect to the orthonormal basis $({\rm e}^{inx})_{n\ge 0}$ of $L^2_+$, we obtain
$${\rm Tr} A_2 = {\rm Tr}\big( (L_u+\lambda Id)^{-1}-S^*(L_u+\lambda Id)^{-1}S \big)
=\langle (L_u+\lambda Id)^{-1}1 | 1\rangle \ .$$
On the other hand, by \eqref{formula for difference}, ${\rm Tr} (A_1 + A_2)$ equals
$$
{\rm Tr} \big( -\frac{\langle \, \cdot \, | S^*(L_u+\lambda Id)^{-1}1\rangle }
{\langle (L_u+\lambda Id)^{-1}1 | 1\rangle }   S^*(L_u+\lambda Id)^{-1}1 \big) 
= -\frac{\|S^*(L_u+\lambda Id)^{-1}1\|^2}{\langle L_u+\lambda Id)^{-1}1 | 1 \rangle}
$$
and hence by \eqref{SS*}
$$
{\rm Tr} (A_1 + A_2) = -\frac{\|(L_u+\lambda Id)^{-1}1\|^2}{\langle (L_u+\lambda Id)^{-1}1 | 1 \rangle } +
\langle (L_u+\lambda Id)^{-1} 1 | 1 \rangle \ .
$$
Altogether we have proved that
$$\sum_{n=0}^\infty \Big( \frac{1}{\lambda_n+\lambda +1}- \frac{1}{\lambda_n+\lambda} \Big)
= - \frac{\|(L_u+\lambda Id)^{-1}1\|^2}{\langle (L_u+\lambda Id)^{-1}1 | 1 \rangle}
$$
or
\begin{equation}\label{trace1}
\sum_{n=0}^\infty \Big(\frac{1}{\lambda_n+\lambda +1}-\frac{1}{\lambda_n+\lambda} \Big) =
-\ \frac{\displaystyle{\sum_{n=0}^\infty \frac{|\langle 1| f_n \rangle|^2}{(\lambda_n+\lambda)^2}}}{\displaystyle{\sum_{n=0}^\infty \frac{|\langle 1 | f_n \rangle|^2}{\lambda_n+\lambda}}}=\frac{d}{d\lambda}\log \mathcal H_\lambda (u)\ .
\end{equation}
Isolating the term $-(\lambda_0+\lambda)^{-1}$ on the left hand side of the latter identity and then rewriting the remaining sum as a telescopic series yields
\begin{equation}\label{trace2}
\frac{d}{d\lambda}\log \mathcal H_\lambda (u)=-\frac{1}{\lambda_0+\lambda}+\sum_{n=1}^\infty\frac{\gamma_n}{(\lambda+\lambda_{n-1}+1)(\lambda+\lambda_n)}
\end{equation}
or, written in a slightly more convenient form, 
$$\frac{d}{d\lambda}\log \big( (\lambda_0 +\lambda )\mathcal H_\lambda (u) \big) =\sum_{n=1}^\infty\frac{\gamma_n}{(\lambda+\lambda_{n-1}+1)(\lambda+\lambda_n)}\ .$$
Integrating both sides of the latter identity from $\lambda$ to $+\infty$, we infer
$$-\log \big( (\lambda_0 +\lambda )\mathcal H_\lambda (u) \big) =-\sum_{n=1}^\infty\log \big(1-\frac{\gamma_n}{\lambda+\lambda_n}\big)\ ,$$
implying that the formula of item (i) holds for any $\lambda >  |- \lambda_0 + 1|$. By analyticity, it is  then valid for any $\lambda \in \mathbb C \setminus \{ -\lambda_n \, : \, n\ge 0\}$. \\
(ii) The claimed trace formulas are  obtained by expanding $\mathcal H_\lambda$ at $\lambda = \infty$. 
For $\e:= \frac{1}{\lambda} >0$ with $\lambda > | - \lambda_0 + 1|$, define 
\begin{equation}\label{definition tilde mathcal H}
\tilde{ \mathcal H}_\e := \frac{1}{\e}\mathcal H_{\frac 1\e} = \sum_{n=0}^\infty \frac{|\langle 1 | f_n \rangle|^2}{1+\e \lambda_n}\ ,
\end{equation}
so that in view of \eqref{L 2 norm} - \eqref{average},
$\tilde {\mathcal H}_0=1$,  $\frac{d}{d\e } |_{ \e = 0} \tilde {\mathcal H}_\e = \langle  u | 1 \rangle$, and 
\begin{equation}\label{second derivative}
 \frac{d^2}{d\e ^2} |_{ \e = 0}\tilde {\mathcal H}_\e = 2\| \Pi u\|^2 = \| u\|^2 + \langle  u | 1 \rangle^2
 \end{equation}
 and hence, using these identities, 
 $$
  \frac{d}{d\e} |_{ \e = 0} \log \tilde {\mathcal H}_\e =  \frac{d}{d\e} |_{ \e = 0}  \tilde {\mathcal H}_\e 
  = \langle  u | 1 \rangle
  $$
  and
 \begin{equation}\label{| u |^2}
  \frac{d^2}{d\e ^2} |_{ \e = 0} \log \tilde {\mathcal H}_\e = - \big( \frac{d}{d\e } |_{ \e = 0} \tilde {\mathcal H}_\e\big)^2  +  \frac{d^2}{d\e ^2} |_{ \e = 0}\tilde {\mathcal H}_\e
  = \| u \|^2 \ .
 \end{equation}
Furthermore, identity \eqref{trace2} becomes
\begin{equation}\label{trace3}
-\frac{d}{d\e }\log \tilde {\mathcal H}_\e =
\frac{\lambda_0}{1+\e \lambda_0}+\sum_{n=1}^\infty \frac{\gamma_n}{(1+\e (\lambda_{n-1}+1))(1+\e \lambda_n)}\ .
\end{equation}
Passing to the limit as $\e \to 0$ in both sides, we obtain that the series $\sum_{n = 1}^\infty \gamma_n$ is summable --- notice that $\gamma_n\ge 0$ and $\lambda_n\ge 0$ for $n$ large enough ---
and that
\begin{equation}\label{mean}
- \langle  u | 1 \rangle=\lambda_0+\sum_{n=1}^\infty \gamma_n\ .
\end{equation}
Taking the derivatives of both sides of \eqref{trace3} with respect to $\e $, one concludes that $\frac{d^2}{d\e ^2}\log \tilde {\mathcal H}_\e$ equals
$$
\frac{\lambda_0^2}{(1+\e \lambda_0)^2}+
\sum_{n=1}^\infty \frac{\gamma_n}{(1+\e (\lambda_{n-1}+1))(1+\e \lambda_n)}\big(\frac{\lambda_{n-1}+1}{1+\e (\lambda_{n-1}+1)}+\frac{ \lambda_n}{1+\e \lambda_n} \big)\ .
$$
Passing to the limit as $\e \to 0$ and using \eqref{| u |^2}, one concludes that
$$\|u\|^2  =\lambda_0^2+\sum_{n=1}^\infty \gamma_n(\lambda_n+\lambda_{n-1}+1)\ .$$
From identity \eqref{mean} and the definition of $\gamma_n$ we have for every $n\ge 0$,
\begin{equation}\label{expandlambda}
\lambda_n=n- \langle  u | 1 \rangle-\sum_{k=n+1}^\infty \gamma_k\ .
\end{equation}
As a consequence,
$$\lambda_n+\lambda_{n-1}+1=2n-2 \langle  u | 1 \rangle-\gamma_n-2\sum_{k=n+1}^\infty \gamma_k\ ,$$
so that 
$$\|u\|^2  = \lambda_0^2 + 2\sum_{n=1}^\infty n\gamma_n-2 \langle  u | 1 \rangle\sum_{n=1}^\infty \gamma_n
- \big( \sum_{n=1}^\infty \gamma_n \big)^2\ ,$$
which, in view of \eqref{mean}, leads to
$$\|u\|^2  = \langle  u | 1 \rangle^2+2\sum_{n=1}^\infty n\gamma_n\ ,$$
whence the claimed formula for $\|u\|^2$. 
\end{proof}
\begin{remark}
(i) It follows from the proof of Proposition \ref{formulae} that
$\mathcal H_\lambda (u)$ equals the trace of the rank one operator $ (L_u+\lambda Id)^{-1}-S^*(L_u+\lambda Id)^{-1}S$.
(ii) Writing $u \in L^2_{r}$ as $v + c$ with $c = \langle u | 1 \rangle$ and hence $v \in L^2_{r,0}$, one computes
$$
\mathcal H(u) = \mathcal H(v) - c \frac{1}{2\pi} \int_0^{2 \pi} v^2 d x - \frac{1}{3}c^3\ .
$$
By the computations in the proof of Proposition \ref{formulae} one has (cf. also \eqref{formula BO Hamiltonian})
$$
\mathcal H(v) =-\frac 16\frac{d^3}{d\e ^3} |_{ \e = 0}  \tilde{\mathcal H}_\e (v)\ , \quad 
\| v\|^2 =  \frac{d^2}{d\e ^2} |_{ \e = 0}\tilde {\mathcal H}_\e (v)\ , 
$$
and $c = \langle  u | 1 \rangle = \frac{d}{d\e} |_{ \e = 0}  \tilde {\mathcal H}_\e(u)$.
\end{remark}

The identity of Proposition \ref{formulae}(i) can be used to obtain product representations for $|\langle 1 | f_n \rangle |^2$,
which will be used later to appropriately scale $\langle 1 | f_n \rangle$ in order to obtain our candidates for (complex) Birkhoff coordinates.
\begin{corollary}\label{1,f} For any $u \in L^2_r$ and $n \ge 1,$
\begin{eqnarray*}
|\langle 1 | f_0 \rangle|^2 &=& \kappa_0\ , \qquad
\kappa_0 := \prod_{p=1}^\infty \Big( 1-\frac{\gamma_p}{\lambda_p-\lambda_0} \Big)\ , \\
|\langle 1 | f_n \rangle|^2 &=& \gamma_n \kappa_n \ , \quad
\kappa_n := \frac{1}{\lambda_n-\lambda_0}
\prod_{1\leq p\ne n} \Big( 1-\frac{\gamma_p}{\lambda_p-\lambda_n}\Big)\ .
\end{eqnarray*}
In particular,  $|\langle 1, f_n \rangle|^2$, $n \ge 0,$ are spectral invariants.
\end{corollary}
\begin{remark}
The formulas of Corollary \ref{1,f} provide a new proof of the fact that $\langle 1 | f_0 \rangle\ne 0$ and that, for $n\ge 1$, $\langle 1 | f_n \rangle=0$ if and only if $\gamma_n=0$.
\end{remark}
\begin{proof}
The claimed product representations are obtained by computing with the help of the identity of Proposition \ref{formulae}(i) 
the residue of $\mathcal H_\lambda (u)$ at $\lambda=-\lambda_n$ for any $n \ge 0$.
\end{proof}

As an application of Proposition \ref{formulae} and Corollary \ref{1,f} we analyze the isospectral set $\text{Iso}(u)$ of an arbitrary potential $u \in L^2_r$,
defined as 
$$
\text{Iso}(u) := \{ v \in L^2_r : \lambda_n(v)= \lambda_n(u) \,\, \forall n \ge 0\}\,. 
$$
Note that $\text{Iso}(u)$ is closed by Proposition \ref{lipschitz} and that  
$$
\text{Iso}(u) = \{ v \in L^2_r : \mathcal H_\lambda (v)= \mathcal H_\lambda(u) \,\, \forall \lambda \in \mathbb C \setminus\{ \lambda_n(u)\,: \, n \ge 0 \}\}\ .
$$
First we need to make some preliminary considerations. Recall that by Proposition \ref{formulae}, 
$(\gamma_n(u))_{n \ge 1}$ is in the weighted $\ell^1-$sequence space $\ell^{1,1}(\mathbb N, \mathbb R)$
and $\gamma_n(u) \ge 0$ for any $n \ge 1$.
Hence the map $\Gamma$
\begin{equation}\label{Gamma}
\Gamma : L^2_{r,0} \to \mathcal C_+^{1,1}\,, u \mapsto (\gamma_n(u))_{n \ge 1}\,
\end{equation}
is well defined. Here  $\mathcal C_+^{1,1}$ denotes the positive cone in $\ell^{1,1}(\mathbb N, \mathbb R) $,
$$
\mathcal C_+^{1,1} := \{ (r_n)_{n \ge 1} 
\in \ell^{1,1}(\mathbb N, \mathbb R) \, : \,  r_n \ge  0 \,\, \, \forall n \ge 1 \}.
$$
\begin{proposition}\label{compact}
The map $\Gamma $ is proper.
\end{proposition}
\begin{proof} Let $(u^{(k)})_{k \ge 1}$ be a sequence in $L^2_{r,0}$ so that 
$$\Gamma (u^{(k)})\to \gamma =(\gamma_n)_{n\ge 1}\in \mathcal C_+^{1,1}\ \,\, \text{ as } \,  k \to \infty.$$
By Proposition \ref{formulae}(ii) and the assumption, 
$$\| u^{(k)}\|^2= 2\sum_{n=1}^\infty n\gamma_n(u^{(k)}) \,\,\, \to \,\,\, 2 \sum_{n=1}^\infty n\gamma_n\,\,\,\, \text{ as } k \to \infty\, ,
$$
and thus we may assume, after extracting a subsequence if needed, that $u^{(k)}\rightharpoonup u$ weakly in 
$L^2_{r,0}$. 
The proof will be complete if we establish that this convergence is strong, or equivalently that
$$\| u\|^2= 2\sum_{n=1}^\infty n\gamma_n\ .$$
Since (again by  Proposition \ref{formulae}(ii))
$$L_{u^{(k)}}\ge \lambda_0(u^{(k)})=-\sum_{n=1}^\infty \gamma_n(u^{(k)})\to -\sum_{n=1}^\infty \gamma_n\ ,$$
we infer that there exists $c >   |- \lambda_0 +1|$ so that for any $k \ge 1$ and $\lambda \ge c$,
$$L_{u^{(k)}}+\lambda Id : H^1_+\longrightarrow L^2_+$$ is an isomorphism whose inverse is bounded uniformly in $k$. Therefore 
$$w_\lambda ^{(k)}:=(L_{u^{(k)}}+\lambda Id )^{-1}[1]\ ,$$
 is a well defined, bounded sequence of $H^1_+$. Let us choose an arbitrary countable subset $\Lambda $ of $[c,+\infty )$ with a cluster point. 
By a diagonal procedure, we extract a subsequence of $w_\lambda ^{(k)}$, again denoted by $w_\lambda ^{(k)}$, so that for every $\lambda \in \Lambda $, 
$w_\lambda ^{(k)}$ converges weakly in $H^1_+$ to some element $w_\lambda \in H^1_+$ as $k \to \infty$. By Rellich's theorem we infer that, weakly in $L^2_+$,
$$
(L_{u^{(k)}}+\lambda Id )w_\lambda^{(k)} \rightharpoonup (L_u+\lambda Id )w_\lambda\,\, \text{ as } k \to \infty\,.
$$ 
But since $(L_{u^{(k)}}+\lambda Id )w_\lambda^{(k)}  = 1$ for any $k \ge 1$ it then follows that for every $\lambda \in \Lambda $, $(L_u+\lambda Id )w_\lambda = 1$ and hence
by the definition of $\mathcal H_\lambda (u)$,
$$\mathcal H_\lambda (u^{(k)})= \langle w_\lambda^{(k)} | 1\rangle \,\, \to \,\, \langle w_\lambda | 1\rangle =\mathcal H_\lambda (u) \,\,  \quad \forall \, \lambda \in \Lambda.$$
On the other hand, for every $n\ge 0$,
$$\lambda_n(u^{(k)})=n-\sum_{j=n+1}^\infty \gamma_j(u^{(k)})\,\, \to \,\, n-\sum_{j=n+1}^\infty \gamma_j\  \text{ as }\,\,\, k \to \infty\ $$
uniformly with respect to $n \ge 0$.  Hence, setting $\lambda_n :=  n-\sum_{j=n+1}^\infty \gamma_j$ for any $n \ge 0$ one infers that 
 for any $\lambda \in [ c,  \infty )$,
$$\frac{1}{\lambda_0(u^{(k)})+\lambda}\prod_{n=1}^\infty \Big(1-\frac{\gamma_n(u^{(k)})}{\lambda_n(u^{(k)})+\lambda}   \Big) \,\, 
\to \,\, \frac{1}{\lambda_0+\lambda}\prod_{n=1}^\infty \Big(1-\frac{\gamma_n}{\lambda_n+\lambda}   \Big).$$
By Proposition \ref{formulae}(i), it then follows that  
$$\mathcal H_\lambda (u)=\frac{1}{\lambda_0+\lambda}\prod_{n=1}^\infty \left(1-\frac{\gamma_n}{\lambda_n+\lambda}   \right)$$
for any $\lambda \in \Lambda $ and hence by analyticity, for any  $\lambda \in [c,+\infty )$.  Thus \eqref{second derivative}
applies and we conclude that $\| u\|^2 = 2 \sum_{n=1}^\infty n\gamma_n.$
\end{proof}
Our result on $\text{Iso}(u)$ then reads as follows.
\begin{proposition}\label{isocompact}
For every $u\in L^2_r$, $\text{Iso}(u) $ is a compact subset of $L^2_r$.
In fact, $\text{Iso}(u) - \langle u | 1 \rangle$ is a compact subset of the sphere in $L^2_{r,0}$ of radius $\|u - \langle u | 1 \rangle\|$, centered at $0$.
In particular, for any $c \in \R$,  $\text{Iso}(c)$ consists of the constant potential $c$ only.
\end{proposition}
\begin{proof}
Let  $v\in \text{Iso}(u)$. Since by Proposition \ref{formulae}(ii)
$$\langle v | 1 \rangle =-\lambda_0(u)-\sum_{n=1}^\infty \gamma_n(u)= \langle  u | 1 \rangle$$
and $\gamma_n(u - \langle u | 1 \rangle) = \gamma_n(u)$ by the definition of $\gamma_n$, $ n \ge 1$, one has
$$v- \langle  u | 1 \rangle\in \Gamma^{-1}\{ (\gamma_n(u))_{n\ge 1}\}. $$
Using that $\Gamma$ is proper by Proposition \ref{compact} and that $\|u - \langle u | 1 \rangle \|$ is a spectral invariant,  it follows that
$\text{Iso}(u - \langle u | 1 \rangle) = \text{Iso}(u) - \langle u | 1 \rangle$ is contained in a compact subset of the sphere in $L^2_{r,0}$ of radius $\|u - \langle u | 1 \rangle\|$, centered at $0$ . 
Since $\text{Iso}(u)$ is closed one then concludes that $\text{Iso}(u) - \langle u | 1 \rangle$ and in turn $\text{Iso}(u)$ are compact.
\end{proof}


\section{Complex Birkhoff coordinates}\label{Birkhoff coordinates}

In this section we introduce our candidates of complex Birkhoff coordinates, define the corresponding Birkhoff map $\Phi$, and discuss first properties of $\Phi$.

For any $u \in L^2_{r,0}$ and $n\ge 0$, define
\begin{equation}\label{birkhoff}
\zeta_n(u) :=\frac{\langle 1 | f_n \rangle}{\sqrt {\kappa_n}} \end{equation}
where we recall that 
$\kappa_0 = \prod_{p\ge 1} \Big( 1-\frac{\gamma_p}{\lambda_p-\lambda_0}\Big)$ and,  for any $n \ge 1,$
$$  
\kappa_n =\frac{1}{\lambda_n-\lambda_0}\prod_{1\leq p\ne n} \Big( 1-\frac{\gamma_p}{\lambda_p-\lambda_n}\Big)\ .
$$ 
The functionals $\zeta_n$, $n \ge 1$, are our candidates
for complex Birkhoff coordinates of the BO equation.
By Corollary \ref{1,f}, $\zeta_0 =1$ and for any $n \ge 1$, 
$|\zeta_n|^2=\gamma_n$.
\begin{proposition}\label{Birkhoff map}
 The map
\begin{equation}\label{Phi}
\Phi : L^2_{r,0} \to h^{1/2}_+\,, \, u \mapsto 
(\zeta_n(u))_{n \ge 1}
\end{equation}
is continuous and proper.
\end{proposition}
\begin{proof}
Since by Proposition \ref{formulae}, $\gamma_n = \frac{1}{n} \ell_n^1$  and in addition $|\zeta_n| = |\gamma_n|^{1/2}$, it follows that
$\Phi(u) \in h^{1/2}_+$ for any $u\in L^2_{r,0}$. Furthermore, since for any $n \ge 1,$ we know from Proposition \ref{lipschitz}, Remark  \ref{f_n continuous} and Proposition \ref{formulae} (ii), that $\langle 1,f_n\rangle $ and $\kappa_n$ depend continuously on $u$, we infer that $\zeta_n$ depends continuously on $u$. Since by Proposition \ref{formulae} (ii),  
$\Phi $ is a bounded map, it then follows
that whenever $u_{k}\to u$ in $L^2_{r,0}$,
 $\Phi (u_{k})$ converges to $\Phi (u)$ weakly in the Hilbert space $h^{1/2}_+$. Since, again by Proposition \ref{formulae}, 
$$\| \Phi (v)\|_{1/2}^2 =2\| v\|^2 \ ,\quad v\in L^2_{r,0}\ ,$$
we infer that $\|\Phi (u_{k}) \|_{1/2} \to \| \Phi (u)\|_{1/2} $ and hence that $\Phi (u_k)\to \Phi (u)$ in $h^{1/2}_+$. This shows that $\Phi $ is continuous. Finally, by Proposition \ref{compact}, the map $\Gamma$ is proper and so is $\Phi$.
\end{proof}

In a next step we want to show that the map $\Phi$ is one-to-one. We prove this fact by showing that any potential $u \in L^2_{r,0}$
can be expressed in terms of the complex numbers $\zeta_n(u),$ $n \ge 1$. First note that 
the function $\Pi u$ extends as a holomorphic function to the unit disc $|z| < 1$, which
by a slight abuse of notation, we denote by $\Pi u(z)$.
It is given by its Taylor expansion at $z=0,$
$$
\Pi u(z)=\sum_{k=0}^\infty \widehat u(k)z^k=\sum_{k=0}^\infty\langle \Pi u | S^k1 \rangle z^k =\sum_{k=0}^\infty \langle (S^*)^k\Pi u | 1\rangle z^k\ .
$$
Denote by $\| S^*\| $ the norm of the operator $S^*:L^2_+\rightarrow L^2_+$. Since $|z| \| S^*\|=|z| <1$, one has $\sum_{k=0}^\infty (z S^*)^k = (Id-zS^*)^{-1} $  and hence 
\begin{equation}\label{Piu=}
\Pi u(z)= \langle (Id-zS^*)^{-1}\Pi u | 1\rangle \ ,\qquad  \forall \,  |z|<1.
\end{equation}
Denote by $M \equiv M(u)$ the matrix representation of the operator $S^*: L^2_+ \to L^2_+$ in the basis $(f_n)_{n \ge 0}$,
\begin{equation}\label{matrix representation of S^*}
M = (M_{np})_{n, p \ge 0}, \qquad  M_{np} \equiv M_{np}(u) = \langle S^*f_p |  f_n \rangle = \langle f_p | Sf_n\rangle \ .
\end{equation}
Then \eqref{Piu=} together with the formula \eqref{u in eigenbasis} for $\Pi u$ yields the following
\begin{lemma}\label{inverse formula}
For any $u \in L^2_r$ and any $z\in \C$ with $|z| < 1,$
\begin{equation}\label{recover}
\Pi u(z)= \langle (Id - zM(u))^{-1}X(u) | Y(u) \rangle _{\ell ^2}\ ,
\end{equation}
with $X(u)$ and $Y(u)$ being the 
column vectors 
\begin{equation}\label{definition X Y}
X(u) := - (\lambda_p \langle 1 | f_p \rangle )_{p\ge 0}\ ,\ Y(u) := (\langle 1 |  f_n \rangle)_{n\ge 0}\ .
\end{equation}
\end{lemma}
Recall that by \eqref{SS*}, $MM^*=Id$ and hence $Id - zM$ is invertible for any $|z| < 1$. 
\begin{proposition}\label{injectivity}
The map  $\Phi$ is one-to-one.
\end{proposition}
\begin{proof}

First recall that that for any $u \in L^2_{r,0}$,
$\langle 1| f_0 \rangle=\sqrt{\kappa_0}$
with $\kappa_0= \prod_{ p\ge 1} \Big( 1-\frac{\gamma_p}{\lambda_p-\lambda_0}\Big)$
and for any $n \ge 1$
$$
\langle 1| f_n \rangle=\sqrt{\kappa_n}\zeta_n\ ,\qquad  
\kappa_n=\frac{1}{\lambda_n-\lambda_0}\prod_{1\leq p\ne n} \Big( 1-\frac{\gamma_p}{\lambda_p-\lambda_n}\Big)
$$
where  $ \gamma_p=|\zeta_p|^2$ and 
$
\lambda_n=n-\sum_{k=n+1}^\infty \gamma_k\ .
$
Therefore the components $\langle 1 | f_n \rangle$ of $Y(u)$ and the components $- \lambda_p \langle 1 | f_p \rangle$ of $X(u)$ can be expressed in terms of
the $\zeta_k$, $k \ge 1$.
We claim that the coefficients $M_{np}$ can also be expressed in this way. Indeed, for any given $n \ge 0$ one argues as follows:
if $\zeta_{n+1}=0$, then $\gamma_{n+1} = 0$ and hence $f_{n+1} = Sf_n$, implying that 
$$M_{np}=\delta_{p, n+1} .$$
If $\zeta_{n+1}\ne 0$, then $\gamma_{n+1} \ne 0$ and $\langle 1 | f_{n+1} \rangle \ne 0$. We then use 
that by Lemma \ref{formula with langle f_p | 1 rangle} for any $p \ge 0$,
$$
(\lambda_p - \lambda_n -1) M_{np} = - \langle u | Sf_n \rangle \langle f_p | 1 \rangle
$$
and hence in particular for $p = n+1$, $\gamma_{n+1} \langle f_{n+1} | Sf_{n} \rangle = - \langle u | Sf_n \rangle \langle f_{n+1} | 1 \rangle$.
Combining these two identities then yields
\begin{equation}\label{Mnp}
M_{np}=  \frac{ \gamma_{n+1} \langle f_{n+1} |  Sf_n \rangle}{\langle f_{n+1} | 1\rangle}   \frac{\langle f_p | 1\rangle}{ \lambda_p-\lambda_n-1} \, .
\end{equation}
Since $ \langle f_{n+1} |  Sf_n\rangle > 0$ the identity $1 =  \|S f_n \|^2 =  \sum_{p = 0}^\infty |M_{np}|^2$ then reads
\begin{equation}\label{eqfSf}
1 =  \langle f_{n+1} |  Sf_n\rangle ^2\frac{\gamma_{n+1}^2}{|\langle f_{n+1} |1 \rangle |^2}\sum_{p=0}^\infty \frac{|\langle f_p | 1\rangle |^2}{(\lambda_p-\lambda_n-1)^2}
\end{equation}
Altogether we thus have shown that all terms in the formula \eqref{Mnp} for $M_{np}$ can be expressed in terms 
of $\zeta_k$, $k \ge 1$, and hence also $\Pi u(z)$ for $|z|<1$ can be expressed in this way. Since $\Pi u(z)$, $|z|<1$, completely determines $u$ we have shown 
that $\Phi$ is one-to-one.
\end{proof}
\begin{remark}
By the definition \eqref{expandHlambda} of $\mathcal H_\lambda$ and formula \eqref{eqfSf},
$$
 1 =\langle f_{n+1} | Sf_n \rangle ^2\frac{\gamma_{n+1}^2}{| \langle f_{n+1} | 1\rangle |^2} \big( -\mathcal H'_{-\lambda_n-1} \big) \ .
$$
In view of Proposition \ref{formulae}(i) this leads to
$\langle f_{n+1} | Sf_n \rangle =\sqrt{\mu_{n+1}}$
with 
\begin{equation}\label{fSf}
\mu_{n+1}:=\Big(1-\frac{\gamma_{n+1}}{\lambda_{n+1}-\lambda_0}\Big)
\prod_{1\leq p\ne n + 1}\frac{\displaystyle{\Big( 1-\frac{\gamma_p}{\lambda_p-\lambda_{n+1}}\Big)}}{\displaystyle{\Big(1-\frac{\gamma_p}{\lambda_p-\lambda_{n}-1}\Big)}}\ .
\end{equation}
\end{remark}
We finish this section with discussing two symmetry properties of the map 
$\Phi$. For any $u \in L^2_r,$ denote by $u_*$
the element in $L^2_r,$ given by $u_*(x):= u(-x).$
Note that the fixed points of the involution $u \mapsto u_*$ are the even functions in $L^2_r$.
\begin{proposition}
For any $u \in L^2_{r}$, 
$$
\lambda_n(u_*) = \lambda_n(u)\, , \quad
f_n(x, u_*) = \overline{f_n(-x, u)}\, , 
\qquad \forall \, n \ge 0\, .
$$
As a consequence, for any $u \in L^2_{r,0}$,
$$
\zeta_n(u_*) = \overline{\zeta_n(u)} \qquad 
\forall \, n \ge 1\ .
$$
Hence $u$ is even if and only if 
for any $n \ge 1$, $\zeta_n(u) \in \R$.
\end{proposition}
\begin{proof}
Given any $u \in L^2_{r, 0}$ and $n \ge 0$, let 
$g_n(x):= \overline{f_n(-x, u)}$ and, to shorten notation, 
write $\lambda_n$ for $\lambda_n(u)$ and $f_n(x)$ instead of
$f_n(x, u)$. Taking the complex conjugate of both sides of the identity 
$L_u f_n = \lambda_n f_n$ and then evaluate them at $-x$ one obtains $L_{u_*} g_n = \lambda_n g_n$.
Hence $\lambda_n = \lambda_n(u_*)$ and $g_n$ is an eigenfunction of $L_{u_*}$, corresponding to the eigenvalue 
$\lambda_n(u_*)$. One verifies inductively that the eigenfunctions $g_n$, $n \ge 0$, satisfy the normalisation
conditions of Definition \ref{on basis}, implying that
$g_n(x, u) = f_n(x, u_*)$. \\
Since $\kappa_n(u)$, $n \ge 0,$ are spectral invariants of $L_u$
one has $\kappa_n(u_*)= \kappa_n(u)$ and in turn for any $n \ge 1,$
$$
\zeta_n(u_*) = \frac{1}{\sqrt{\kappa_n(u)}} 
\langle 1 | f_n(\cdot, u_*) \rangle =
\frac{1}{\sqrt{\kappa_n(u)}} \langle 1 | g_n \rangle =
\overline{\zeta_n(u)} \, .
$$
It then follows that for $u \in L_{r,0}$ even, 
$\zeta_n(u) \in \R$ for any $n \ge 1$. Conversely,
if for a given $u \in L_{r,0}$,
$\zeta_n(u) \in \R$ for any $n \ge 1$, then
$\Phi(u) = \Phi(u_*)$. Since $\Phi$ is one-to-one
we conclude that $u = u_*$.
\end{proof}

The second result concerns potentials in $L^2_{r,0}$ which are
$2\pi / K-$ periodic for some integer $K \ge 2.$ Let 
$L^{2,K}_{r,0}$ denote the subspace of $L^2_{r,0},$ consisting of such elements and by $L^{2,K}$, $L^{2,K}_{+}$ the corresponding subspaces of 
$L^{2}$, $L^{2}_{+}$. 
For any $u \in L^{2,K}_{r,0}$, let $L^{(K)}_u$ be the Lax operator
$-i\partial_x - T_u^{(K)}$, acting on $L^{2,K}_+$ where $T_u^{(K)}$
denotes the Toeplitz operator given by $T_u^{(K)}(f) = \Pi^{(K)}(uf)$
and $\Pi^{(K)}$ is the Szeg\H{o} projector $L^{2,K} \to L^{2,K}_+$.
The spectrum of $L^{(K)}_u$ is given by a sequence of eigenvalues
which we list in increasing order, $(\lambda^{(K)}_n(u))_{n \ge 0}$.
Following the arguments of the proof of Proposition \ref{simple} 
with the operator $S$ replaced by 
$$
S^{(K)}: L^{2,K}_+ \to L^{2,K}_+, f \mapsto e^{iKx}f
$$
one verifies that $\lambda^{(K)}_n(u) \ge \lambda^{(K)}_{n-1}(u) + K$
for any $n \ge 1.$ Since all the eigenvalues of $L_u$
are simple and any eigenvalue $\lambda^{(K)}_n$ of $L_u^{(K)}$
is also an eigenvalue of $L_u$ it follows by a simple homotopy argument
applied to the path $t \mapsto tu$, $0 \le t \le 1$, that 
$\lambda^{(K)}_n(u) = \lambda_{nK}(u)$ for any $n \ge 0.$
\begin{proposition}\label{prop:lambda^(K)}
Assume that $u \in L^2_{r,0}$ is $2\pi / K-$periodic for some integer 
$K \ge 2$. Then for any $(n, k)$ with $n \ge 0$ and $1 \le k \le K-1,$
$$
\lambda_{nK+k}(u) = \lambda_{nK}(u) + k\, , \qquad
f_{nK+k} (x, u) = e^{ikx} f_{nK}(x, u).
$$
Hence for such pairs $(n, k)$,
$\gamma_{nK+k}(u) = 0$ and thus $\zeta_{nK+k}(u) = 0.$
\end{proposition}
\begin{proof}
Let $u \in L^2_{r,0}$ be
$2\pi / K-$ periodic for some integer $K \ge 2.$
Since for any $n \ge 0,$ $\lambda^{(K)}_n(u) = \lambda_{nK}(u)$,
and in view of Lemma \ref{f_n for gamma_n = 0}
it suffices to prove that for any $(n , k)$ with $n \ge 0$, $1 \le k \le K-1$
one has
\begin{equation}\label{lambda^(K)}
\lambda_{nK+k}(u) = \lambda_{nK}(u) + k\ .
\end{equation}
To verify the latter identity, denote by $f_n^{(K)}$ the eigenfunction of $L_u^{(K)}$
corresponding to the eigenvalue $\lambda_n^{(K)}$.

Since $u$ and $f_n^{(K)}$ are both $2\pi / K$ periodic one has
$\Pi(u f_n^{(K)}) = \Pi^{(K)}(u f_n^{(K)})$ and hence
for any $1 \le k \le K-1$, 
$e^{ikx} \Pi(u f_n^{(K)}) = \Pi(e^{ikx}u f_n^{(K)})$ implying that 
$$
 -i\partial_x(e^{ikx} f_n^{(K)}) - 
\Pi^{(K)}(e^{ikx}u f_n^{(K)}) = (\lambda^{(K)}_n + k) e^{ikx} f_n^{(K)}
$$
By the homotopy argument mentioned above,
\eqref{lambda^(K)} then follows.
\end{proof}
\begin{remark}
Conversely, any element $u \in L^2_{r,0}$ with $\zeta_{nK + k}(u) = 0$
for any $(n, k)$ with $n \ge 0$ and $1 \le k \le K-1,$ is
$2 \pi/K-$periodic. Indeed, denote by 
$\Phi^{(K)}: L^{2,K}_{r,0} \to h^{1/2}_+, 
u \mapsto (\zeta_n^{(K)}(u))_{n \ge 1}$
the Birkhoff map on $L^{2,K}_{r,0}$, where $(\zeta_n^{(K)}(u))_{n \ge 1}$
are the Birkhoff coordinates of $u$, viewed as an element in $L^{2,K}_{r,0}$.
Arguing as in the proof of Theorem \ref{main result} concerning  the case $K=1$ (cf Section \ref{proof of Theorem 1}) one infers that $\Phi^{(K)}$ is bijective.
Hence there exists $v \in L^{2,K}_{r,0}$so that 
$\Phi^{(K)}(v) = (\zeta_{nK}(u))_{n \ge 1}$.
By Proposition \ref{prop:lambda^(K)} it then follows that 
$\Phi(v) = (\zeta_n(u))_{n \ge 1}$
and hence by the uniqueness of $\Phi$ one concludes that $v=u$.
\end{remark}

\section{Analyticity and Gradients}\label{gradients}
In this section, we establish that the Birkhoff coordinates $\zeta_n$ are real analytic functions,
and that their gradients $\nabla \zeta_n$ are real analytic maps with values in $H^1$. Hence the Gardner brackets $\{ \zeta_n,F\}$ are well defined real analytic functions for every real analytic functional $F$ on $L^2_{r,0}$. 

 As a first result we establish that the eigenvalues $ \lambda_n$ are real analytic.
Note that for any $v \in L^2 \ (\equiv L^2(\T, \C)),$ 
the Lax operator $L_v = -i \partial_x - T_v$ on $L^2_+$ with
domain $H^1_+$ is no longer selfadjoint, but it is still a closed
unbounded operator with compact resolvent. Hence its spectrum
consists of a sequence of complex eigenvalues, each with finite
multiplicity. For any $r > 0$, $\lambda \in \C$
and for any $\e > 0$, $u \in L^2_{r}$,
 let
$$
D_r(\lambda) = \{ z \in \C \ : \ |z - \lambda | < r \}\ ,
\quad
B_\e (u) = \{ v \in L^2 \ : \ \| v - u \| < \e  \}
$$
and for $N \ge 1$ and $n \ge 0$, denote by $ {\rm Box}_{N,n}(u)$
the closed rectangle in $\C$ given by the set of complex numbers $\lambda$ satisfying
$$
-N +\lambda_0(u)\le {\rm Re}(\lambda ) \le \lambda_n(u) + 1/2 \ ,  \qquad |{\rm Im}(\lambda )| \le N \ .
$$
For $u\in L^2_r$, we define
$r_0(u) := 1/4\ ,\ \tau_0(u) := \lambda_0(u)$, and, for $p\ge 1$,
$$\ r_p(u) := \frac 14 + \frac{\gamma_p(u)}2\ ,\ \tau_p(u) := 
\frac{\lambda_p(u) + \lambda_{p-1}(u) + 1}2\ .$$
\begin{lemma}\label{lambda analytic} For any $u \in L^2_{r}$, $N \ge 1,$ and $n \ge 0,$ there exists
$\e_n > 0$ so that for any $v \in B_{\e_n}(u)$ and $0 \le k \le n,$
$\# \big( spec(L_v) \cap D_{r_k}(\tau_k(u)) \big) = 1$
and
$$ spec(L_v) \cap  {\rm Box}_{N,n}(u)
\subset \bigcup_{0 \le p \le n} D_{r_p(u)}(\tau_p(u))\ .
$$
The unique eigenvalue of $L_v$ in $D_{r_k(u)}(\tau_k(u))$,
denoted by $\lambda_k(v)$, is real analytic on
$B_{\e_n}(u)$. It then follows that for any $1 \le k \le n,$ 
$\gamma_k$ is an analytic functional on $B_{\e_n}(u)$ as well.
\end{lemma}
\begin{proof} For simplicity, we do not indicate the  dependence of $r_k$ and $\tau_k$ on $u$.
In a first step we prove that there exists $\delta_n > 0$
so that for any $v \in B_{\delta_n}(u)$
$$
 spec(L_v)\cap {\rm Box}_{N,n}(u)
\subset  \bigcup_{0 \le k \le n} D_{r_k}(\tau_k) \ .
$$
Assume to the contrary that such a $\delta_n > 0$ does not exist.
Then there exists a sequence $(v_\ell)_{\ell \ge 1} \subset L^2$
with $\| v_\ell - u \| < 1/\ell$ and an eigenvalue 
$\mu_\ell \in spec(L_{v_\ell})$ in 
${\rm Box}_{N,n}(u)$ so that
$| \mu_\ell - \tau_k| \ge r_k$ for any $0 \le k \le n$.
Let $g_{\ell}$ be an eigenfunction
of $L_{v_\ell}$ in $H^1_+$ corresponding to
$\mu_\ell,$ $L_{v_\ell}g_{\ell} = \mu_\ell g_\ell$, with $\|g_\ell\| = 1$.
Since $(\mu_\ell)_{\ell \ge 1}$ is bounded, $(g_\ell)_{\ell \ge 1}$
is a bounded sequence in $H^1_+$. Hence by choosing subsequences
if needed, we can assume without loss of generality that 
$(\mu_\ell)_{\ell \ge 1}$ converges to a complex number $\mu$ 
and $(g_\ell)_{\ell \ge 1}$ converges weakly in $H^1_+$ to an
element $g \in H^1_+$. Then 
$\mu \in {\rm Box}_{N,n}(u) \setminus
\bigcup_{0 \le k \le n} D_{r_k}(\tau_k)$,
and $\lim_{\ell \to \infty}\mu_\ell g_\ell = \mu g$ as well as
$\lim_{\ell \to \infty}L_{v_\ell} g_\ell = L_u g$
weakly in $H^1_+$. It then follows that $L_u g = \mu g$ and hence
that $\mu$ is an eigenvalue of $L_u$, contradicting the fact that
$ spec{L_u} \cap  {\rm Box}_{N,n}(u) \subset 
\bigcup_{0 \le k \le n} D_{r_k}(\tau_k).$
 Hence there exists $\e_n > 0$ so that for any 
 $v \in B_{\e_n}(u)$
$$ spec{L_v} \cap {\rm Box}_{N,n}(u) \subset
\bigcup_{0 \le k \le n} D_{r_k}(\tau_k) \ .
 $$  
 For any $0 \le k \le n,$ denote by $C_k$ the circle of radius $1/3 + \gamma_k(u)/2$
 with counterclockwise orientation, 
 centered at $\tau_k$.
 Then for any $\lambda \in C_k$ and $v \in B_{\e_n}(u)$,
  $L_v - \lambda$ is invertible and 
  $$
  C_k \times B_{\e_n}(u) \to \mathcal L(L^2_+,H^1_+), \, 
  (\lambda, v) \mapsto (\lambda - L_v)^{-1}
  $$
  is analytic and so is the Riesz projector
  $$
  P_k: B_{\e_n}(u) \to \mathcal L(L^2_+,H^1_+), \
  v \mapsto \frac{1}{2\pi i} \int_{C_k} 
  (\lambda - L_v)^{-1} d\lambda \ .
  $$
  Since for any given $v \in B_{\e_n}(u)$, $L_v$ has a compact resolvent, $P_k(v)$ is an operator of finite rank. Hence its trace,
  $Tr P_k(v)$, is finite. Actually, $Tr P_k(v) = 1$
  since $Tr P_k(v)$ is  constant and $Tr P_k(u)=1$.
  Hence for any $v \in B_{\e_n}(u)$ and $0 \le k \le n,$
  $L_v$ has precisely one eigenvalue in $D_{r_k}(\tau_k)$
  and this eigenvalue is simple. We denote it by $\lambda_k(v)$.
  By functional calculus one has
  $\lambda_k(v) = 
  Tr\frac{1}{2\pi i} \int_{C_k} \lambda ( \lambda-L_v)^{-1} 
  d\lambda$ and hence it follows that 
  $\lambda_k : B_{\e_n}(u) \to \C, v \mapsto \lambda_k(v)$
  is analytic.
\end{proof}
As a second step, we prove that the eigenfunctions $f_n$ defined by Definition \ref{on basis} are real analytic maps  with values in $H^1_+$. To state 
this result in more detail, we first need to make some
preliminary considerations.
Denote by $L^2_- \equiv L^2_-(\T, \C)$ the Hardy space
$$
L^2_- := \{ h \in L^2 \ : \ 
h(x) = \sum_{k = - \infty}^0 \widehat h(k) e^{ikx} \}
$$
and the corresponding Szeg\H{o} projector 
$\Pi^- : L^2 \to L^2_-$. For any $u \in L^2_r,$
denote by $L^-_u$ the operator 
$$
L^-_u= i \partial_x - T^-_u : H^1_- \to L^2_-\ ,
\qquad H^1_-:= H^1\cap L^2_-
$$
where $T^-_u$ denotes the Toeplitz operator
$$
T^-_u : H^1_- \to L^2_-\ , \ h \mapsto \Pi^-(uh) \ .
$$
Using that $u$ is real valued, one verifies that the spectrum of $L^-_u$ coincides with the one of $L_u$ and that for any
$n \ge 0$, $L^-_u f_n^- = \lambda_n f_n^-$ where
$f^-_n = \overline{f_n}$. Hence $(f^-_n)_{n \ge 0}$
is an orthonormal basis of $L^2_-$, 
normalized in such a way that
$\langle 1 | f_0^- \rangle >0$ and 
$\langle f_{n+1}^- | S^-f_n^- \rangle >0$ for any 
\textcolor{red}{$n \ge 0$}.
Here $S^- : L^2_- \to L^2_-, \ h \mapsto e^{-ix}h$
denotes the shift operator to the left. Furthermore,
the Riesz projector $P^-_n(u)$ onto $span(f_n^-)$ is
given by
$$
P^-_n(u) = \frac{1}{2\pi i} 
\int_{C_n} (\lambda - L^-_u)^{-1} d \lambda 
$$
where $C_n$ is the same circle appearing in the definition
of the Riesz projector $P_n(u),$ defined in the 
proof of Lemma \ref{lambda analytic}.
Finally we introduce for any $n \ge 0$ the functions
\begin{equation}\label{definition f_n,1}
f_{n, 1}(\cdot \ , u) := 
\big( f_n(\cdot \ , u) + f_n^-(\cdot \ , u) \big) /2
\end{equation}
and
\begin{equation}\label{definition f_n,2}
f_{n, 2}(\cdot \ , u) := 
\big( f_n(\cdot \ , u) - f_n^-(\cdot \ , u) \big) /2i
\end{equation}
implying that 
$f_n(\cdot \ , u) = 
f_{n,1}(\cdot \ , u) + i f_{n,2}(\cdot \ , u)$.
Since for $u$ real valued, $f_{n,1}(x, u)$ is the real part of $f_{n}(x, u)$ and $f_{n,2}(x, u)$ its imaginary part
one has
$$
\langle 1 | f_n( \cdot\ , u) \rangle =
\langle 1 , f_{n,1}( \cdot\ , u) 
-i f_{n,2}( \cdot\ , u) \rangle =
\langle 1 , f_n^-(\cdot \ , u) \rangle 
$$
and in turn
$$
P_n(u) 1 =  \langle 1 , f_n^-( \cdot\ , u) \rangle  
f_n (\cdot \ , u) \ .
$$
\begin{lemma}\label{eigenfunction analytic}
For any $u \in L^2_{r}$ and $n \ge 0$ there exists
$ 0 < \delta_n \le \e_n$, with $\e_n$ as in 
Lemma \ref{lambda analytic}, so that for any 
$0 \le k \le n$, $f_k$ and $f_k^-$ admit analytic extensions
$$
f_k : B_{\delta_n}(u) \to H^1_+\ , \qquad 
f_k^- : B_{\delta_n}(u) \to H^1_- \  .
$$
As a consequence, $f_{k,1}$ and $f_{k,2}$, defined
by \eqref{definition f_n,1} - \eqref{definition f_n,2},  
admit analytic extensions, 
$f_{k,j} : B_{\delta_n}(u) \to H^1$ and 
so does $\langle 1 , f_k^- \rangle : B_{\delta_n}(u) \to \C$
and $P_k 1$, meaning that 
 $$
 P_k(\cdot) 1 : B_{\delta_n}(u) \to H^1_+, 
 v \mapsto P_k(v) 1 =
 \langle 1 , f_k^-(\cdot \ , v)\rangle f_k(\cdot \ , v)
 $$
 is analytic.
 Furthermore, $f_k$ satisfies the normalisation condition
 $$ \langle f_k , f_k^- \rangle = 
 \langle f_{k,1} , f_{k,1} \rangle^2 + 
 \langle f_{k,2} , f_{k,2} \rangle^2 = 1 \ .
 $$
\end{lemma}
\begin{proof}
We argue inductively and begin with $f_0$. By the normalisation of $f_0,$ for any 
$v \in B_{\e_n}(u) \cap L^2_r,$ 
$\langle 1 | f_0(\cdot \, , v) \rangle > 0.$
Since for such a $v$ the operator $L_v$ is selfadjoint, 
$P_0(v) = \frac{1}{2\pi i} 
\int_{C_0} (\lambda - L_v)^{-1} d\lambda$
is the 
$L^2-$orthogonal projector onto 
$span(f_0(\cdot \ , v))$ and hence
$P_0(v) 1 = \langle 1 | f_0(\cdot \, , v) \rangle 
f_0(\cdot \, , v)$.
It implies that
$\langle P_0(v) 1 | 1 \rangle =
 \langle 1 | f_0(\cdot \, , v) \rangle^2 > 0$
and thus
$$
\langle 1 | f_0(\cdot \, , v) \rangle = 
\sqrt{\langle P_0(v) 1 | 1 \rangle}
= \sqrt{\langle P_0(v) 1 , 1 \rangle}\ .
$$ 
Since $P_0: B_{\e_n}(u) \to \mathcal L(L^2_+, H^1_+)$
is analytic
we can shrink $\e_n$ if needed so that 
$Re \langle P_0(v) 1 , 1 \rangle > 0$ is uniformly bounded
away from $0$ on $B_{\e_n}(u)$. As a consequence, the principal branch of the root 
$\sqrt{\langle P_0(v) 1 , 1 \rangle}$ is analytic on 
$B_{\e_n}(u)$. The analytic extension of 
$f_0$ is then defined by
$$
f_0(\cdot \ , v) := \frac{1}{\sqrt{\langle P_0(v) 1 , 1 \rangle}} P_0(v) 1 \in H^1_+ \ .
$$
By functional calculus,
$$
L_v P_0(v) 1 = \frac{1}{2\pi i} 
\int_{C_0} \lambda (\lambda - L_v)^{-1} 1 d\lambda
= \lambda_0(v) P_0(v) 1 
$$
and hence $L_v f_0(\cdot \ ,v) 
= \lambda_0(v) f_0(\cdot \ ,v)$.
Using the Riesz projector $P^-_0(v)$ instead of $P_0(v)$
one sees by the same arguments that the analytic 
extension of $f_0^-$ is defined for $v \in B_{\e_n}(u)$ by 
$$
f_0^-(\cdot \ , v) := 
\frac{1}{\sqrt{\langle P_0^-(v) 1 , 1 \rangle}} 
P_0^-(v) 1 \in H^1_- \ ,
$$
and that $L_v^-(v) f_0^-(\cdot \ , v) =
 \lambda_0(v) f_0^-(\cdot \ , v)$.
\\
Now assume that after shrinking $\e_n$ if necessary,  
$f_\ell$ and $f_\ell^-$, $0 \le \ell \le k$, 
have been extended analytically to $B_{\e_n}(u)$ for a given $k \le n-1$. More precisely, for any $0 \le \ell \le k$
$$
f_{\ell} : B_{\e_n}(u) \to H^1_+, \qquad
f_{\ell}^- : B_{\e_n}(u) \to H^1_-
$$
are analytic maps and for any $v \in B_{\e_n}(u)$,
$$
L_v f_\ell(\cdot \ ,v) 
= \lambda_\ell(v) f_\ell(\cdot \ ,v) \ , \qquad
L_v^- f_\ell^-(\cdot \ ,v) 
= \lambda_\ell(v) f_\ell^-(\cdot \ ,v) \ .
$$
 In case
$v \in B_{\e_n}\cap L^2_r,$ one has 
$\langle f_{k+1}(\cdot \ , v) | Sf_{k} (\cdot \ , v)  \rangle >0$. Since in such a case $L_v$ is
selfadjoint, the Riesz projector 
$P_{k+1}(v) = \frac{1}{2\pi i} 
\int_{C_{k+1}} (\lambda - L_v)^{-1} d\lambda$
is the 
$L^2-$orthogonal projector onto 
$span(f_{k+1}(\cdot \ , v))$ and hence
$$
P_{k+1}(v) Sf_{k} (\cdot \ , v) = 
\langle Sf_{k} (\cdot \ , v) \ | \
f_{k+1}(\cdot \ , v) \rangle \, 
f_{k+1}(\cdot \, , v) \,.
$$
It implies that 
$$
\langle Sf_{k} (\cdot \ , v) \ 
| f_{k+1}(\cdot \ , v) \rangle = 
\sqrt{\langle P_{k+1}(v) Sf_{k} (\cdot \ , v)
\ | \  Sf_{k} (\cdot \ , v) \rangle} \ .
$$
Since $P_{k+1}: B_{\e_n}(u) \to \mathcal L(L^2_+, H^1_+)$
is analytic
we can once more shrink $\e_n$ if needed so that 
$ {\rm Re} \ \alpha_k > 0$ is uniformly bounded
away from $0$ on $B_{\e_n}(u)$
where $\alpha_k \equiv \alpha_k(v)$
is defined by
$$
\alpha_k(v)
: = \langle P_{k+1}(v) S f_k(\cdot \ , v)
\ , \  S^- f_k^-(\cdot \ , v) \rangle \ .
$$ 
As a consequence, the principal branch of the root 
$\sqrt{\alpha_k}$ is analytic on 
$B_{\e_n}(u)$. The analytic extension of 
$f_{k+1}$ to $B_{\e_n}(u)$ 
is then given by
$$
f_{k+1}(\cdot \ , v) = \frac{1}
{\sqrt{\alpha_k(v)}}
P_{k+1}(v) Sf_{k} (\cdot \ , v) \in H^1_+ 
$$
and by functional calculus one has
$L_v f_{k+1}(\cdot \ ,v) 
= \lambda_{k+1}(v) f_{k+1}(\cdot \ ,v)$.
Using the Riesz projector $P^-_{k+1}(v)$ instead of 
$P_{k+1}(v)$
one sees by the arguments used above that the analytic 
extension of $f_{k+1}^-$ is given for 
$v \in B_{\e_n}(u)$ by 
$$
f_{k+1}^-(\cdot \ , v) = 
\frac{1}
{\sqrt{\alpha_{k}^-(v)}}
P_{k+1}^-(v) S^-f_{k}^- (\cdot \ , v) \in H^1_- 
$$
and that $L_v^-(v) f_{k+1}^-(\cdot \ , v) =
 \lambda_{k+1}(v) f_{k+1}^-(\cdot \ , v)$.
 Here $\alpha_{k}^-(v)$ is defined by
 $$
 \alpha_k^- (v)
: = \langle P_{k+1}^-(v) S^- f_k^-(\cdot \ , v)
\ , \  S f_k(\cdot \ , v) \rangle \ .
 $$ 
We then denote by $0 < \delta_n \le \e_n$ the number obtained
after shrinking $\e_n$ successively so that for any $0 \le k \le n$ , $f_{k}$ extends to an analytic map,
$f_{k} : B_{\delta_n}(u) \to H^1_+$ and similarly, $f_{k}^-$ extends to one on $B_{\delta_n}(u)$ with values in $H^1_-$. 
\end{proof}

We now present a formula for the $L^2-$gradient of the eigenvalues $\lambda_n$, $n \ge 0$.
By Lemma \ref{lambda analytic}, for any $n \ge 0,$
 $\lambda_n: u\in L^2_r \mapsto \lambda_n(u)\in   \R$ is real analytic and hence its $L^2-$ gradient $\nabla \lambda_n$ is well defined.
 More precisely, we have the following
\begin{corollary}\label{formula nabla lambda_n} 
For any $u \in L^2_r$, $n \ge 0$, and $0 \le k \le n$,
the $L^2-$gradient $\nabla \lambda_k(u)$ of $\lambda_k$ 
at $u$ is given by 
$\nabla \lambda_k(u) = -  |f_k|^2(\cdot \ , u)$
and extends to an analytic function on 
$B_{\delta_n}(u)$ with $\delta_n > 0$ given as in 
Lemma \ref{eigenfunction analytic}. More precisely,
$$
\nabla \lambda_k: B_{\delta_n}(u) \to  H^1,\  v \mapsto 
  - f_{k, 1}(\cdot \ , v)^2 - f_{k, 2}(\cdot \ , v)^2
$$ 
is  analytic where $f_{k,1}$ and $f_{k,2}$ are given
as in Lemma \ref{eigenfunction analytic}.
\end{corollary}
\begin{proof} Let $u \in L^2_r$ and $0 \le k \le n$ be given.
 To compute the gradient $\nabla \lambda_k (u)$, consider
for any $v \in L^2_r$ the one parameter family $u_\e:= u + \e v$. 
It is convenient to introduce the notation $f_k := f_k (\cdot, u),$ $\lambda_k := \lambda_k(u)$, and
$$
 \delta f_k := 
 \frac{d}{d\e}|_{ \e = 0}  f_k(\cdot, u_\e), \qquad 
\delta \lambda_k := \frac{d}{d\e}|_{ \e = 0}  \lambda_k(u_\e).
$$
Taking the derivative with respect to $\e$ of both sides of the identity
$(L_{u_\e} - \lambda_k(u_\e)) f_k(\cdot, u_\e) = 0$ at $\e = 0$ one obtains 
$$
(L_u - \lambda_k) \delta f_k - T_v f_k - \delta{\lambda_k} f_k = 0.
$$
Taking the inner product with $f_k$ of both sides of the latter identity and using that $L_v - \lambda_k$ is selfadjoint one then concludes
$$
- \langle T_v f_k | f_k \rangle  - \delta{\lambda_k}\langle f_k |  f_k \rangle  = 0.
$$ 
Since $\| f_k \| = 1$ and $\Pi f_k = f_k$, the definition of the $L^2-$gradient then implies that
$$
\langle \nabla \lambda_k, v \rangle  = \delta{\lambda_k} = - \langle v f_k |  f_k \rangle  = \langle - |f_k|^2 | v \rangle  \,,
$$
yielding the claimed formula for $\nabla \lambda_k(u)$. 
Since by Lemma \ref{eigenfunction analytic}, the real
and imaginary parts $f_{k, 1}$, $f_{k, 2}$ of $f_k$,  extend to analytic functions on $B_{\delta_n}(u)$ 
with values in $H^1$ and by Lemma \ref{lambda analytic},
and $\lambda_k$ also extends to an analytic functional on $B_{\delta_n}(u)$, it follows that the gradient 
$\nabla \lambda_k$ and its formula extend analytically,
$$
\nabla \lambda_k: B_{\delta_n}(u) \to  H^1,\  v \mapsto 
  - f_{k, 1}(\cdot \ , v)^2 - f_{k, 2}(\cdot \ , v)^2
$$
proving the lemma.
\end{proof}

As a consequence of  Corollary \ref{formula nabla lambda_n},
we obtain  a formula for the gradient of $\gamma_n = \lambda_n - \lambda_{n-1} -1,$ 
\begin{equation}\label{gradient gamma_n}
\nabla \gamma_n = - | f_n |^2 + |f_{n-1}|^2.
\end{equation}

As a third step, we show that for any $n \ge 0$, the $L^2-$gradient 
$\nabla \langle 1 | f_n \rangle$ is a real analytic
map $L^2_r \to H^1, \ u \mapsto 
\nabla \langle 1 , f_n(\cdot \ , u) \rangle$.
Recall that 
for any $u \in L^2_{r}$ and $n \ge 1$, the eigenfunctions
$f_k$, $0 \le k \le n$, admit an analytic extension 
$f_k = f_{k,1} + i f_{k,2}$ to $B_{\delta_n}(u)$ 
, i.e. 
$$
f_{k, j}  : B_{\delta_n}(u) \to H^1, 
\ v \mapsto f_{k, j}(\cdot \ , v) \ , \quad j = 1,2,
$$ 
are  analytic maps (cf. Lemma \ref{eigenfunction analytic}). For any $v \in B_{\delta_n}(u)$, we denote
by $df_{k, j}(\cdot \ ,v)[w]$ the derivative of $f_{k, j}$ at $v$
in direction $w \in L^2$.
\begin{lemma}\label{1 | f_n in H^1}
For any $n \ge 0$ and $v \in B_{\delta_n}(u)$ there exists
a constant $c_n > 0$ so that for any $0 \le k \le n$, 
$j = 1,2$, and $w \in L^2$,
$$
 \| df_{k, j}(\cdot \ ,v)[w]\| \leq c_n \| w\|_{H^{-1}} \ . 
 $$
The constant $c_n$ can be chosen locally uniformly with respect to $v$.
As a consequence, $\nabla \langle 1,f_{k,j} \rangle : B_{\delta_n}\to H^1 $, $j = 1,2$, are  analytic.
\end{lemma}
\begin{proof}
The latter assertion easily follows from the first one, since for any $0 \le k \le n$, $j = 1,2$, $v\in B_{\delta_n}(u)$, and $w\in L^2$,
$$
\begin{aligned}
| \langle \nabla \langle 1,f_{k,j}(\cdot \ , v)\rangle , 
w\rangle | & 
= | \langle 1, df_{k,j}(\cdot \ ,v) [w] \rangle | \\
& \leq \| df_{k,j}(\cdot \ ,v) [w]\|
\leq c_n \| w\|_{H^{-1}}\ ,
\end{aligned}
$$
implying that $g_{k,j}(v):=\nabla \langle 1, f_{k,j}(\cdot \ , v)\rangle$   is bounded in $H^1$ 
uniformly with respect to  $v\in B_{\delta_n}(u)$.
Since
$g_{k, j}  : B_{\delta_n}(u) \to L^2$
is an analytic map (cf. Lemma \ref{eigenfunction analytic}), so is 
$\langle g_{k, j} | e^{inx} \rangle :  B_{\delta_n}(u) \to \C$ 
for any $n \in \Z$. Together with the boundedness of 
$g_{k, j}$ in $H^1$ it then follows that 
$g_{k,j}: B_{\delta_n}(u) \to H^1$ is analytic
(cf. e.g. \cite[Appendix A]{GK1}).
\\
To prove the first assertion of the lemma, we 
proceed by induction. First we need
to make some preliminary considerations.
Without further reference, we will use
the notation introduced in the derivation of 
Lemma \ref{lambda analytic} and Lemma \ref{eigenfunction analytic}.
According to Lemma \ref{eigenfunction analytic},
$f_{k,1} =( f_k + f_k^{-})/2$ and 
$f_{k,2} = (f_k -f_k^{-})/2i$ where $f^-_k$, $k \ge 0,$
are the eigenfunctions of 
$L^-_v : H^1_- \to L^2_-, \  
h \mapsto i \partial_x h - \Pi^-(vh)$ and $\Pi^-$
is the Szeg\H{o} projector $L^2 \to L^2_-$. 
The first assertion
of the lemma then follows by proving corresponding estimates
for $df_k(\cdot \ , v)[w]$ and 
$df_k^-(\cdot \ , v)[w]$, $0 \le k \le n.$
By the definition of $r_n,$
$$
(D_{r_n + \frac{1}{2}} \setminus D_{r_n + \frac{1}{4}})
\times B_{\delta_n} \to \mathcal L(L^2_+, H^1_+), 
\ (\lambda, v) \mapsto (\lambda - L_v)^{-1} 
$$
is  analytic. Hence the adjoint of
$(\lambda - L_v)^{-1}$ with respect to the
dual pairing of $H^{-1}_+$ and $H^1_+$, also
gives rise to an analytic map,
$$
(D_{r_n + \frac{1}{2}} \setminus D_{r_n + \frac{1}{4}} )
\times B_{\delta_n} \to \mathcal L( H^{-1}_+, L^2_+), 
\ (\lambda, v) \mapsto (\lambda - L_v)^{-1} \, .
$$
Since the circles $C_k $, $0 \le k \le n$, are compact, it implies that for any 
$v \in B_{\delta_n}$ there exists $M_n > 0$ so that
for any $\lambda \in C_k$, $0 \le k \le n$,
\begin{equation}\label{resolvent estimate}
\| (\lambda -L_v)^{-1}\|_{L^2_+\to H^1_+}\leq M_n\ ,\qquad 
\| (\lambda -L_v)^{-1}\|_{H^{-1}_+\to L^2_+}\leq M_n\ ,
\end{equation}
where $M_n$ can be chosen locally uniformly with respect to
$v$. Furthermore, for any $0 \le k \le n$,
the Riesz projector 
$$
P_k(v)= \frac{1}{2\pi i} 
\int_{C_k}(\lambda - L_v)^{-1} d\lambda
$$
gives rise to an analytic map 
$P_k : B_{\delta_n} \to \mathcal L(L^2_+, H^1_+)$.
For any $h \in L^2_+$, the derivative
of $P_k(v) h$ with respect to $v$ in direction $w \in L^2$ is given by 
$$
(dP_k(v)[w])h = - \frac{1}{2\pi i}
\int_{C_k}(\lambda - L_v)^{-1}\Pi ( w (\lambda -L_v)^{-1}h) 
d\lambda\ .
$$
We want to estimate the operator norm  of $dP_k(v)[w]$,
viewed as an operator on $L^2_+$. Note that for any
$g, h \in L^2_+$, 
$$
\langle (dP_k(v)[w])h , \ g \rangle =
- \frac{1}{2\pi i} \int_{C_k} 
\langle w (\lambda -L_v)^{-1}h , \
(\lambda - L_v)^{-1} g \rangle d \lambda \, .
$$
Taking into account \eqref{resolvent estimate} and that multiplication by $f \in H^1$
yields a bounded operator on $H^{-1}$,
we infer that for any $v \in B_{\delta_n}$, there exists a constant $d_n > 0$
so that for any $0 \le k \le n$
\begin{equation}\label{estP}
\| dP_k(v)[w]\|_{L^2_+\to L^2_+} \leq d_n \| w \| _{H^{-1}}
\end{equation}
where $d_n$ can be chosen locally uniformly with respect 
to $v$.\\
We are now ready to present the induction argument
with respect to $0 \le k \le n$ for proving the claimed
estimates for $df_k(\cdot \ , v)[w]$
and $df_k^-(\cdot \ , v)[w]$. We start with $k=0.$
Recall that the formula for the analytic extension of 
$f_0$ to $B_{\delta_n}$ of 
Lemma \eqref{eigenfunction analytic} reads
$$
f_0(\cdot \ ,v)= \frac{1}{\sqrt{\langle P_0(v)1,1\rangle }} P_0(v)1 ,
$$
implying that for any $w \in L^2$
$$
df_0(\cdot \ ,v)[ w ]=
\frac{1}{\sqrt{\langle P_0(v)1,1 \rangle }} (dP_0(v)[w])1 -
\frac 12\frac{\langle (dP_0(v)[w])1,1\rangle}
{\langle P_0(v)1,1\rangle ^{\frac 32}}P_0(v)1\ .$$
Applying \eqref{estP} for $k=0$, we immediately conclude
that there exists $A_n >0$ so that
$$
\| df_0(v)w\| _{L^2}\leq A_n\| w\|_{H^{-1}} \ .
$$
The corresponding estimate for  $df_0^-(v)[w]$ 
is proved in an analogous way by using
the estimate for $dP^-_0(v)[w]$ corresponding 
to the one of $dP_0(v)[w]$.\\
Now assume that the claimed estimates for
$df_\ell(\cdot \ , v)[w]$ and
$df_\ell^-(\cdot \ , v)[w]$ are valid for
any  $0 \le \ell \le k$ for a given $0 \le k \le n-1$. 
We claim that they also hold for $k+1$. 
Let us first consider the one for $df_{k+1}(\cdot \ , v)[w]$
Again we use the formula for the analytic extension of $f_{k+1}$ to $B_{\delta_n}$ derived in Lemma \ref{eigenfunction analytic},
$$
f_{k+1}(\cdot \ ,v) = \frac{1}{\sqrt{\alpha_{k+1}(v)}}
P_{k+1}(v)\big(Sf_k(\cdot \ ,v)\big)
$$
where
$$
\alpha_{k+1}(v) = \big\langle \ P_{k+1}(v)(Sf_k(\cdot \ ,v))\vert 
\ 
e^{-ix} f_{k}^-(\cdot \ , v) \ \big\rangle 
$$
The  directional derivative $f_{k+1}(\cdot \ ,v)[w]$ 
can then be computed as
$$
df_{k+1}(\cdot \ ,v)[w] = \frac{1}{\sqrt{\alpha_{k+1}(v)}}I 
- \frac12 \frac{d\alpha_{k+1}(v)[w]}{\alpha_{k+1}(v)^{3/2}}  
P_{k+1}(v)\big(Sf_k(\cdot \ ,v)\big)
$$
where
$$
I:= dP_{k+1}(v)[w] (Sf_k(\cdot ,v)) + 
P_{k+1}(v)( Sdf_k(\cdot ,v)[w])
$$
and $d\alpha_{k+1}(v)[w] = II_1 + II_2 + II_3$ with
$$
II_1 :=  \langle  dP_{k+1}(v)[w](Sf_k( \cdot \ ,v)) , \
e^{-ix} f_{k}^-(\cdot \ ,v) \rangle  
$$
$$
II_2 := \langle P_{k+1}(v)\big(S df_k(\cdot \ ,v)[w]\big) , \
e^{-ix} f_{k}^-(\cdot \ ,v) \rangle  
$$
$$
II_3:= \langle P_{k+1}(v)[Sf_k( \cdot \ ,v)], \
e^{-ix}df_k^-(\cdot \ ,v)[w]\rangle \ .
$$
By increasing $A_n$ if needed, 
the estimate \eqref{estP} for $k + 1$ together with the induction hypothesis then yield
$$
\| df_{k+1}(v)[w]\| _{L^2}\leq A_{n}\| w\|_{H^{-1}}\ ,
$$
The corresponding estimate for  $df_{k+1}^-(v)[w]$, 
is proved in an analogous way by using the
estimate for $dP^-_k(v)[w]$ corresponding 
to the one of $dP_k(v)[w]$.
Altogether we have proved the induction step.
Going through the arguments of the proof one verifies that
the constant $A_n$ can be chosen locally uniformly with respect to $v$.
\end{proof}
 \begin{remark}\label{ langle 1 | f_n rangle at 0}
At $u = 0,$ one has  $ f_k (x) = {\rm e}^{i k x}$ for any $k \ge 0$ and thus $\langle 1 | f_n \rangle = 0$ for any $n \ge 1$. Let us compute the gradient of $\langle 1\vert f_n\rangle $ at $u=0$ for $n\ge 1$. For any $v \in L^2_r$ and $\e \in \R,$ define  $\delta f_n := \frac{d}{d\e} _{| \e = 0} f_n( \cdot, \e v)$ and $\delta \lambda_n := \frac{d}{d\e} _{| \e = 0} \lambda_n( \cdot, \e v)$.
Arguing as in the proof of Corollary \ref{formula nabla lambda_n} one has for any $n \ge 0$
$$
(D-n) \delta f_n = T_v({\rm e}^{inx}) + \delta \lambda_n {\rm e}^{inx}\ .$$
If $n\ge 1$, taking the inner product of both sides with $1$, we obtain
$$-n \langle \delta f_n\vert 1\rangle =\langle {\rm e}^{inx}, v\rangle $$ which leads to
$$
 \nabla \langle 1 | f_n \rangle =   - \frac{1}{n} {\rm e}^{-i n x} \ . 
 $$
\end{remark}
As a fourth step, we  discuss the analytic extension  of
$\kappa_n$, $n\ge 0$, given by the formulae in  Corollary \ref{1,f}. A difficulty here is that these formulae involve infinitely many eigenvalues $\lambda_n$, while  the domain of analyticity of $\lambda_n$ might shrink as $n \to \infty$. Therefore we need an extra argument. We shall appeal to the generating functional $\mathcal H_\lambda $ .First we need to make some preliminary considerations. Let $u \in L^2_r$ be given. If needed, shrink the radius
$\delta_n >0$ of the ball 
$B_{\delta_n} \equiv B_{\delta_n}(u)$ so that for any
$v \in B_{\delta_n}$ and $1 \le k \le n$, 
$\lambda_{k-1}(v) + 1$ is in the disk $D_{r_n}(\tau_n)$.
For any $v \in B_{\delta_n}$ and $0 \le k \le n,$
$\mathcal H_\lambda(v) = 
\langle (L_v + \lambda Id)^{-1} 1, 1 \rangle$
is holomorphic for $- \lambda$ in 
$D_{r_k + 1/4}(\tau_k) \setminus \{ \lambda_k(v) \}$ and
possibly has a simple pole at $- \lambda = \lambda_k(v)$.
Let $P_k^\bot(v) = Id - P_k(v)$ where $P_k(v)$ denotes
the Riesz projector, introduced in the proof of 
Lemma \ref{lambda analytic}. Note that 
$$
\mathcal H^\bot_{\lambda, k}(v) := 
\langle (L_v + \lambda Id)^{-1} P_k^\bot(v) 1, 1 \rangle
$$ 
is holomorphic on $- D_{r_k + 1/4}(\tau_k)$. 
By Lemma \ref{eigenfunction analytic}, 
$$
P_k(v) 1 = \langle 1 \ , \ f_{k,1}(\cdot \ , v)
 - i f_{k, 2}(\cdot \ , v) \rangle f_k(\cdot \ , v)
$$
implying that
$\mathcal H_{\lambda, k}(v) := 
\langle (L_v + \lambda Id)^{-1} P_k(v) 1, 1 \rangle$
equals 
$$
\big( \langle f_{k,1}(\cdot \ , v) , 1 \rangle^2
+ \langle  f_{k,2}(\cdot \ , v) , 1 \rangle^2 \big) \frac{1}{\lambda_k(v) + \lambda}\ .
$$
Hence the residue $res_k(v)$ of $\mathcal H_{\lambda}(v) =
\mathcal H_{\lambda, k}(v) + 
\mathcal H^\bot_{\lambda,k}(v)$
at $\lambda = - \lambda_k(v)$ is given by 
\begin{equation}\label{residue}
res_k(v) = 
\langle f_{k,1}(\cdot \ , v) , 1 \rangle^2
+ \langle  f_{k,2}(\cdot \ , v) , 1 \rangle^2 \ .
\end{equation}
Furthermore, for any $v \in B_{\delta_n}$,
$(\lambda_k(v) + \lambda)\mathcal H_{\lambda}(v)$
is holomorphic on $- D_{r_k + 1/4}(\tau_k)$. Actually,
the map
$$
- D_{r_k + 1/4}(\tau_k) \times B_{\delta_n} \to \C, \, 
(\lambda, v) \mapsto (\lambda_k(v) + \lambda)
\mathcal H_{\lambda}(v)
$$
is analytic. Let us now discuss this map in more detail.
We begin with the case $k=0$.
For any $(\lambda, v)$ in 
$- D_{r_0 + 1/4}(\tau_0) \times B_{\delta_n}$ set 
$$
\chi_0(\lambda, v) := (\lambda_0(v) + \lambda)
\mathcal H_{\lambda}(v)\ , \quad
\kappa_0(v) := \chi_0(- \lambda_0(v), v) \ .
$$ 
Notice that Proposition \ref{formulae} (i) implies that, for $v \in 
B_{\delta_n}\cap L^2_r$, the above definition of $\kappa_0(v)$ is consistent with 
the one in Corollary \ref{1,f}. 
With this notation we then get
$$
res_0(v) = 
 \langle f_{0,1}(\cdot \ , v) , 1 \rangle^2
+ \langle  f_{0,2}(\cdot \ , v) , 1 \rangle^2 
= \kappa_0(v) \ .
$$
By Corollary \ref{1,f}, for any $v \in 
B_{\delta_n}\cap L^2_r$, one has $\kappa_0( v) > 0$.
Hence by shrinking $\delta_n$ if necessary, we conclude
that ${\rm Re}\,  \kappa_0( v) > 0$ is uniformly bounded away
from $0$ on $B_{\delta_n}$.\\
Now let us consider the case $1 \le k \le n$.
Recall that our choice of $\delta_n$ mentioned above assures that
$\lambda_{k-1}(v) + 1 \in  D_{r_k}(\tau_k)$
for any $v \in B_{\delta_n}$. Hence it 
follows from Lemma \ref{lambda analytic} that the map 
$$
B_{\delta_n} \to \C, \, 
v \mapsto (\lambda_k(v) + \lambda)
\mathcal H_{\lambda}(v)|_{\lambda = 
- \lambda_{k-1}(v) - 1}
$$
is also analytic. Since by Proposition \ref{formulae}, 
for any $v \in B_{\delta_n}\cap L^2_r,$
$$
(\lambda_k(v) + \lambda)
\mathcal H_{\lambda}(v)|_{ \lambda = 
- \lambda_{k-1}(v) - 1} = 0 \ , 
$$
$(\lambda_k(v) + \lambda)
\mathcal H_{\lambda}(v)|_{ \lambda = 
- \lambda_{k-1}(v) - 1}$
vanishes identically on $B_{\delta_n}$.
Thus the function $(\lambda_k(v) + \lambda) \mathcal H_{\lambda}(v)$
is of the form $- (\lambda + \lambda_{k-1}(v) + 1)\chi_k(\lambda, v)$ where
$$
\chi_k : - D_{r_k}(\tau_k) \times B_{\delta_n} \to \C
$$
is analytic. Hence we have proved that for any
$(\lambda , v) \in - D_{r_k}(\tau_k) \times B_{\delta_n}$,
\begin{equation}\label{identity for mathcal H_lambda}
\mathcal H_{\lambda}(v) = -
\frac{\lambda_{k-1}(v) + 1 + \lambda}{\lambda_k(v) + \lambda} 
\chi_k(\lambda, v)\ .
\end{equation}
Using that $\lambda_{k-1}(v) + 1 =
\lambda_{k}(v) - \gamma_k(v)$ one then concludes
from \eqref{residue} that the residue $res_k(v)$ of
$\mathcal H_\lambda(v)$ at $\lambda = - \lambda_k(v)$
is given by 
$$ 
 \langle f_{k,1}(\cdot \ , v) , 1 \rangle^2
+ \langle  f_{k,2}(\cdot \ , v) , 1 \rangle^2 = 
 \gamma_k(v) \kappa_k(v)\ , \quad
 \kappa_k(v) := \chi_k(- \lambda_k(v), v) \ .
$$
Again Proposition \ref{formulae} (i) implies that, for $v \in 
B_{\delta_n}\cap L^2_r$, the above definition of $\kappa_k(v)$ is consistent with 
the one in Corollary \ref{1,f}. Moreover, by Corollary \ref{1,f},  $\kappa_k(v) > 0$ for any $v \in  B_{\delta_n}\cap L^2_r$.
Hence by shrinking $\delta_n$ if necessary, we conclude
that ${\rm Re}\, \kappa_k(v) > 0$ is uniformly bounded away from $0$ on $B_{\delta_n}$.
For later reference we record our findings as follows:
\begin{lemma}\label{kappa holomorphic}
After shrinking $\delta_n > 0$ if needed, the following holds: for
any $u \in L^2_r$, $v \in B_{\delta_n}(u)$, 
$$
\langle f_{0,1}(\cdot \ , v) , 1 \rangle^2
+ \langle  f_{0,2}(\cdot \ , v) , 1 \rangle^2 =
 \kappa_0(v)
$$
and for $1 \le k \le n$,
$$
\langle f_{k,1}(\cdot \ , v) , 1 \rangle^2
+ \langle  f_{k,2}(\cdot \ , v) , 1 \rangle^2 =
\gamma_k(v) \kappa_k(v)\ .
$$
For any $0 \le k \le n$, the functional 
$\kappa_k : B_{\delta_n}(u) \to \C$
is analytic with positive real part and ${\rm Re}\, \kappa_k(v) > 0$ is uniformly bounded away from $0$.
\end{lemma}
Finally we analyze the $L^2-$gradient of
$\kappa_n$. 
\begin{lemma} \label{nabla kappa_n}
 For any $n \ge 1,$ 
$\nabla \kappa_n$ takes values in 
$H^1_r$ and 
$\nabla \kappa_n : L^2_{r} \to H^1_r$
is real analytic.
\end{lemma}
\begin{remark} (i) The map 
$\nabla \kappa_0 :L^2_{r} \to H^1_r$
is real analytic. Indeed, by Lemma \ref{kappa holomorphic},
$\kappa_0 (u) = 
\langle 1 , f_{0,1}(\cdot \ , u) \rangle ^2
+ \langle 1 , f_{0,2}(\cdot \ , u) \rangle ^2$.
The claim then follows from Lemma \ref{1 | f_n in H^1}.\\
(ii) 
For $u=0$ one infers from the product representation
of $\kappa_n$ and the facts that
$f_n(x , 0) = e^{inx}$, $n \ge 0$,
and both $\gamma_p(0)$
and $\nabla \gamma_p (0)$, $p \ge 1$, vanish, that 
$\nabla \kappa_n(0) = 0$ for any $n \ge 0$. 
\end{remark}
\begin{proof} 
Let $n \ge 1$ and $u \in L^2_r.$
For any $1 \le k \le n$, $\kappa_n $ 
is an analytic map $B_{\delta_n}(u) \to \C$
(cf. Lemma \ref{kappa holomorphic}),
implying that $\nabla \kappa_n:  B_{\delta_n}(u) \to L^2$ is 
such a map. To simplify notation, we write
$B_{\delta_n}$ for $B_{\delta_n}(u)$.
We now prove that 
$\nabla \kappa_n$  is in fact an analytic map 
$B_{\delta_n} \to H^1$. Let us start from the identity \eqref{identity for mathcal H_lambda},
\begin{equation}\label{basic}
(\lambda _k+\lambda )\mathcal H_\lambda (v)=-(\lambda+\lambda_{k-1}(v)+1)\chi_k(\lambda ,v)\ ,
\end{equation}
which holds for $(\lambda ,v)\in -D_{r_k+\frac 14}(\tau_k)\times B_{\delta_n}$. Here $\chi_k $ is a holomorphic function on the latter domain 
and is related to $\kappa_k$ by 
$\kappa_k(v)=\chi_k(-\lambda_k(v), v)$.
By the chain rule we then have
$$\nabla \kappa_k(v)=-\frac{\partial \chi_k}{\partial \lambda}(-\lambda_k(v), v)\nabla \lambda_k(v)+\nabla \chi_k(-\lambda_k(v),v)\ .$$
Since $\nabla \lambda_k : B_{\delta_n} \to H^1$ is  
an analytic map
(cf. Corollary \ref{formula nabla lambda_n}, 
Lemma \ref{eigenfunction analytic} ), 
it remains to prove that 
$v\mapsto \nabla \chi_k(-\lambda_k(v),v)$ is 
an analytic map $B_{\delta_n} \to H^1$. 
It suffices to prove that 
$$
\nabla \chi_k :  
- D_{r_k+\frac 14}(\tau_k)\times B_{\delta_n} \to H^1
$$
is analytic. Indeed, let us take the gradient with respect to $v$ of both sides of the identity \eqref{basic}. Writing 
$$G_k(\lambda ,v):=(\lambda _k(v)+\lambda )\mathcal H_\lambda (v)\ ,$$
we obtain the following identity,
\begin{equation}\label{basicprime}
\nabla G_k(\lambda ,v)=
-(\lambda+\lambda_{k-1}(v)+1)\nabla \chi_k(\lambda ,v)-\chi_k(\lambda ,v)\nabla \lambda_{k-1}(v)\ .
\end{equation}
We claim  that $\nabla G_k$ is $H^1$ valued and
$\nabla G_k : -D_{r_k+\frac14}(\tau_k)\times B_{\delta_n} \to  H^1$ is analytic.
Indeed, $G_k$ satisfies the identity (cf. derivation of Lemma \ref{kappa holomorphic}) 
$$
G_k(\lambda ,v)=\langle f_{k,1}(\cdot \ ,v),1\rangle ^2+
\langle f_{k,2}(\cdot \ ,v),1\rangle ^2 +
(\lambda_k+\lambda )\mathcal H_{\lambda, k} ^\perp (v)
$$
where
$$
\mathcal H_{\lambda,k} ^\perp (v)=
\langle (L_v+\lambda Id)^{-1}P_k^\perp (v)1,1\rangle \ ,
\qquad P_k^\perp (v) = Id -P_k(v)\ ,
$$
and $P_k(v)$ denotes the Riesz projector
onto $span(f_k(\cdot \ ,v))$. 
For any 
$(\lambda, v) \in 
-D_{r_k+\frac14}(\tau_k) \times B_{\delta_n}$, 
we rewrite
$$
(L_v+\lambda Id)^{-1}P_k^\perp (v)=
(L_v+\lambda Id)^{-1} - (L_v+\lambda Id)^{-1}P_k (v) \ ,
$$
using 
$$
(L_v+\lambda Id)^{-1} = 
\frac{1}{2\pi i} \int_{C_k}(L_v + \lambda Id)^{-1} 
\frac{1}{\mu + \lambda} d \mu \ ,
$$
together with
$$
(L_v+\lambda Id)^{-1}P_k (v) = 
\frac{1}{2\pi i} \int_{C_k}
(L_v+\lambda Id)^{-1} (\mu Id -L_v)^{-1} d\mu\ ,
$$
and the resolvent identity
$$
(\mu + \lambda) (L_v+\lambda Id)^{-1} (\mu Id -L_v)^{-1} =
(L_v+\lambda Id)^{-1} + (\mu Id -L_v)^{-1} \, .
$$
It yields
$$
(L_v+\lambda Id)^{-1}P_k^\perp (v)=
- \frac{1}{2\pi i} \int_{C_k}
\frac{1}{\mu + \lambda} (\mu Id -L_v)^{-1}d\mu \ .
$$
Consequently,
$$\nabla \mathcal H_{\lambda ,k}^\perp (v)=
\frac{1}{2\pi i} \int_{C_k}(\mu Id -L_v)^{-1}
[1] 
\overline {(\overline \mu Id -L_{\overline v})^{-1}
[1]} \frac{1}{\lambda +\mu } d\mu 
$$
implying that 
$\nabla \mathcal H_{\lambda ,k}^\perp (v) \in H^1$ and
$\nabla \mathcal H_{\cdot \, ,k}^\perp: 
-D_{r_k + \frac{1}{4}}(\tau_k)\times B_{\delta_n } 
\to H^1$  
 is analytic. Finally, since
  $ \nabla \langle f_{k,j}(\cdot \ ,v),1\rangle:
 B_{\delta_n } \to H^1$, $j = 1,2$, are analytic
 (cf. Lemma \ref{1 | f_n in H^1}), it follows 
$$
\nabla G_k : - D_{r_k + \frac{1}{4}}(\tau_k)\times B_{\delta_n}
\to H^1
$$
is analytic. 
Coming back to \eqref{basicprime}, we infer that the holomorphic map
$$
\nabla G_k +  \chi_k \nabla \lambda_{k-1} :
 - D_{r_k +  \frac{1}{4}}(\tau_k)\times B_{\delta_n} 
 \to H^1
$$ 
vanishes at
$\lambda =-\lambda_{k-1}(v)-1$ for any
$v \in  B_{\delta_n}$.
 Therefore there exists a holomorphic map
$$
F_k : -D_{r_k+\frac 14}(\tau_k)\times B_{\delta_n} 
\to H^1
$$
such that for any $(\lambda, v) \in 
-D_{r_k+\frac 14}(\tau_k)\times B_{\delta_n} $
$$
\nabla G_k(\lambda, v) +
\chi_k(\lambda ,v)\nabla \lambda_{k-1}(v) =
- (\lambda+\lambda_{k-1}(v) + 1)F_k(\lambda, v) \,.
$$
Comparing with \eqref{basicprime}, one deduces that
$$\nabla \chi_k(\lambda ,v)= F_k(\lambda ,v)$$
and thus has the claimed property.
\end{proof}

With these preparations made, we now can study functionals $\zeta_n$. Recall from the definition \eqref{birkhoff} that for any $n \ge 1$, 
$u \in L^2_{r,0}$,
$\zeta_n(u) =\frac{\langle 1 | f_n(\cdot \  , u) \rangle}{\sqrt {\kappa_n(u)}}$.
The functionals
$ \langle 1 | f_n \rangle : L^2_{r,0} \to \C$
(cf. Lemma \ref{eigenfunction analytic}),
and $ \kappa_n : L^2_{r,0} \to \R$ (Lemma \ref{kappa holomorphic}) are real analytic.
Hence we claim that $\zeta_n: L^2_{r,0} \to \C$ is real analytic. Indeed,  let $u \in L^2_{r,0}$, $n \ge 1,$ and $1 \le k \le n.$
By Lemma \ref{eigenfunction analytic},
$\langle 1 , f_{k,1} - i f_{k, 2} \rangle  $
is analytic on $B_{\delta_n}(u)\cap L^2_{r,0}$ and
by Lemma \ref{kappa holomorphic}, so is $\kappa_k$.
Since by the same lemma, ${\rm Re}\, \kappa_k > 0$ is uniformly
bounded away from $0$ on $B_{\delta_n}(u) \cap L^2_{r,0}$
 it then follows
that the principal branch of the root
$\sqrt{\kappa_k}$ is a well defined, analytic functional
which is also uniformly bounded away
from $0$ on this ball. We thus have proved that
$\zeta_k = \frac{\langle 1 | f_{k,1} \rangle}
{\sqrt {\kappa_k}} 
- i \frac{\langle 1 | f_{k,2} \rangle}
{\sqrt {\kappa_k}}$ extends analytically
to $B_{\delta_n}(u) \cap L^2_{r,0}$.
\\
Now we study the gradient of $\zeta_n$. 
By the chain rule we have on $L^2_{r,0}$ 
\begin{equation}\label{formula nabla zeta_n}
\nabla \zeta_n = - \frac{1}{2} \frac{\nabla \kappa_n}{\kappa_n^{3/2}} \langle 1 | f_n \rangle + \frac{1}{\sqrt{\kappa_n}} \nabla \langle 1 | f_n \rangle 
\, .
\end{equation}
Since
$\nabla \langle 1 | f_n \rangle : L^2_{r,0} \to H^1$
(cf. Lemma \ref{1 | f_n in H^1})
and 
$\nabla \kappa_n : L^2_{r,0} \to H^1$
(cf. Lemma \ref{nabla kappa_n})
are real analytic, the following result holds true. 
\begin{proposition}\label{zeta_n in H^1}
 For any $n \ge 1,$ the functional 
 $\zeta_n : L^2_{r,0} \to \C$ and the map
$\nabla \zeta_n : L^2_{r,0} \to H^1$
are real analytic.
\end{proposition}
\begin{remark}\label{nabla zeta_n at 0}
The $L^2-$gradient of $\zeta_n$ can be computed explicitly
at $u=0$. Indeed,
since $\kappa_n(0) = \frac{1}{n}$ and
$\nabla \langle 1 | f_n \rangle (0) 
= -\frac{1}{n} e^{-inx}$
(cf. Remark \ref{ langle 1 | f_n rangle at 0})
one has
$$
\nabla \zeta_n(0) = -\frac{1}{\sqrt{n}} e^{-inx} \, .
$$
\end{remark}


\section{Poisson brackets}
The main goal of this section is to compute the Poisson brackets $\{\gamma_p, \gamma_n  \}$ and $\{ \gamma_p, \langle 1 | f_n \rangle \}$.
As a first step we  show that the flow, corresponding to the generating function $\mathcal H_{\lambda}$, leaves the eigenvalues of $L_u$ invariant.
Recall from \eqref{Hlambda} that for $u \in L^2_r$, the generating function is given by
$$
\mathcal H_\lambda (u)=\langle (L_u + \lambda Id)^{-1}1 \, | 1 \rangle .
$$
Consider $\mathcal H_\lambda$ with 
$\lambda \in \mathbb R \setminus \{ - \lambda_n(u) \, : \, n \ge 0 \} $. For any  $v \in L^2_r$ one has
$$
d\mathcal H_\lambda (u) [v] = \langle (L_u+\lambda Id)^{-1}T_v (L_u+\lambda Id)^{-1}1 \ | 1 \rangle = \langle v | \, |w_\lambda |^2 \rangle ,
$$
where $w_\lambda \equiv w_\lambda ( \cdot, u)$ is given by
$$w_\lambda  := (L_u+\lambda Id)^{-1} 1 \in H^1_+\ .$$
Hence by the definition of the $L^2-$gradient
$$
\nabla \mathcal H_\lambda (u)= |w_\lambda (\cdot,  u) |^2
$$
and the Hamiltonian equation associated to $\mathcal H_\lambda $ is
\begin{equation}\label{evolHlambda}
\partial_t u=\partial_x |w_\lambda ( \cdot , u)|^2\ .
\end{equation}
Since the map $ u \in L^2_r\mapsto  w_\lambda( \cdot, u)\in H^1_+$ is Lipschitz on bounded subsets,  so is
 the Hamiltonian vector field $ u \in L^2_r \mapsto \partial_x | w_\lambda ( \cdot , u)|^2\in L^2_r$ 
and hence the corresponding initial value problem  admits a local in time solution for
any initial data $u \in L^2_r$ and $\lambda \in 
\C \setminus \{ \lambda_n(u) :  n \ge 0 \}$.
Furthermore, the Toeplitz operators $T_{w_\lambda}$ and $T_{\overline w_\lambda}$
are bounded operators on $L^2_+$. The following result says that equation
\eqref{evolHlambda} has a Lax pair representation.
\begin{proposition}\label{laxpairHlambda}  For any $\lambda \in \mathbb R \setminus \{ - \lambda_n(u) \, : \, n \ge 0 \} $,
$$
\frac{dL_u}{dt}=[B_u^\lambda , L_u]\ ,\qquad  B_u^\lambda :=iT_{w_\lambda}T_{\overline w_\lambda}
$$
along the evolution \eqref{evolHlambda}.
\end{proposition}
\begin{remark}
Recall that for $\e > 0$ we introduced $\tilde{ \mathcal H}_{\e} = \frac{1}{\e} \mathcal H_{\frac{1}{\e}}$ and that $\frac{d^k}{d\e^k}\big|_{\e = 0}\tilde{ \mathcal H}_{\e} $, $k \ge 0$,
are referred to as the Hamiltonians in the BO hierarchy. Introducing $\tilde{ B}_\e := \frac{1}{\e} B^{\frac{1}{\e}}_u$, Proposition \ref{laxpairHlambda}
then leads to a Lax pair formulation for the equations corresponding to the Hamiltonians in the BO hierarchy, 
$$
\frac{dL_u}{dt}=[\frac{d^k}{d\e^k}\big|_{ \e = 0} \tilde{B}_\e , \,  L_u]\ ,
$$
where now $u$ evolves according to the flow of $\frac{d^k}{d\e^k}\big|_{ \e = 0}\tilde{ \mathcal H}_{\e}$.
In particular, in the case $k=3$, for any $u \in L^2_{r,0}$,  $\mathcal H(u) = - \frac{1}{6}  \frac{d^3}{d\e^3}\big|_{ \e = 0}\tilde{ \mathcal H}_{\e}(u)$
and one can check that the operator $- \frac{1}{6} \frac{d^3}{d\e^3}\big|_{ \e = 0} \tilde{B}_\e$ coincides 
with $\tilde B_u -\frac 12 i\| u\|^2 Id$, where
 $\tilde B_u  =  i (T_{|\partial_x|u}-T_u^2)$ 
 was introduced in Remark \ref{Btilde}.
\end{remark}
Recall that by Corollary \ref{formula nabla lambda_n}, $\nabla \lambda_n = - | f_n |^2$ for any $n \ge 0$. Hence
$\nabla \lambda_n \in \textcolor{red}{H^1_r}$ and the Poisson bracket $\{ \lambda_p, \lambda_n \}$ is well-defined on $L^2_r$
for any $p \ge 0$. Clearly, the same is true for $\{ \gamma_p,\gamma_n \}$.
\begin{corollary}\label{bracketlambda} 
(i) For any $p, n \ge 0$ and $u \in L^2_r$,\,
$\{ \lambda_p, \lambda_n \}(u)=0$. \\
(ii) For any $p, n \ge 1$ and $u \in L^2_r$,\, 
$ \{ \gamma_p,\gamma_n \}(u)=0 \ $
\end{corollary}
\begin{remark}
By the arguments of the proof of Corollary \ref{bracketlambda},
$$
\{\mathcal H_\lambda, \mathcal H_\mu \} = 0\ , \qquad  \{\mathcal H_\lambda, \lambda_n \} = 0
$$
 for any $\lambda, \mu \in \mathbb C \setminus \{ - \lambda_k(u) \, : \, k \ge 0 \} $ and $n \ge 0$.
\end{remark}

\noindent
{\em Proof of Corollary \ref{bracketlambda}.}
(i) In view of Proposition \ref{laxpairHlambda}, 
for any $n \ge 0$ and  $\lambda$ in $\mathbb R \setminus \{ - \lambda_k(u) \, : \, k \ge 0 \} $,
$\lambda_n(u)$ is constant along the evolution \eqref{evolHlambda}.
Hence on $L^2_r$, $\{ \mathcal H_\lambda ,\lambda_n \}=0$. Using \eqref{expandHlambda}, we therefore have
$$\sum_{p=0}^\infty \left ( -\frac{|\langle 1| f_p \rangle|^2}{ (\lambda_p+\lambda)^2 }\{ \lambda_p , \lambda_n \}
+\frac{\{ |\langle 1| f_p \rangle|^2, \lambda_n\}}{\lambda_p+\lambda}\right )=0\ .$$
By analytic continuation, the latter identity continues to hold for any $\lambda$ in 
$ \mathbb C \setminus \{ - \lambda_k(u) \, : \, k \ge 0 \} $.
Since the left hand side of this identity is a meromorphic function in $\lambda$,
 the principal part of the expansion at any of its pole vanishes, implying that $\{ \lambda_p, \lambda_n \}(u)=0$
 for any $p \ge 0.$
\hfill $\square$

\medskip

Let us come to the proof of Proposition \ref{laxpairHlambda}.

\noindent
{\em Proof of Proposition \ref{laxpairHlambda}.}
Assume that $u \equiv u(t)$ is a local in time solution of \eqref{evolHlambda} with $\lambda \in \R 
\setminus \{ \lambda_n(u(0)) : n \ge 0 \} $.
In the sequel of the proof, we denote also by $D$
the operator $- i \partial_x$, when viewed as operator
acting on $H^1.$
Since by the definition of the operator $L_u$ one  has
$$\frac{dL_u}{dt}=-T_{\partial_tu}=-iT_{D|w_\lambda|^2}\ ,$$
 the claimed result is equivalent to the identity
\begin{equation}\label{identity}
[D-T_u,T_{w_\lambda}T_{\overline w_\lambda}] - T_{D|w_\lambda|^2} = 0\ .
\end{equation}
We will use Hankel operators with symbols in $L^\infty_+$  to provide a formula for the commutators of Toeplitz operators.
Given a symbol $v\in L^\infty_+$, the Hankel operator $H_v$ is given by
$$
H_v : L^2_+ \to L^2_+, \, h \mapsto  H_v(h)=\Pi (v\overline h)\ .
$$
Notice that $H_v(h)$ is antilinear in $h$, but linear in $v$. 
\begin{lemma}\label{commutator of Toeplitz operators}
For any $v,w\in L^\infty_+$
\begin{equation}\label{bracket}
[ T_v, T_{\overline w} ] = \langle \, \cdot \, |  w\rangle v - H_v H_w \ .
\end{equation}
Furthermore, for any $u\in L^\infty_r$,
\begin{equation}\label{HT}
H_{T_uw}=T_wH_{\Pi u} +H_wT_u- \langle \Pi u |  \, \cdot \rangle w 
\end{equation}
\begin{equation}\label{HT2}
H_{T_uw} = H_{\Pi u}T_{\overline w} + T_uH_w - \langle w | \, \cdot \rangle \Pi u\ . 
\end{equation}
\end{lemma}
\noindent
{\em Proof of Lemma \ref{commutator of Toeplitz operators}.} We shall make use of the following elementary identities,
valid for any $f\in L^2$ and $ h\in L^2_+$,
$$ f=\Pi f+\overline{\Pi \overline f}- \langle f, 1 \rangle \ ,\quad  \Pi (f \overline h)=\Pi ((\Pi f) \overline h)\ .$$
Then, for any $h\in L^2_+$,
$$
[T_v,T_{\overline w}]h = \Pi (v\Pi (\overline wh))-\Pi (\overline w\Pi (vh))=\Pi (v\Pi (\overline wh))-\Pi (\overline w vh)\,.
$$
Using that $\overline w h = \Pi (\overline w h) + \overline{ \Pi (w \overline h) } - \langle \overline w h | 1 \rangle$ one then gets
$$
[T_v,T_{\overline w}]h = -\Pi (v\overline{\Pi (w\overline h)})+ \langle \overline w h | 1 \rangle v\ ,
$$
which yields \eqref{bracket}. As for \eqref{HT}, we have
$$
H_{T_uw}h = \Pi (\Pi (uw)\overline h) = \Pi (uw\overline h)\ .
$$
Since $u \overline h = \Pi (u \overline h) + \overline{ \Pi (\overline u h)} - \langle u | h \rangle$ and $u$ is real valued
one obtains 
\begin{eqnarray*}
H_{T_uw}h &=&\Pi (w\Pi (u\overline h))+\Pi (w\overline{\Pi (u h)}) - \langle u | h \rangle w\\
&=&\Pi (w\Pi ((\Pi u) \overline h))+\Pi (w\overline{\Pi (u h)})- \langle \Pi u | h \rangle w
\end{eqnarray*}
which establishes \eqref{HT}.
The identity \eqref{HT2} is proved in a similar way or alternatively,  by taking adjoints with respect to the real inner
product, induced by $\langle \ \cdot | \, \cdot \rangle$.
\hfill $\square$

\vskip 0.25cm
Let us now return to the proof of Proposition \ref{laxpairHlambda} and 
complete the verification of identity \eqref{identity}. We write $w$ for $w_\lambda $ for simplicity.
Using \eqref{bracket}, the Leibniz rule, and the antilinearity of $H_w$, we get
\begin{eqnarray*}
&&[D, T_wT_{\overline w}] = [D,T_{\overline w}T_w] + [D, \langle \, \cdot \, |  w \rangle w - H_w^2]\\
&&=T_{D\overline w}T_w + T_{\overline w}T_{Dw}+ \langle \, \cdot \, \vert w \rangle Dw- \langle \,  \cdot \, \vert Dw \rangle w - H_{Dw}H_w + H_wH_{Dw}\\
&&=T_{D|w|^2} + \langle \,  \cdot \, |  w\rangle Dw - \langle \, \cdot \, |  Dw \rangle w - H_{Dw}H_w + H_wH_{Dw}\ .
\end{eqnarray*}
Regarding $[T_u,T_wT_{\overline w}]$, note that  $[T_u,T_wT_{\overline w}] =[T_u,T_w]T_{\overline w}+T_w[T_u,T_{\overline w}]$.
Using $u = \Pi u + \overline{ \Pi u} - \langle u | 1 \rangle$ and again \eqref{bracket}, one gets
\begin{eqnarray*}
[T_u,T_wT_{\overline w}]&=&[T_{\overline {\Pi u}},T_w]T_{\overline w} + T_w[T_{\Pi u},T_{\overline w}]\\
&=&H_wH_{\Pi u}T_{\overline w} - T_wH_{\Pi u}H_w + \langle \  \cdot \, |  w \rangle T_w(\Pi u) - \langle \, \cdot \, | T_w (\Pi u) \rangle  w\ .
\end{eqnarray*}
Combining these two identities, we get
\begin{eqnarray*}\label{identity for T D|w|^2}
&&[D-T_u,T_wT_{\overline w}] - T_{D|w|^2}= - (H_{Dw}-T_wH_{\Pi u})H_w + H_w(H_{Dw} - H_{\Pi u}T_{\overline w}) \\
&& +  \langle  \, \cdot\, | \, w \rangle (Dw - T_w(\Pi u)) - \langle \, \cdot \, , Dw-T_w(\Pi u) \rangle w\,.
\end{eqnarray*}
We now have to analyze the terms $H_{Dw} - T_wH_{\Pi u}$ and $H_{Dw} - H_{\Pi u}T_{\overline w}$ in more detail.
Since $w = (L_u + \lambda)^{-1}1$, hence
$Dw = T_u w + 1  - \lambda w$, and since by \eqref{HT}
$$
T_wH_{\Pi u} = H_{T_uw}   - H_wT_u  + \langle \Pi u |  \, \cdot \rangle w
$$
one sees that
$$
\begin{aligned}
H_{Dw} - T_wH_{\Pi u} & =  H_{T_u w + 1  - \lambda w}  -  H_{T_uw}   +  H_wT_u  - \langle \Pi u |  \, \cdot \rangle w\\
& = H_1 - \lambda H_w + H_wT_u  - \langle \Pi u |  \, \cdot \rangle w\,.
\end{aligned}
$$
Using that $u$  is real valued and hence $\langle \Pi u | \, H_w h \rangle  = \langle h | \, H_w(\Pi u)\rangle$
one is led to the following identity for the term $- (H_{Dw}-T_wH_{\Pi u})H_w$,
$$
- (H_{Dw} - T_wH_{\Pi u} ) H_w = - H_1 H_w  + \lambda H_w H_w  - H_wT_u H_w  + \langle \, \cdot \, | \, H_w(\Pi u) \rangle w \,.
$$
Similarly, by  \eqref{HT2}
$$
H_{\Pi u}T_{\overline w} =  H_{T_uw} - T_uH_w +  \langle w | \,  \, \cdot \rangle \Pi u\
$$
one sees that
$$
\begin{aligned}
H_{Dw} - H_{\Pi u}T_{\overline w}  & =  H_{T_u w + 1  - \lambda w}  - H_{T_uw} + T_uH_w -  \langle w\vert  \, \cdot \rangle \Pi u\\\
& = H_1  - \lambda H_w  + T_uH_w -  \langle w\vert \, \cdot \rangle \Pi u\ .
\end{aligned}
$$
Finally, since $H_w$ is antilinear and $\lambda$ is real one concludes that
$$
H_w (H_{Dw} - H_{\Pi u}T_{\overline w} ) = H_w H_1 - \lambda H_w H_w + H_wT_uH_w -  \langle \, \cdot \, |  \, w \rangle H_w(\Pi u)\,.
$$
Substituting these identities into the formula for $[D-T_u,T_wT_{\overline w}] - T_{D|w|^2}$,
 obtained above, one infers that
$$
\begin{aligned}
&[D-T_u,T_wT_{\overline w}] - T_{D|w|^2}= -H_1H_w+H_wH_1\\
&+ \langle  \cdot \, | w \rangle \big( Dw - T_w(\Pi u) - H_w(\Pi u) \big) - \langle  \cdot \, | \,  Dw-T_w(\Pi u) - H_w(\Pi u) \rangle w\,.
\end{aligned}
$$
By \eqref{HT},
$$
\begin{aligned}
T_u(w)  & = H_{T_u (w)} (1) =T_wH_{\Pi u}(1) +H_wT_u(1)  - \langle \Pi u |  1 \rangle w\\
& = T_w(\Pi u) + H_w(\Pi u) - \langle u |  1 \rangle w\,.
\end{aligned}
$$
The identity $Dw = T_u w + 1  - \lambda w$ then reads
$$
Dw - T_w(\Pi u) - H_w(\Pi u) =  1 - \lambda w  - \langle u |  1\rangle w\,.
$$
Hence
$$
[D-T_u,T_wT_{\overline w}] - T_{D|w|^2}
= - H_1H_w + H_wH_1 + \langle  \cdot \, |  w \rangle 1 - \langle  \cdot \, | 1\rangle w
$$
which vanishes since $\langle \, \cdot \, |  w \rangle 1 = H_1H_w$ and $\langle \,  \cdot \, | 1\rangle w = H_wH_1$.
\hfill $\square$

\medskip

The second main result of this section calculates the Poisson brackets of $\langle 1 |  f_n   \rangle$ with the functionals $\gamma_p$. Recall that for any $p \ge 1,$ $\nabla \gamma_p$ is in $H^1$
(cf. Corollary \ref{formula nabla lambda_n}) 
and that for any $n \ge 0$, also
$\nabla \langle 1 \, | \,   f_n \rangle$ is in $H^1$ 
(cf. Lemma \ref{1 | f_n in H^1}). Hence
 the Poisson bracket $\{ \gamma_p, \langle 1 | f_n \rangle \}$ is well-defined on $L^2_r$
for any $p \ge 1$ and $n \ge 0$.
\begin{proposition}\label{gamma,1,f} For any $p\ge 1, n\ge 0,$ and $u \in L^2_{r,0}$,
$$\{ \gamma_p, \langle 1 | f_n \rangle\}(u) = i\, \langle 1 |  f_n( \cdot, u) \rangle \delta_{pn} \ . $$
\end{proposition}
\begin{proof}
Again we use the Hamiltonian flow of the generating function $\mathcal H_\lambda $ and its Lax pair formulation provided by Proposition \ref{laxpairHlambda},
$$\frac{dL_u}{dt} = [B_u^\lambda , L_u]\ ,\qquad  B_u^\lambda =iT_{w_\lambda }T_{\overline w_\lambda}\ .$$
Given $u_0\in L^2_{r,0}$, $n \ge 0,$ and $\lambda \ge | - \lambda_0 + 1|$, denote by $u=u(t)$ the trajectory issuing from $u_0$ at $t=0$, and define $g_n^\lambda \equiv g_n^\lambda (t)$ by 
$$
\frac{dg_n^\lambda }{dt} = B_{u}^\lambda g_n^\lambda\ ,\qquad  g_n^\lambda (0)=f_n(u_0)\ .
$$
Since  $B_u^\lambda$ is skew symmetric, $\|f_n(\cdot , u_0) \| =1$ ,
$\| g_n^\lambda (t) \| =1$ and since 
$$
\frac{d}{dt}  \Big( (L_{u} - \lambda_n )g_n^\lambda \Big) 
= B_u^\lambda(L_{u} - \lambda_n )g_n^\lambda \ ,\qquad \big((L_{u} - \lambda_n )g_n^\lambda \big) |_{ t = 0} =0\ ,
$$
$g_n^\lambda (t)$ is an eigenfunction of $L_{u(t)}$ corresponding to the eigenvalue $\lambda_n$, coinciding with the eigenfunction $f_n( \, \cdot , u(t))$ up to a phase factor. 
To determine this phase factor we need to compute the evolution of $\langle g_0^\lambda |  1 \rangle$ and of 
$\langle g_n^\lambda | Sg_{n-1}^\lambda \rangle$, $n \ge 1$. Using that 
$$
T_{\overline w_\lambda }1 = {\overline {\langle w_\lambda | 1 \rangle}} 1\ , \qquad 
 \langle g_n^\lambda \, | \, w_\lambda \rangle =  \langle  (L_u + \lambda)^{-1}g_n^\lambda \, | \, 1 \rangle = \frac{1}{\lambda_n + \lambda} \langle g_n^\lambda \, | \, 1 \rangle 
 $$ 
 one gets for any $n \ge 0$
$$
\frac{d}{dt} \langle g_n^\lambda | 1\rangle = \langle B_{u}^\lambda g_n^\lambda | 1 \rangle = i \langle g_n^\lambda |  T_{w_\lambda}T_{\overline w_\lambda }1 \rangle 
 = i \langle w_\lambda | 1\rangle  \langle g_n^\lambda \, | \, w_\lambda \rangle 
 $$
 implying that
 \begin{equation}\label{evolution g_0 | 1 }
 \frac{d}{dt} \langle g_n^\lambda | 1\rangle
= i \mathcal H_\lambda  \frac {\langle g_n^\lambda | 1\rangle}{\lambda_n+\lambda} \ .
\end{equation}
Similarly, using in addition that $B_{u}^\lambda$ is skew symmetric and that by \eqref{T_uS}
$$
ST_{w_\lambda} = T_{w_\lambda} S - \langle S w_\lambda \cdot \,  | 1\rangle = T_{w_\lambda} S 
$$ 
one has for any $n \ge 1$
$$
\frac{d}{dt} \langle g_n^\lambda | \, Sg_{n-1}^\lambda \rangle =  \langle g_n^\lambda \, | \,  [S,B_{u}^\lambda ]g_{n-1}^\lambda \rangle 
= i \langle g_n^\lambda \, | \,  T_{w_\lambda} [T_{\overline w_\lambda},S ]g_{n-1}^\lambda \rangle\,.
$$
Using  \eqref{T_uS} once more one gets 
$$
T_{w_\lambda} [T_{\overline w_\lambda},S ]g_{n-1}^\lambda  =   \langle S {\overline w_\lambda} g_{n-1}^\lambda \,  | 1 \rangle T_{w_\lambda}1
=  { \overline { \langle   w_\lambda  \,  | S g_{n-1}^\lambda \rangle}} w_\lambda
$$
and in this way concludes that
$$
\frac{d}{dt} \langle g_n^\lambda | \, Sg_{n-1}^\lambda \rangle
=   i \langle g_n^\lambda \, | \, w_\lambda \rangle \langle w_\lambda \, | \, Sg_{n-1}^\lambda \rangle 
 =   i\frac{\langle g_n^\lambda | 1\rangle}{\lambda_n+\lambda}  \langle w_\lambda | Sg_{n-1}^\lambda \rangle\ .
$$
Note that by \eqref{L_uS}
$$(\lambda_{n-1}+\lambda +1)Sg_{n-1}^\lambda =   (L_u+\lambda)Sg_{n-1}^\lambda  + \langle Sg_{n-1}^\lambda  | \, u \rangle 1 \ ,$$
and that 
$$
\langle w_\lambda  \, | \,  (L_u+\lambda)Sg_{n-1}^\lambda \rangle = \langle 1 \, | \, Sg_{n-1}^\lambda \rangle = 0\ ,
$$
yielding 
$$
\langle w_\lambda  \, | \, Sg_{n-1}^\lambda \rangle
= \frac{\langle u | Sg_{n-1}^\lambda \rangle \mathcal H_\lambda}{\lambda_{n-1}+\lambda +1}\ .
$$
Since by Lemma \ref{formula with langle f_p | 1 rangle}
$$
 - \langle g^{\lambda}_n | 1\rangle \langle u | Sg^{\lambda}_{n-1}\rangle  = (\lambda_n-\lambda_{n-1}-1) \langle g^{\lambda}_n | Sg^\lambda_{n-1} \rangle
  = \gamma_n \langle g^{\lambda}_n | Sg^\lambda_{n-1} \rangle 
$$
we  end up with the following identity
\begin{align}\label{evolution g_n | S g_(n-1)}
\frac{d}{dt} \langle g_n^\lambda \, | \, Sg_{n-1}^\lambda \rangle & =
- i \mathcal H_\lambda \cdot \frac{1}{\lambda_n+\lambda } \cdot \frac{\gamma_n}{\lambda_{n-1}+\lambda +1} \cdot \langle g_n^\lambda | \, Sg_{n-1}^\lambda \rangle \nonumber \\
&= i \mathcal H_\lambda \cdot \left ( \frac{1}{\lambda_n + \lambda} - \frac{1}{\lambda_{n-1}  + \lambda + 1} \right ) \langle g_n^\lambda | \, Sg_{n-1}^\lambda \rangle\
\end{align}
In view of \eqref{evolution g_0 | 1 }  and \eqref{evolution g_n | S g_(n-1)} we set
$$
\widetilde g_p^\lambda (t) = g_p^\lambda (t)\,{\rm e}^{-it\alpha_p(\lambda )}\ ,\ p\ge 0
$$
where $\alpha_0(\lambda ) : = \frac{\mathcal H_\lambda}{\lambda_0+\lambda}$ and for any $n \ge 1$,
$$
\alpha_n(\lambda ) : =\mathcal H_\lambda \cdot  \left ( \frac{1}{\lambda_0+\lambda} + \sum_{p=1}^n  \left (\frac{1}{\lambda_p+\lambda} -  \frac{1}{\lambda_{p-1}+\lambda + 1}\right )\right ) \ ,
$$
so that
$$\frac{d}{dt} \langle \widetilde g_0^\lambda | 1\rangle = 0\ ,\qquad
\frac{d}{dt} \langle \widetilde g_n^\lambda | S\tilde g_{n-1}^\lambda \rangle =0\ ,\ n\ge 1\ ,$$
implying that for any $p \ge 0$,  $\widetilde g_p(t) = f_p(u(t)).$
Taking into account \eqref{evolution g_0 | 1 }, one sees that for any $\lambda \ge | - \lambda_0 + 1|$ and $n \ge 1,$ $\{ \mathcal H_\lambda , \langle 1 |  f_n   \rangle\} 
=\frac{d}{dt} \langle 1\, | \,  f_n \rangle$, when evaluated at $t = 0$, is given by
$$
\begin{aligned}
& \{ \mathcal H_\lambda , \langle 1 |  f_n  \rangle\}  =
 i \langle 1 | f_n   \rangle   \left ( \alpha_n - \frac{\mathcal H_\lambda}{\lambda_n + \lambda} \right )\\
& =  i \langle 1 | f_n   \rangle \cdot  \mathcal H_\lambda \cdot  
\left (  \frac{1}{\lambda_0+\lambda} - \frac{1}{\lambda_n + \lambda}  + \sum_{p=1}^n  \frac{1}{\lambda_p+\lambda} -  \frac{1}{\lambda_{p-1}+\lambda +1}	 \right )
\end{aligned}
$$
or expressed in a shorter form,
$$
 \{ \mathcal H_\lambda , \langle 1 |  f_n  \rangle\}  =  i \langle 1 | f_n   \rangle \cdot  \mathcal H_\lambda \cdot  
 \sum_{p=0}^{n-1}   \left (  \frac{1}{\lambda_p+\lambda} -  \frac{1}{\lambda_{p}+\lambda +1}  \right )\, .
$$
On the other hand, using  formula \eqref{expandHlambda} for $ \mathcal H_\lambda$, one computes
$$
\{ \mathcal H_\lambda , \langle 1 | f_n \rangle\} =
\sum_{p=0}^\infty \left (-\frac{| \langle 1 | f_p \rangle |^2}{(\lambda_p+\lambda)^2}\{ \lambda_p, \langle 1 | f_n  \rangle\}+\frac{\{ |\langle 1 | f_p \rangle |^2,\langle  1 | f_n   \rangle\} }{\lambda_p+\lambda }   \right )\ .
$$
Since the left and right hand side of the latter identity are meromorphic functions in $\lambda$,
they have the same poles. In particular
 the principal part of the expansion at any of their poles of order two coincide, hence
for any $p \ge 0$ with $\langle 1 | f_p \rangle \ne 0$ one obtains
$$\{ \lambda_p, \langle 1 | f_n   \rangle\} =\begin{cases} -  i \langle  1 | f_n   \rangle\ &{\rm if}\ 0\le p\le n-1\ , \\
0\ &{\rm if}\  p\ge n\ .  \end{cases}$$
Since, for $p\ge 0$,  $\lambda_p=p-\sum_{k=p+1}^\infty \gamma_k\ ,$
it follows that for any $k \ge 1,$
\begin{equation}\label{pb gamma_k 1 | f_n}
\{ \gamma_k, \langle  1 | f_n \rangle \}  =   i \langle 1, f_n \rangle \delta_{kn}
\end{equation}
at least on the set
$$\mathcal O_\infty =\{ u\in L^2_{r,0} : \ \gamma_j(u)>0\ \, \forall j \ge 1\}\ .$$
Since for any $k \ge 1,$ $\gamma_k$ is a real analytic, not identically vanishing functional,
$\mathcal O_\infty $ is a dense $G_\delta $ subset of $L^2_{r,0}$, and hence by continuity,  
 identity \eqref{pb gamma_k 1 | f_n} holds on all of $L^2_{r,0}$.
\end{proof}

For any $n \ge 1$, on the dense open subset $L^2_{r,0} \setminus Z_n$ of $L^2_{r,0}$, consisting of potentials $u$ with $\gamma_n(u) >  0$, 
the argument $\varphi_n$ of $\langle 1 | f_n \rangle$ is well defined.
Since $| \langle 1 | f_n \rangle | = \gamma_n^{1/2} \kappa_n^{1/2}$, Corollary \ref{bracketlambda} and
 Proposition \ref{gamma,1,f} imply that
\begin{equation}\label{gammaphi} 
\{ \gamma_p,\varphi_n\} =\delta_{pn}\ ,\ p\ge 1\ , \qquad \forall u \in L^2_{r,0} \setminus Z_n \ .
\end{equation}

In order to show that the angle variables $\varphi_n$, $n \ge 1$, are conjugate to the action variables $\gamma_p$, $p \ge 1$,
it remains to establish on $L^2_{r,0} \setminus (Z_p \cup Z_n)$ the complementary identities
\begin{equation}\label{phiphi}
\{ \varphi_p,\varphi_n\}=0\ .
\end{equation}
In the subsequent section, we prove these identities for $1 \le p, n \le N$ with $N \ge 1$ arbitrary on the set
$$
\mathcal O_N =\{ u\in L^2_{r,0} : \ \gamma_j(u)>0\ \, \forall 1 \le j \le N;  \gamma_j(u) = 0\ \, \forall  j \ge N +1\}\ .
$$
In fact, we show that $\mathcal O_N$ is a symplectic manifold of dimension $2N$ with 
$\gamma_1,\dots ,\gamma_N$ and $\varphi_1,\dots ,\varphi_N$ as action and angle variables in the classical sense.


\section{Finite gap potentials}\label{finite gap potentials}
In this section we study the set $G_J$ of $J-$gap potentials, introduced in Section \ref{Lax operator},
with $J \subset \N$ finite (cf. Definition \ref{definition finite gap}).
In fact, to simplify the exposition we will consider only the case $J = \{ 1, \ldots , N \}$
where $N \ge 1$ is an arbitrary integer. Then
$$
G_{\{ 1, \ldots , N\} } =  \{ u \in L^2_r \, : \, \gamma_j (u)> 0 \,\, \forall 1 \le  j  \le N; \,\, \gamma_j(u) = 0 \,\, \forall  j > N  \}\ .
$$
Note that 
$$
G_{\{ 1, \ldots , N\} } = \bigcup_{c \in \R} (c + \mathcal O_N)\ , \qquad \mathcal O_N = G_{\{ 1, \ldots , N\} }  \cap L^2_{r,0}\, ,
$$
where $c + \mathcal O_N$ are the level sets of the restriction of the average functional $L^2_r \to \R, \, u \mapsto \langle u \, | \, 1 \rangle$,
to $G_{\{ 1, \ldots , N\} }.$  Furthermore, $L^2_{r,0}$ is endowed with the symplectic form
\begin{equation}\label{symplectic form}
\omega (u, v)= \langle u \, | \,  \partial_x^{-1}v \rangle ,
\end{equation}
so that one has the identity 
$\{ F,G\} = 
\omega (\partial_x(\nabla F), \partial_x(\nabla G))$. 
Note that the average functional is a Casimir for the Gardner bracket.
The aim of this section is to characterize elements in $\mathcal O_N$  and to show that the restriction of $\Phi$
to $\mathcal O_N$ is a real analytic symplectic diffeomorphism onto its image where we recall 
that $\Phi$ denotes the Birkhoff map, introduced in \eqref{Phionto}. 
As an important application we will prove that $\Phi$ is onto (cf. Corollary \ref{Phionto}).
We refer to Appendix \ref{general one gap potentials} for a study of the set of one  gap potentials $G_{\{ N\}}$ with $N \ge 1$
arbitrary.

Actually, we consider a slightly larger set than $\mathcal O_N$: for any $N \ge 1, $ define
\begin{equation}\label{definition U_N}
\mathcal U_N := \{ u \in L^2_{r,0}\, : \,  \gamma_N(u)>0\ ,\,\, \gamma_j(u)=0 \, \, \forall \,  j > N \} .
\end{equation}
It is convenient to define $\mathcal U_0 := \{ 0 \}$.
Note that for any $u \in \mathcal U_N$ and $n \ge N$ one has by \eqref{expandlambda} and
 Lemma \ref{f_n for gamma_n = 0}
 $$
 \lambda_n(u) = n\,, \qquad f_n(x, u) = e^{i(n-N)x} f_N(x, u)\, .
 $$
Furthermore, denote by $ \C_{N}[z]$ \big(respectively $ \C_{\le N}[z]$\big) the space of polynomials $P$ in the complex variable $z$ of degree  $N$
\big(of degree at most $N$\big) and by $ \C_{N}^+[z]$ the open subset  of polynomials $P\in \C_N[z]$ with the property that  $\{ P(z)=0\} \subset \{ |z|>1\}$.
We identify $\C_N[z]$ with the space $\C^{N} \times \C^*$
of coefficients of polynomials in $\C_N[z]$.
The following theorem characterizes elements in $\mathcal U_N$.
\begin{theorem}\label{Ngap}
For any $N\ge 1$
\begin{equation}\label{characterization}
\mathcal U_N =\big\{   u=h+\overline h \, : \, \ h(x)=-{\rm e}^{ix}\frac{Q'({\rm e}^{ix})}{Q({\rm e}^{ix})}\, , \, \, Q \in \C_{N}^+[z]
\big\} 
\end{equation}
where $Q'(z) := \partial_z Q(z).$
Furthermore, $\mathcal U_N$ is a connected, 
real analytic,
symplectic submanifold of $L^2_{r,0}$ of dimension $2N$. 
The restriction $\Phi_N$ of $\Phi$ to $\mathcal U_N$,
$$\Phi _N:   \mathcal U_N \to \C^{N-1}\times \C^* , \,   u  \mapsto (\zeta_n(u))_{1\le n\le N}$$
is a real analytic, symplectic diffeomorphism,
$$
(\Phi_N)_*\omega =
i\sum_{n=1}^N d\zeta_n\wedge d \overline{ \zeta_n}\ .
$$
In particular, $\mathcal O_N$
is  a dense open subset of $\mathcal U_N$ with $\Phi _N(\mathcal O_N) = (\C^*)^N$.
\end{theorem}
\begin{remark}\label{formula Ngap}
(i) Alternatively, potentials $u \in \mathcal U_N$ with $N \ge 1$ can be written in the form (cf. formula \eqref{formula Pi u} below)
\begin{equation}\label{u in terms of roots}
u(x) = \sum_{k=1}^{\infty} (\sum_{j = 0}^{N-1}  q_j^k ) e^{ikx} + \sum_{k=1}^{\infty} (\sum_{j = 0}^{N-1} \overline{q_j}^k) e^{- ikx} 
\end{equation}
where $1/q_0, \ldots , 1/q_{N-1} \in \C^\ast$ denote the roots of the polynomial
$Q(z)$ of degree $N$, defined in \eqref{definition Q} below. Note that the above formula for $u$ can also be expressed in terms of the Poisson kernel
$P_r (x)= \frac{1 - r^2}{1 - 2r\cos(x) + r^2}$,
$$
u(x) = \sum_{j = 0}^{N-1} ( P_{r_j} (x + \alpha_j) -1 )\,, \quad q_j = r_j e^{i \alpha_j} \,\,\, \forall \,  0 \le j \le N-1\ .
$$
(ii) Real coordinates on $\mathcal U_N$ are given by the
real and imaginary parts $(Q_{j,1})_{1 \le j \le N}$, 
$(Q_{j,2})_{1 \le j \le N}$, 
of the coefficients $(Q_j)_{1 \le j \le N}$
of the polynomial $Q \in \C_{N}^+[z]$ normalized by $Q(0)=1$, $Q(z) = 1 + \sum_{j =1}^N 
(Q_{j,1} + i Q_{j,2})z^j$. 
The corresponding potential $u$ can then be written as
$$
u(x) = - e^{ix}\frac{Q'(e^{ix})}{Q(e^{ix})} 
 - e^{-ix}\frac{P'(e^{-ix})}{ P(e^{-ix})}
$$
where
$$
P(z) = 1 + \sum_{j =1}^N P_jz^j\ , \qquad P_j := Q_{j,1} - i Q_{j,2}\ .
$$
Note that this formula for $u$ is real analytic in the coordinates
$Q_{j,1}$, $Q_{j,2}$, $1 \le j \le N$. Complex valued coefficients
$Q_{j,1}$, $Q_{j,2}$, $1 \le j \le N$ then will lead to complex valued
potentials. \\
Furthermore note that by Proposition \ref{Birkhoff map},
for any $n \ge 1,$ the real part $\zeta_{n, 1}$ and the imaginary part $\zeta_{n, 2}$ of $\zeta_n$ are real analytic functionals on
$L^2_{r,0}$ and so are their restrictions to $\mathcal U_N$.
It then follows that $\Phi_N$ is real analytic when viewed as map
$\mathcal U_N \to \R^{2(N-1)} \times \R^2\setminus \{(0, 0)\}$.
\end{remark}
\begin{proof}
We first prove that any element $ u \in \mathcal U_N$ has the claimed form. 
To simplify notation within this proof, we do not indicate the dependence of various quantities on $u$.
Our starting point is the inverse formula
\eqref{recover}, stated in Lemma \ref{inverse formula},
$$
\Pi u(z)= \langle (Id-zM)^{-1}X | Y \rangle _{\ell ^2}\, , \quad  z\in \C, \,\, |z| < 1,
$$
where $\ X = - (\lambda_p \langle 1 | f_p \rangle )_{p\ge 0}$ and $Y =  (\langle 1 | f_n \rangle)_{n\ge 0}$.
According to the computations in the proof of Proposition \ref{injectivity}, the $n$--th row of  $M$ with $n \ge N$
in the case at hand satisfies
$$
(M_{np} )_{p \ge 0} = (\delta_{ n+1, p})_{p \ge 0} \, \quad  \forall \, n \ge N\ .
$$
Since the coefficients of $X=(-\lambda_p \langle 1 \, | f_p \rangle )_{p\ge 0}$ with $p\ge N+1$ vanish, 
$(\xi_n(z))_{n\ge 0}:=(Id-zM)^{-1}[X]$ satisfies 
$$
\xi_N(z) - z\xi_{N+1}(z)= - \lambda_N \langle 1 | f_N \rangle, \quad 
 \xi_n(z) - z\xi_{n+1}(z)=0, \,\,\,  \forall \, n\ge N+1,
$$
or
$$
\xi_{N+1}(z) = z^{n-N-1}\xi_n(z) \ ,\quad \forall \, n \ge N+1.
$$
Combined with the fact that $\sum_{n \ge 0} | \xi_n(z) |^2 < \infty $ for any $|z|<1$, we infer that
$$
\xi_n(z)=0\ ,\ n\ge N+1\ ,  \qquad  \xi_N(z)= - \lambda_N \langle 1 | f_N \rangle \ .
$$
This means that
\begin{equation}\label{inverse formula for N gap}
\Pi u(z)=\langle (Id - zM_N)^{-1}X_N \, | \, Y_N \rangle _{\C^{N+1}}\ ,
\end{equation}
where
$$
X_N=(-\lambda_p \langle 1 | f_p \rangle )_{0\le p\le N},\,\,\, Y_N=(\langle 1 | f_n \rangle)_{0\le n\le N},\,\,\, M_N=(M_{np})_{0\le n,  p\le N}.
$$
Notice that the last row of $M_N$ is $0$. With the same arguments used to derive the formula for $\Pi u(z)$ (cf. \eqref{Piu=})
one obtains a similar formula for the extension $f_n(z)$ of the eigenfunctions $f_n$ of $L_u$ to the unit disc. 
For $0 \le n \le N$ and $|z| < 1$ one has
$$
f_n(z)=\langle (Id-zM_N)^{-1}{\bf 1}_n | Y_N \rangle_{\C^{N+1}}\ ,
$$
where ${\bf 1}_n:=(\delta_{pn})_{0\le p\le N}\ .$ We thus conclude that 
$\Pi u$ and $f_0,\dots ,f_N$  belong to the $\C-$vector space 
\begin{equation}\label{definition Q}
\mathcal R_N := \left \{\frac{P(z)}{Q(z)}\ : \, P\in \C_{\le N}[z] \right \}, \qquad Q(z):=\det (Id - zM_N),
\end{equation}
and that $f_0,\dots ,f_N$ are $N+1$ linearly independent elements in $\mathcal R_N$ and thus form a basis. As a consequence, $1 / Q(z)$ belongs to the Hardy space of the unit disc, hence $Q(z)$ cannot vanish at any point $z$ with $|z|\le 1$. Furthermore one has $Q(0)=1$. 
To see that $\deg (Q)=N$ first note that the last row of $M_N$ is identically zero and thus
\begin{equation}\label{formula Q(z)}
\det (Id - zM_N)=\det (Id - zM_{N-1}),
\end{equation}
where $M_{N-1}=(M_{np})_{0\le n,  p\le N-1}$. It is to prove that $\det (M_{N-1})\ne 0\ .$ 
The formula \eqref{Mnp} for the coefficients $M_{np}$ 
for $\zeta_{n+1}\ne 0$ and the definition \eqref{fSf} of $\mu_{n+1}$ imply
that in the case at hand the coefficients $M_{np}$ with $0 \le n, p \le N-1$ read
\begin{equation}\label{formulas for Mnp}
M_{np}=\begin {cases} \delta_{p,n+1}\ &{\rm if}\ \zeta_{n+1}=0\ ,\\
\sqrt{\mu_{n+1}}\gamma_{n+1}
\frac{ \langle f_p | 1 \rangle} { (\lambda_p-\lambda_n-1) \langle f_{n+1} | 1\rangle}\  &{\rm if}\  \zeta_{n+1} \ne 0\ .
\end{cases}
\end{equation}
Expanding $\det (M_{N-1})$ by the formula for Cauchy determinants, one sees that 
\begin{eqnarray*}
&&|\det(M_{N-1})| = \left |\det \left ( \sqrt{\mu_n} \gamma_n  \frac{\langle f_p | 1 \rangle} { (\lambda_p - \lambda_{n-1}-1) \langle f_{n} | 1 \rangle} \right )_{n \in J, p\in \tilde J  }   \right | \nonumber\\
&&=\left ( \prod_{n\in J}  \sqrt{\mu_n}\gamma_n  \right )\, \frac{\langle f_0\vert 1\rangle }{\vert \langle f_N\vert 1\rangle \vert } \left |\det\left ( \frac{1}{\lambda_{n-1}+1-\lambda_p}  \right )_{n\in J, p \in \tilde J} \right|  \,  \ne \, 0
\end{eqnarray*}
where  $J:=\{ 1 \le  j  \le N \, | \,  \gamma_j >0 \}, \tilde J:=\{ 0\} \cup J\setminus \{ N\}$.
Thus we showed that $\deg (Q)=N.$  

Next we prove that
$\Pi u(z) = - z Q'(z) /Q(z).$
Indeed since $ \langle  u | 1 \rangle=0$, there exists a polynomial $R \in  \C_{\le N-1} [z]$ so that
$\Pi u(z)=z R(z) / Q(z).$
Furthermore, since $L_uf_n = \lambda_n f_n$ and since $f_0, \ldots , f_N$  is a basis of $\mathcal R_N$,
$L_u$ leaves $\mathcal R_N$ invariant. Let us look at the action of $L_u$ on $\mathcal R_N$ in more detail.
For any $h \in L^2_+,$ one has
$$
T_u h= (\Pi u) h +H_h(\Pi u)\ ,\quad H_h(f) := \Pi (h \overline f) 
$$
By elementary properties of Hankel operators --- see e.g. \cite{GG0}, Prop. 11---, for every $h\in \mathcal R_N$,  the range of the Hankel operator $H_h$ is included in  $\mathcal R_N$.
Hence the linear map $h\mapsto L_uh - H_h(\Pi(u)) =  - i \partial_xh -(\Pi u)h$ leaves $\mathcal R_N$ invariant. 
 Expressed in an alternative way, it means that $$z\partial _z - z\frac{R(z)}{Q(z)}$$
leaves $\frac{1}{Q(z)} \C_{\le N}[z]$ invariant. Using the latter fact, one verifies in a straightforward way that
$R(z)= - Q'(z),$ showing that $u$ has indeed the claimed form
$$
\Pi(u)(z) = - z \frac{Q'(z)}{Q(z)}\,.
$$
 Now let us prove the converse. Assume that $Q \in  \C_{N}^+[z]$ with $Q(0) = 1$. Thus $Q$ is of the form  
 $$
 Q(z)=\prod _{j=0}^{N-1}(1-q_jz)\, , \quad q_j \in \C \, ,  0<|q_j|<1 \ , \,\,\, \forall \,  0 \le j \le N-1\, ,
 $$
and in turn
 \begin{equation}\label{formula Pi u}
 \Pi u(z)=\sum_{j=0}^{N-1}\frac{q_j z}{1 - q_j z} \, .
 \end{equation}
 Let
$$f(z):=\prod_{j=0}^{N-1}\frac{z-\overline q_j}{1-q_jz}\ .$$
Clearly, $f$ is in $\mathcal R_N$. 
Writing  $L_u(f)$ as $z \partial_z f - (\Pi u) f - H_f(\Pi u)$ one computes
$$
L_u(f) = f \, \sum_{j=0}^{N-1}\left (\frac{z}{z-\overline q_j}+\frac{q_jz}{1-q_jz}-\frac{q_jz}{1-q_jz}-\frac{\overline q_j}{z-\overline q_j}\right )=Nf\ .
$$
It means that $f$ is an eigenfunction of $L_u$ with eigenvalue $N$.
Furthermore, since for any $|z|=1$, $$f(z)= z^N \overline{ Q(z)} / Q(z)\ ,$$
one verifies that for any $k\ge 1$, $S^k f$ is orthogonal to $\mathcal R_N$ and thus in particular to $\Pi u $. 
By Lemma \ref{vanishing of gamma_n} this implies  that for any $k \ge 1$, $S^{k}f$ is an eigenfunction of $L_u$ with eigenvalue $N+k$. 
Taking into account that $\mathcal R_N$ is of dimension $N+1$ and that the eigenvalues $\lambda_n$ satisfy  \eqref{expandlambda}
with $\langle \Pi u | 1 \rangle = 0$ one concludes that  $\lambda_{N+k} = N+k$ for any $k \ge 0$ 
and $\gamma_{N+k} = 0$ for any $k \ge 1.$
On the other hand, $\langle 1 | f \rangle=(-1)^{N}\prod_{j=0}^{N-1} q_j\ne 0$ implying that $\gamma_N \ne 0$.
Altogether this shows that $u \in \mathcal U_N$.

Having established identity \eqref{characterization} we know that $\Pi (\mathcal U_N)$ is a connected, complex manifold of dimension $N$, 
parametrized by the open subset of $ \C^N$ described by the coefficients of the polynomials $Q\in \C_N^+[z]$ satisfying $Q(0)=1$. 
Furthermore, since for any $u, v \in L^2_{r,0}$, 
$$
\omega( u, v) = \langle  u | \partial_x^{-1} v \rangle 
=  -2 {\rm Im} \, \langle D^{-1}\Pi u | \, \Pi v \rangle
$$
$\Pi (\mathcal U_N)$ is a K\"ahler manifold with Hermitian form $\langle D^{-1}\Pi u | \, \Pi v \rangle$
and hence $\mathcal U_N$ is a real analytic symplectic submanifold of $L^2_{r,0}$  of real dimension $2N$. 
We claim that, for every $1 \le n \le N $, $\gamma_n$ does not identically vanish on $\mathcal U_N$. 
Indeed, denote by $J$ the set of such indices $n$.
By the definition of $\mathcal U_N,$ $\gamma_N(u) \ne 0$ for any $u \in \mathcal U_N$ and hence $N \in J$. 
Since each $\gamma_n$ is a real analytic functional and since $\mathcal U_N$ is connected,  
$G_{J} \cap \mathcal U_N$ is an open dense subset of $\mathcal U_N$.
From the proof of Proposition \ref{injectivity} we know that the coefficients $M_{np}$ of the matrix $M_N$
can be expressed in terms of the Birkhoff coordinates $(\zeta_j)_{1 \le j \le N}$. Similarly, this is the case for the vectors
$X_N$ and $Y_N$, introduced above.
Hence the right hand side of the identity 
\eqref{inverse formula for N gap} locally extends analytically, yielding
a real analytic map $\Psi_N$ on $(\C^\ast)^J $ so that for any $u$
in the open subset $G_J \cap \mathcal U_N$ of $\mathcal U_N$,
$$
\Psi_N \circ \Phi_N (u) = \Pi(u)\ .
$$
This implies that the rank of the map $\Phi_N$ at any point in $G_J \cap \mathcal U_N$ has to be $2N$. 
Since  in view of the definition of $J$ it cannot be bigger than $2\sharp J$, we infer that $J=\{ 1,2,\dots ,N\}$. 
As a consequence,  $\mathcal O_N = G_{\{1,\dots,N \}} \cap \mathcal U_N$  is an open dense subset of $\mathcal U_N$ 
and a symplectic submanifold of $\mathcal U_N$ of dimension $2N$. 
On $\mathcal O_N$, the smooth functions $\gamma_1,\dots ,\gamma_{N}$, $\varphi_1,\dots ,\varphi_{N}$, referred to as actions and angles,  
are well defined. From Corollary \ref{bracketlambda} and identity \eqref{gammaphi}, we know that 
\begin{equation}\label{commutation relations A}
\{ \gamma_n,\gamma_p\}=0\ ,\ \{ \gamma_n,\varphi_p\}=\delta_{pn}\ ,\ 1\leq n,p\leq N\ .
\end{equation}
This implies that $\Phi_N$ is a local diffeomorphism on $\mathcal O_N$ and therefore, since $\Phi_N$ is one to one, 
 a diffeomorphism onto its range which is contained in $(\C^*)^N$. In order to prove that the range is $(\C^*)^N$, we observe that
$(\C^*)^N$ is connected and $\Phi_N$ proper (cf. Proposition \ref{Birkhoff map}). 
Being a local diffeomorphism, $\Phi_N$ is open and hence $\Phi_N (\mathcal O_N) = (\C^*)^N$. 
To prove that $\Phi_N$ is symplectic on $\mathcal O_N$ introduce the two--form
$$\tilde \omega =\omega -\sum_{n=1}^{N}d\gamma_n\wedge d\varphi_n\ .$$
In view of the commutation relations \eqref{commutation relations A}, 
$${\partial_{\varphi_p}} \,
\intprod \, \tilde \omega =0\ , \quad  \forall  \, 1 \le p \le  N\ .$$
Since moreover $\tilde \omega $ is a closed two form there exist smooth functions $a_{np}$ so that
\begin{equation}\label{omega tilde}
\tilde \omega =\sum_{1\le n<p\le N}a_{np}(\gamma_1,\dots ,\gamma_{N})\,  d\gamma_n\wedge d\gamma_p\ .
\end{equation}
On the other hand,  the pullback of $\sum_{n=1}^{N}d\gamma_n\wedge d\varphi_n$
to the submanifold $\varphi_1=\dots =\varphi_{N}=0$ vanishes 
and by the inverse formula \eqref{inverse formula for N gap}, on this submanifold, 
$\Pi u(\overline z)=\overline {\Pi u(z)}$, implying that $u$ is even, $u(-x)=u(x)$.
Therefore, on this submanifold, formula \eqref{symplectic form} leads to $\omega =0$. Altogether we thus have shown that $\tilde \omega =0$. 
By \eqref{omega tilde} we then infer that $\tilde \omega$ vanishes identically on $\mathcal O_N$, showing that $\Phi _N$  is symplectic on $\mathcal O_N$. 

Note that
$$
d\zeta_n =  \frac{1}{2\sqrt{\gamma_n}} e^{i\varphi_n} d \gamma_n + \sqrt{\gamma_n} e^{i \varphi_n} i d \varphi_n 
$$
and thus $d\zeta_n \wedge d \overline \zeta_n  = - i d \gamma_n \wedge d \varphi_n$.
When expressed in the coordinates $\zeta_n$, the pull back of the symplectic form $\omega$ to $\mathcal O_N$ is therefore  
$$\omega = i\sum_{n=1}^N d\zeta_n\wedge d\overline \zeta_n\ .$$
Since $\mathcal O_N$ is dense in $\mathcal U_N$, this identity holds on all of $\mathcal U_N$. 
Using again that $\Phi_N$ is one to one and proper, we conclude that  $\Phi_N$ is a symplectic diffeormorphism from $\mathcal U_N$ onto $\C^{N-1}\times \C ^*$.
By Remark \ref{formula Ngap} (ii), $\Phi_N$ is a real analytic map
and hence the proof of Theorem \ref{Ngap} is complete.
\end{proof}
As an important application of Theorem \ref{Ngap}, we prove that the Birkhoff map $\Phi$ is onto.
Recall that $u \in L^2_{r,0}$ is a finite gap potential if $J(u)= \{n \ge 1\ | \, \gamma_n(u) > 0 \}$ is finite.
The set of finite gap potentials in $L^2_{r,0}$ is thus given by $\{ 0\} \cup \bigcup_{N\ge 1}\mathcal U_N$.
\begin{corollary}\label{Phionto}
The map $\Phi : L^2_{r,0}\longrightarrow h^{1/2}_+\,, u \mapsto   (\zeta_n)_{n\ge 1} $ is a homeomorphism and the set of finite gap potentials is dense in $L^2_{r,0}$. 
Furthermore, the identities
$$\{ \zeta_n,\zeta_m\} =0\ ,\ 
\{ \zeta_n,\overline \zeta_m\} = - i \delta_{nm}\ ,\ 
n,m\ge 1\ , $$
are valid on all of $L^2_{r,0}$.
\end{corollary}
\begin{proof}
We already know that $\Phi $ is continuous and proper (cf. Proposition \ref{Birkhoff map}) and one to one (cf. Proposition \ref{injectivity}).
To prove that $\Phi $ is onto, let $\zeta :=(\zeta_n)_{n\ge 1}$ be any nonzero sequence in 
$h^{1/2}_+ \setminus \bigcup_{N \ge 1} 
\Phi (\mathcal U_N )$ 
where $\mathcal U_N$ is defined by \eqref{definition U_N}.
Then there  exists an increasing sequence $(N_k)_{k \ge 1}$ with $N_k\to \infty $ such that $\zeta_{N_k}\ne 0$ for any $k \ge 1$.
We approximate $\zeta$ by the truncated sequence 
$$\zeta^{(k)}_n= \begin{cases}\zeta_n \ &{\rm if}\ n\le N_k\\
0\ &{\rm if}\ n>N_k\ . \end{cases}$$
By Theorem \ref{Ngap}, there exists a unique element $u^{(k)}\in \mathcal U_{N_k}$ such that 
$\Phi (u^{(k)})=\zeta^{(k)}$.
Since $\Phi $ is proper and continuous, $u^{(k)}$ has a limit point $u$ in $L^2_{r,0}$ which satisfies
$$\Phi (u)=\zeta \ .$$
Thus we have proved that $\Phi$ is onto. Since $\Phi$ is proper and continuous it then follows that $\Phi^{-1}$ is continuous as well
and in turn $\bigcup_{N\ge 1}\mathcal U_N$ is dense in $L^2_{r,0}$. 
It remains to verify the claimed Poisson bracket 
relations. By Proposition \ref{zeta_n in H^1},
for any $n,$ $m \ge 1$,
the Poisson brackets $\{ \zeta_n, \zeta_m \}$
and  $\{ \zeta_n, \overline{\zeta_m} \}$ are real analytic
functionals on $L^2_{r,0}$.
Since by Theorem \ref{Ngap}
the claimed Poisson bracket relations hold on 
$\bigcup_{N \ge 1} \mathcal U_N$ and
$\bigcup_{N \ge 1} \mathcal U_N$ is dense in $L^2_{r,0}$
it then follows that they hold on all of $L^2_{r,0}$.
\end{proof}


\section{Proof of Theorem \ref{main result} and Theorem \ref{almost periodic} }\label{proof of Theorem 1}

In this section, we  make a detailed synopsis of the proof of Theorem \ref{main result}.
As an application, we derive formulas for the BO frequencies at the end of this section (cf. Proposition \ref{BO frequencies})
and prove Theorem \ref{almost periodic}.

\smallskip
\noindent
{\em Proof of Theorem \ref{main result}}
By Corollary \ref{Phionto}, $\Phi$ is a homeomorphism and the Poisson bracket identities (B2) hold.
It remains to express the BO Hamiltonian $\mathcal H(u)$, defined for $u \in H^{1/2}_{r,0}$, in terms of the Birkhoff coordinates $\zeta_n$, $n \ge 1$.
For $u \in  H^{1/2}_{r,0}$ one computes, 
using that
$\Pi u = 
- \sum_{n=0}^\infty \lambda_n \langle 1 | f_n \rangle f_n$,
\begin{eqnarray*}
\mathcal H(u) = \frac 12 \langle |D|u \,  | \, u \rangle -  \frac{1}{2\pi} \int _0^{2\pi} \frac{1}{3} u^3\, dx
= \langle L_u(\Pi u) |  \Pi u \rangle = \sum_{n=0}^\infty \lambda_n^3|\langle 1 | f_n \rangle|^2 \,.
\end{eqnarray*}
With the help of $\tilde {\mathcal H}_\e,$ defined by \eqref{definition tilde mathcal H} as
$$\tilde {\mathcal H}_\e =\sum_{n=0}^\infty \frac{|\langle 1 | f_n \rangle|^2}{1+\e \lambda_n}\ ,$$
one can express $\mathcal H(u)$ as follows,
\begin{equation}\label{formula BO Hamiltonian}
\mathcal H(u)=-\frac 16\frac{d^3}{d\e ^3}\big|_{ \e = 0}  \tilde{\mathcal H}_\e 
= - \frac 16 \frac{d^3}{d\e ^3}\big|_{ \e = 0} \log \tilde{\mathcal H}_\e  \ .
\end{equation}
Using the formula for $ \frac{d}{d\e }\log \tilde{\mathcal H}_\e $ derived in Section \ref{trace} (cf. \eqref{trace3}), one then obtains
$$\mathcal H(u)=\frac 13 \lambda_0^3+\frac 13 \sum_{n=1}^\infty \gamma_n \big( \lambda_n^2+(\lambda_{n-1}+1)^2+\lambda_n(\lambda_{n-1}+1) \big)$$
or using that  by \eqref{expandlambda},  $\lambda_n = n - s_{n+1}$ with $s_{n+1}:= \sum_{k=n + 1}^\infty \gamma_k$, $n\ge 0$,
\begin{equation}\label{formula A for H(u)}
\mathcal H(u)  = -\frac 13 s_1^3  + \sum_{n=1}^\infty  \gamma_n \big( \big( n - s_n \big)^2 + \gamma_n \big( n - s_n \big)+\frac 13 \gamma_n^2  \big)\,. 
\end{equation}
This proves Theorem \ref{main result}.
\hfill $\square$

\medskip

Actually, formula \eqref{formula A for H(u)} of the BO Hamiltonian $\mathcal H(u)$ can be further simplified as follows.
\begin{proposition}\label{BO frequencies}
For any $u\in H^{1/2}_{r,0}$,
\begin{equation}\label{formula Hgamma}
\mathcal H(u)=\sum_{n=1}^\infty n^2\gamma_n-\sum_{n=1}^\infty \big(\sum_{k=n}^\infty \gamma_k \big)^2\ .
\end{equation}
As a consequence, the BO frequencies  $\omega_n=\frac{\partial \mathcal H}{\partial \gamma_n}$, $n \ge 1$, read
\begin{equation}\label{formula frequencies}
\omega_n 
=n^2-2\sum_{k=1}^\infty \min(k,n)\gamma_k\ .
\end{equation}
Thus for any $u_0\in L^2_{r,0}$, the solution $u$ of the BO equation \eqref{BO} with  $u(0)=u_0$, when expressed in Birkhoff coordinates, is given by 
\begin{equation}\label{BOsolution}
\zeta_n(u(t))=\zeta_n(u_0)\, 
e^{i\omega_n(u_0) t}\ ,\quad t\in \R .
\end{equation}
\end{proposition}
\begin{remark}\label{extension of frequencies}
(i) Note that the right hand side of \eqref{formula frequencies} is well defined for any $u \in L^2_{r,0}$.
Hence the BO frequencies continuously extend to $L^2_{r, 0}$. 
(ii) Note that by the trace formula of Proposition \ref{formulae} (ii), the formula \eqref{formula frequencies}
for $\omega_n$ can be written as
$$
\omega_n = 
n^2 - \|u\|^2 + 2\sum_{k= n + 1}^\infty (k - n) \gamma_k \ .
$$
The sum $2\sum_{k= n + 1}^\infty (k - n) \gamma_k$ satisfies the asymptotics
$$
 2\sum_{k= n + 1}^\infty (k - n) \gamma_k = o(1) \quad  \text{as} \quad n \to \infty
$$
and hence in particular, $\lim_{n \to \infty} (\omega_n - n^2) = - \|u\|^2$.
\end{remark}
\begin{proof}
We view $\mathcal H(u)$ as a function of the actions $\gamma_n,$ $n \ge 1$, given by formula \eqref{formula A for H(u)}
and write the claimed identity \eqref{formula Hgamma} as
\begin{equation}\label{Hgamma}
\mathcal H(u)=\sum_{n=1}^\infty n^2\gamma_n-\sum_{n=1}^\infty s_n^2\ .
\end{equation}
Note that both sides are real analytic functionals on the space $\ell^{1,2}(\N, \R)$, defined as the Banach space of real sequences $(\xi_n)_{n \ge 1}$ satisfying
$$\sum_{n=1}^\infty n^2|\xi_n|<+\infty \ .$$
By \eqref{formula A for H(u)}, the partial derivative $\frac{\partial \mathcal H}{\partial \gamma_r}$, $r\ge 1$, is given by
$$\frac{\partial \mathcal H}{\partial \gamma_r}=-s_1^2+(r-s_r)^2-\sum_{n<r}\gamma_n (2(n-s_n)+\gamma_n )\ ,$$
and hence for any $m\ge r$,
$$\frac{\partial ^2\mathcal H}{\partial \gamma_r\partial \gamma_m}=-2s_1-2(r-s_r)+2\sum_{n<r}\gamma_n=-2r\ ,$$
while the corresponding second partial derivatives of the right hand side of \eqref{Hgamma} all equal $-2r$ as well. Therefore the difference of the left hand side and the right hand side of 
\eqref{Hgamma} is a linear function of the $\gamma_r$, which moreover vanishes at $\gamma=0$.
Since the first derivative of this difference with respect to $\gamma_r $ at $\gamma=0$ is 
clearly $0$, this proves \eqref{Hgamma}. 
Finally, by a straightforward computation one has
$\omega_n = n^2 - 2 \sum_{k=1}^n s_k$. Since
$$
\sum_{k=1}^n s_k = n \sum_{j=n+1}^\infty \gamma_j
+ \sum_{k=1}^n \sum_{j=k}^n \gamma_j = 
\sum_{j=1}^{\infty} \min{(j,n)} \gamma_j
$$
the claimed formula
\eqref{formula frequencies} follows.
 Finally, formula \eqref{BOsolution}, expressing the BO solution $u(t)$ with initial data $u_0$ in terms of Birkhoff coordinates, is an immediate consequence of the properties of the Birkhoff map, applied to the BO Hamiltonian 
$\mathcal H$, and the wellposedness result of \cite{MP}.\end{proof}

As a further application of the properties of Birkhoff map we show that the isospectral set of any potential $u \in L^2_{r,0}$,
$$
{\rm{Iso}}(u) = \{ v \in L^2_{r,0} : \mathcal H_\lambda (v)= \mathcal H_\lambda(u) \,\, \forall \lambda \in \mathbb C \setminus\{ \lambda_n(u)\,: \, n \ge 0 \}\}\ .
$$
is an invariant torus in $L^2_{r,0}$ of the BO equation.
To state this result in more detail, introduce for any sequence $(\zeta_n)_{n \ge 1}$ in $h^{1/2}_+$
$$
{\rm{Tor}}((\zeta_n)_{n \ge 1}) 
:= \{ (z_n)_{n \ge 1}  \in h^{1/2}_+ \, :
 \, | z_n | = |\zeta_n| \,\,  \forall \, n \ge 1 \}\, .
$$
It then follows from the construction of the Birkhoff map in a straightforward way that the following result holds:
\begin{corollary}\label{invariant tori}
For any $u \in L^2_{r,0},$
$$
\Phi ({\rm{Iso}}(u)) = \rm{Tor}( \Phi(u)) \,.
$$
In particular, $\rm{Iso}(u)$, being homeomorphic to a countable product of circles, is compact and connected.
Since $\rm{Iso}(u)$ is invariant by the flow of the BO equation, 
we say -- by a slight abuse of terminology -- that 
$\rm{Iso}(u)$ is an invariant torus. Note that
$L^2_{r,0}$ is the union of such tori. 
\end{corollary}

Finally we prove Theorem \ref{almost periodic}. \\
{\em Proof of Theorem \ref{almost periodic}} The solutions with initial data in $L^2_{r,0}$
can be constructed using Theorem \ref{main result} and Proposition \ref{BO frequencies}.
Note that for any $u_0 \in L^2_{r,0},$ the solution $t \mapsto u(t)$ with $u(0) = u_0$ stays
on $\text{Iso}(u_0)$ and hence by Proposition \ref{isocompact} (cf. also Proposition \ref{invariant tori}), the orbit of the solution 
is relatively compact in $L^2_{r,0}$.
In order to prove that $t\in \R\mapsto u(t)\in L^2_{r,0}$ is almost periodic, we appeal to Bochner's  characterization of such functions (cf. e.g. \cite{LZ}) : a bounded continuous function $f:\R \to X$
with values in a Banach space $X$ is almost periodic if and only if the set $\{ f_\tau, \tau \in \R\}$  of functions defined by 
$f_\tau (t):=f(t+\tau)$
is relatively compact in the space $\mathcal C_b(\R, X)$ of bounded continuous functions on $\R $ with values in $X$.
Since $\Phi $ is a homeomorphism, in the case at hand, it suffices to prove that for every sequence $(\tau_k)_{k\ge 1}$ of real numbers,  the sequence 
$f_{\tau_k}(t):=\Phi (u_{\tau_k}(t))$, $k\ge 1,$ 
in $\mathcal C_b(\R,h^{1/2}_+)$ admits a subsequence which converges uniformly in
$\mathcal C_b(\R,h^{1/2}_+)$. Notice that
$$f_{\tau_k}(t)=\big(\zeta_n(u(0)){\rm e}^{i\omega_n(t+\tau_k)}\big)_{n\ge 1}.$$
 By Cantor's diagonal process
and since the circle is compact, 
there exists a subsequence of $(\tau_k)_{k\ge 1}$,
again denoted by $(\tau_k)_{k\ge 1}$, so that for any $n \ge 1,$ $\lim_{k \to \infty}{\rm e}^{i\omega_n\tau_k}$ exists,
implying that the sequence of functions $f_{\tau_k}$
converges 
uniformly in $\mathcal C_b(\R,h^{1/2}_+)$.
\hfill $\square$

\bigskip


%

\appendix

\section{On the Lax pair for the BO equation}\label{verification Lax pair equation}
The purpose of this appendix is to derive the Lax pair formulation of the Benjamin-Ono equation,
described in Section \ref{Introduction}. The corresponding computations 
for the Benjamin-Ono equation on the line can be found in \cite[Appendix A]{Wu}.
In order to be comprehensive, we include its derivation.\\
Recall that for any $u \in L^2_r$, we introduced the (unbounded) operators $L_u$ and $B_u$ on $L^2_+$,
$$
L_u  = -i \partial_x  - T_u \,, \qquad B_u  := - i \partial^2_x  + 2 T_{ \partial_x( \Pi u)}  - 2\partial_xT_u 
$$
where $T_u: L^2_+ \to L^2_+$ denotes the Toeplitz operator given by $T_u f = \Pi(u f)$ and $\Pi : L^2 \to L^2_+$ the Szeg\H{o} projector.
We claim that for any smooth function $u(t, x)$ on  $ \R \times \T$, and any $h \in L^2_+$
\begin{equation}\label{Lax pair formulation}
\left ( \frac d{dt}\ L_u + [L_u, B_u ] \right ) h = - \Pi \big( \big(\partial_t u + 2u\partial_xu - H\partial^2_{x}u\big) h \big)
\end{equation}
To verify this identity, we write
$$
[L_u, B_u ] = -i[\partial_x, B_u] - [T_u, B_u] = Q_0 + Q_1 + Q_2
$$
where for any $0 \le k \le 2$, $Q_k$ is an operator, homogenous of degree $k$ in $u$.
One computes $Q_0 = - [\partial_x, \partial_x^2] = 0$ and
$$
Q_1 = - 2i [\partial_x,  T_{ \partial_x( \Pi u)}  - \partial_xT_u]  + i [T_u, \partial^2_x]\,,
$$
$$
Q_2 = - 2 [T_u,  T_{ \partial_x( \Pi u)}  - \partial_xT_u ] \,.
$$
Let us first discuss $Q_2$. For any $h \in L^2_+$, one has 
$$
T_u \big( T_{ \partial_x( \Pi u)}h \big) = \Pi \big( \partial_x(\Pi u) u h  \big)\,, \quad
T_{ \partial_x( \Pi u)} \big(T_u h \big) = \Pi \big( \partial_x(\Pi u) \Pi(uh)). 
$$
Since $\Pi (I - \Pi) =0$ one then obtains
$$
 [T_u,  T_{ \partial_x( \Pi u)}] h = \Pi \big( \Pi(\partial_x u)  (I - \Pi)(uh)   \big) =  \Pi \big( \partial_x u) (I - \Pi)(uh)   \big).
$$
On the other hand, $[T_u, \partial_xT_u ] h = - \Pi \big(\partial_x u \Pi(uh)  \big)$. Combining the two results then leads to
$$
Q_2 h = -2 \Pi \big( (\partial_x u) u h \big) = - \Pi \big( (2u\partial_x u) h  \big)\,.
$$
Concerning the operator $Q_1$, note that 
$$
- 2i [\partial_x,  T_{ \partial_x( \Pi u)}] = - 2iT_{ \partial^2_x( \Pi u)}\,, \quad 
2i  [\partial_x,  \partial_xT_u] = 2i \partial_xT_{\partial_x u}\,, 
$$
and
$$
i [T_u, \partial^2_x] = -i T_{ \partial^2_x u} - 2iT_{\partial_x u} \partial_x\,.
$$
Combining the three results then leads to
$$
Q_1 h = - 2iT_{ \partial^2_x( \Pi u)} h  + i T_{ \partial^2_x u} h 
= -2i \Pi\big( \Pi(\partial_x^2 u) h \big) + i \Pi\big( (\partial_x^2 u) h \big)
$$
or $Q_1 h =  - \Pi \big( i(2 \Pi - I)(\partial_x^2 u)  h \big)$. Since $ i(2 \Pi - I)(\partial_x^2 u) =  - H\partial_x^2 u$
we conclude that
$$
Q_1 h =  \Pi \big( (H\partial_x^2 u ) h \big).
$$
Finally, one has $\partial_t L_u h= - \Pi (\partial_t u h)$. Altogether, we thus have proved that
$$
\left ( \frac d{dt} L_u + [L_u, B_u ] \right ) (h) = - \Pi \big( \big(\partial_t u + 2u\partial_xu - H\partial^2_{x}u \big) h \big).
$$

\bigskip


\section{Traveling waves and one gap potentials }\label{general one gap potentials}
In this appendix we identify traveling wave solutions of the Benjamin Ono equation as one gap potentials, which we describe in more detail, obtaining  a new proof of the result
of Amick--Toland in \cite{AT}. As we already observed, we may 
restrict our analysis to solutions with average $0$.
\begin{proposition}\label{traveling}
The nonzero traveling wave solutions of the Benjamin--Ono equation coincide with the solutions with initial data given by
one gap potentials.
\end{proposition}
 \begin{proof}
A traveling wave solution is of the form 
$u(t,x)=u_0(x+ct)$
for some $c\in \R $, or, with our notation,
$u(t)=Q_{ct}u_0$.
In view of  identity \eqref{identity for (f_n , 1)}, we have
$$\zeta_n (Q_\tau v)=e^{i\tau n}\zeta_n (v)\ ,\ v\in L^2_{r,0}\ .$$
Comparing with the general expression \eqref{BOsolution} of the BO solution, we infer
$$e^{icnt}\zeta_n (u_0)=e^{i\omega_n(u_0) t}\zeta_n(u_0)\ .$$
Hence for any $n \ge 1$ with $\zeta_n(u_0)\ne 0$,
one has $cn=\omega_n(u_0)$.
In view of formula \eqref{formula frequencies}, the mapping $n\mapsto \omega_n/n $ is strictly increasing. Consequently, there cannot be more than one $n$ 
with $\zeta_n(u_0) \ne 0$.
If $u_0$ is not identically $0$, this precisely means that $u_0$ is a one gap potential.
\end{proof}  
Let $N \ge 1$ be given and introduce  $\mathcal O_{ \{ N\}} := G_ {\{ N\}} \cap L^2_{r,0}$.
Without further reference, we use the notation introduced in the main body of the paper.
Denote by $\Phi_{ \{ N\}}$ the restriction of the Birkhoff map $\Phi$ to $\mathcal O_{ \{ N\}}$. It is given by
$$
\Phi_{ \{ N\}} : \mathcal O_{ \{ N\}} \to \C^\ast, \, u \mapsto \zeta_N(u) =  \frac{1}{\sqrt{\kappa_N}} \langle 1 \, | \, f_N  \rangle 
$$
where $\kappa_N  = 1 / (N + \gamma_N)$ and hence
\begin{equation}\label{formula for 1 f_N}
\langle 1 \, | \, f_N \rangle  = \frac{ 1 }{ \sqrt{N + \gamma_N} } \zeta_N\ .
\end{equation}
By Theorem \ref{Ngap} in Section \ref{finite gap potentials},
$$
\Pi u(z) = - z Q'(z) / Q(z)
$$
where by \eqref{formula Q(z)},  $Q(z) = \det (Id - zM_{N-1})$ and   
 $M_{N-1}$ is the $N \times N$ matrix $(M_{np})_{0\le n,  p\le N-1}$
with coefficients $M_{np}$ given by
\begin{equation}
M_{np}=\begin {cases} \delta_{p,n+1}\ &{\rm if}\ \zeta_{n+1}=0\ , \nonumber\\
\sqrt{\mu_{n+1}}\gamma_{n+1}\frac{ \langle f_p,1 \rangle} { (\lambda_p-\lambda_n-1) \langle f_{n+1} | 1\rangle}\  &{\rm if}\  \zeta_{n+1} \ne 0\ .
\end{cases}
\end{equation}
Since $u \in G_{ \{N \}}$ and hence $ \langle 1 | f_p \rangle = 0$ for any $p \ne 0, N$ one computes 
$$
Q(z) = 1 + z^N \frac{\sqrt{\mu_N} \gamma_N}{N} \frac{\langle f_0 | 1 \rangle}{\langle f_N | 1 \rangle}
$$
where by \eqref{fSf}
$$
\mu_{N} = 1-\frac{\gamma_{N}}{\lambda_{N}-\lambda_0} = \frac{N}{N + \gamma_N}.
$$
To compute $ \langle 1 | f_0 \rangle$, use again that $ \langle 1 | f_k \rangle = 0$ for $k \ne 0, N$ to infer
 $1 =  \langle 1 | f_0 \rangle f_0 + \langle 1| f_N\rangle f_N$. By the definition of $f_0$ one has $\langle 1 \, | \, f_0 \rangle > 0$ and hence
\begin{equation}\label{formula for 1 f_0}
\langle 1 | f_0 \rangle   = (1 -  | \langle 1 | f_ N \rangle |^2)^{1/2} = \frac{\sqrt N}{\sqrt{N + \gamma_N}} \,.
\end{equation}
Since $\zeta_N = \sqrt{\gamma_N} e^{ i \varphi_N}$ and $\langle f_N | 1 \rangle = \overline{\langle 1 |  f_N \rangle}$, one obtains, using \eqref{formula for 1 f_N} and \eqref{formula for 1 f_0},
$$
Q(z) = 1 + z^N \frac{\sqrt{\mu_N} \gamma_N}{N} \frac{\sqrt N}{\overline{\zeta_N}} =  1 + z^N  \frac{1}{\sqrt{N + \gamma_N}} \zeta_N \ .
$$
Combining the results obtained leads to the following formula
$$
\Pi u(z) = - z Q'(z) / Q(z) = N \frac{ w z^N}{1 - w z^N} \, , \qquad w:= - \langle 1 | f_N \rangle 
= - \frac{1}{\sqrt{N + \gamma_N}} \zeta_N
$$
or by substituting $e^{ix}$ for $z$, 
$$
u(x) = N \frac{ w e^{iNx}}{1 - we^{iNx}}  +  N \frac{ \overline{w} e^{-iNx}}{1 - \overline{w}e^{-iNx}} \,.
$$
Note that $ 0 < |w| = \frac{\sqrt{\gamma_N}}{\sqrt{N + \gamma_N}} < 1$,  
showing that the potential $u(x + iy)$ is analytic in a strip $|y| < y_0$.

For any $u \in \mathcal O_{\{ N \} }$ , the eigenvalues of $L_u$ are given by
$$
\lambda_n =  n - \gamma_N \,, \,\,\, \,  \forall \, 0 \le n < N\,, \qquad \lambda_n =  n \,, \,\,\, \forall \, n \ge N\,,
$$
\
and
$$
\gamma_N = N \frac{|w|^2}{ 1- |w|^2}\,, \qquad \gamma_n = 0\,, \quad \forall n \ne N.
$$
 In a straightforward way one verifies that 
the eigenfunctions $f_n$, $n \ge 0$, corresponding to the eigenvalues of $\lambda_n$, $n \ge 0$, are given by
 \begin{eqnarray*}
 f_n(x)&=& e^{inx} \frac{\sqrt{1 - |w|^2}}{1 - w e^{iNx}}\,, \quad \forall \, 0 \le n < N\ ,\\
 f_n(x) &=& e^{i(n-N)x} 
\left( \frac{(1 - |w|^2)e^{iNx}}{1 - w e^{iNx}} -\overline{w} \right)
 \,, \quad \forall \, n \ge N \,.
 \end{eqnarray*}


 \section{On the spectrum of $L_u$}\label{L_u on the line}
  
By Proposition \ref{simple}, for any $u \in L^2_r,$ 
the spectrum of $L_u$ is discrete and the eigenvalues 
$\lambda_n(u)$, $n \ge 0$, of $L_u$ satisfy
$\lambda_0(u) < \lambda_1(u) < \cdots$ with 
$\gamma_n(u)= \lambda_n(u) - \lambda_{n-1}(u) -1 \ge 0$
for any $n \ge 1.$ The purpose of this appendix is to provide a spectral interpretation of the real numbers
$\gamma_n(u),$ $n \ge 1$, by considering the operator $L_u$ on the real line $\R$.\\
Denote by $L^2(\R) \equiv L^2(\R, \C)$ the standard 
$L^2-$ space with corresponding inner product
$\langle f , g \rangle = \int_{-\infty}^{\infty}f(x) \overline{g(x)} dx $ and for any $f \in L^2(\R)$ by $\mathcal F(f)$ its Fourier transform,
$$
\mathcal F(f)(\eta) = \int_{-\infty}^{\infty} f(x) 
e^{-ix\eta} dx\ .
$$
Furthermore, let $L^2_+(\R)$ be the closed subspace of $L^2(\R)$ given by
$$
L^2_+(\R ) = \{ f \in L^2(\R) \ : \  \mathcal F(f)(\eta) = 0 \quad \forall \, \eta < 0 \}\ ,
$$
$\Pi^\R: L^2(\R) \to L^2_{+}(\R)$
the corresponding orthogonal projector,
$H^1(\R) \equiv H^1(\R, \C)$ the standard $H^1-$Sobolev space and $H^1_+(\R):= H^1(\R)\cap L^2_+(\R).$
For any $u \in L^2_r,$ denote by $L_u^{\R}$ the operator
on $L^2_+(\R)$ with domain $H^1_+(\R)$ given by
$$
L_u^{\R}f = -i \partial_x f - \Pi^{\R}(uf)\ , \quad \forall \, f \in H^1_+(\R)\ .
$$
To see that $L_u^{\R}$ is selfadjoint, we first need
to establish the following
\begin{lemma}\label{relative bound} 
For any $\epsilon > 0$, 
$u \in L^2_r$, and $f \in H^1(\R)$,
$$
\int_{-\infty}^{\infty}|uf|^2 \ d x \le
\|u\|^2\Big(\epsilon \|f' \|^2_{L^2(\R)} + 
\left (1 + \frac{1}{\epsilon} \right )\|f\|^2_{L^2(\R)} \Big)\ .
$$
\end{lemma}
\begin{proof} For any $u \in L^2_r$ and $f \in H^1(\R)$ one has
$$\int_{-\infty}^{\infty} |u(x)f(x)|^2\, dx =
\int_0^{2\pi}|u(x)|^2\sum_{n=-\infty}^{+\infty}
|f(x+2n\pi)|^2\, dx\ .$$
Let  $g(t; x):=f(x+2\pi t)$. 
It then suffices to show that for any 
$0 \le x \le 2\pi,$
$$
\sum_{n=-\infty}^{+\infty} |g(n; x)|^2 \le
\epsilon \|f' \|^2_{L^2(\R)} + 
\left (1 + \frac{1}{\epsilon} \right )\|f\|^2_{L^2(\R)}\ .
$$
For any given $0 \le x \le 2\pi,$  $n\in \Z$, and $t\in [n,n+1]$, 
$$
|g(t; x)|^2-|g(n; x)|^2=
\int_n^t 2{\rm Re}
\big(\overline{ g(s; x)}\partial_sg(s; x)\big)\, ds
$$
and hence 
$$\Big||g(t; x)|^2-|g(n; x)|^2 \Big|\leq \int_n^{n+1}
\big(\ \e |\partial_sg(s; x)|^2 + \frac 1\e |g(s; x)|^2 \ \big)\, ds\ .$$
Integrating in $t\in [n,n+1]$, one gets
$$|g(n; x)|^2\leq \int_n^{n+1}
\Big( \ \e |\partial_tg(t; x)|^2 + \big(1+\frac 1\e \big)|g(t; x)|^2 \Big)\, dt $$
and then summing over $n\in \Z$ yields the claimed inequality.
\end{proof}
Clearly, $-i \partial_x$ is a selfadjoint operator on $L^2_+(\R)$ with domain $H^1_+(\R)$. By the Kato-Rellich theorem 
(cf. \cite{RS}) one then infers
from Lemma \ref{relative bound} that $L^\R_u$ is
selfadjoint as well. 
\begin{proposition}\label{Rspectrum}
For any $u \in L^2_r,$ the spectrum $spec(L^\R_u)$
of $L^\R_u$ is absolutely continuous and consists of a union of bands,
$$
spec(L^\R_u) = 
\bigcup_{n=0}^{\infty} [\lambda_n(u), \lambda_n(u) +1]\ .
$$ 
Hence, for any $n \ge 1$, $\gamma_n(u)$ is the length of the nth gap in the spectrum of $L^\R_u$.
\end{proposition}
To prove Proposition \ref{Rspectrum}, we first need to make some preliminary considerations. Introduce the following $L^2-$space
$$
L^2\big(\T \times [0,1]) \equiv 
L^2\big(\T \times [0,1], \C; \frac{1}{4\pi^2} dx\, d\xi\big)
$$ 
which is clearly isometric to 
$L^2\big([0,1], L^2(\T); \frac{1}{2\pi} d\xi\big)$. 
Furthermore
for any $f\in L^2(\R )$, denote by $\psi[f]\in L^2(\T\times [0,1])$ the function
$$\psi[f](x,\xi )=\sum_{n\in \Z}{\rm e}^{inx}
\mathcal F( f)(\xi +n)\ .$$
Since
\begin{equation}\label{normofpsi}
\int_{-\infty}^{\infty} |f(x)|^2\, dx =
\frac{1}{4\pi^2}\int_{0}^{2\pi}
\Big( \int_0^1|\psi[f](x,\xi)|^2\, d\xi \Big) \, dx\ ,
\end{equation}
\begin{equation}\label{psitof}
f(x)=\frac{1}{2\pi}\int_0^1{\rm e}^{ix\xi}\, \psi[f](x,\xi)\, d\xi \ ,
\end{equation}
it follows that the linear map 
$L^2(\R) \to L^2\big(\T \times [0,1]\big), \, f \mapsto \psi [f]$ 
is unitary. Restricting it to the Sobolev space $H^1(\R)\equiv H^1(\R, \C)$ one sees that $H^1(\R)$
is isometric to 
$L^2([0,1],H^1(\T ); \frac{1}{2\pi} d\xi)$.
Finally, we claim  that for any $0 \le \xi < 1,$
\begin{equation}\label{Piandpsi}
 \psi[\Pi^\R f](\cdot,\xi )=\Pi \big(\psi[f](\cdot ,\xi)\big)\ .
\end{equation}
Indeed, for any $0 \le \xi < 1, \, $ 
$\psi[\Pi^\R f](x,\xi )=\sum_{n\in \Z}{\rm e}^{inx}
\mathcal F(\Pi^\R f)(\xi +n)$
is given by
$$\psi[\Pi^\R f](x,\xi ) =\sum_{n=0}^\infty {\rm e}^{inx}\mathcal F (f)(\xi +n)=\Pi (\psi[f](\cdot,\xi))(x)\ .$$
\big(Note that for $\xi = 1,$ one has
$\psi[\Pi^\R f](x, 1) =\sum_{n=-1}^\infty {\rm e}^{inx}\mathcal F (f)(1 +n)$ which equals $\Pi (\psi[f](\cdot,1))(x) + e^{-ix} \mathcal F(f)(0)\ .\big)$
\begin{lemma}\label{Landpsi}
For any $f\in H^1_+(\R )$ and $0 \le \xi < 1$, 
$$\psi [L_u^\R f](\cdot,\xi )=
(\xi + L_u)(\psi [f](\cdot,\xi ))\ .$$
\end{lemma}
\begin{proof}
Let $f \in H^1_+(\R)$ and $0 \le \xi < 1$.
In view of the definition of $\psi [f]$,
$$\psi [Df](x,\xi)=(\xi +D_x)\psi[f](x,\xi )\ .$$
Since $u$ is $2\pi $ periodic, one has
$\psi [uf](x,\xi )=u(x)\psi[f](x,\xi )$ and hence
one infers from \eqref{Piandpsi} that
$$
\psi[\Pi^\R (uf)](\cdot,\xi )=\Pi \big(\psi[uf](\cdot ,\xi) \big) = \Pi \big( u \psi[f](\cdot,\xi) \big)
= T_u \psi[f](\cdot ,\xi)\ .
$$
Combining these computations one obtains the claimed formula.
\end{proof}
\noindent
{\em Proof of Proposition \ref{Rspectrum}.}
Let $u \in L^2_r.$
For any $f \in L^2_+(\R)$ introduce the sequence
$(\psi_n^u[f](\xi ))_{n \ge 0}$ where for any $0 \le \xi \le 1,$
$$ \psi_n^u[f](\xi ) :=\langle \psi [f](\cdot,\xi )\, \vert \, f_n(\cdot,u)\rangle \ .$$
Then 
$$
L^2_+(\R ) \to \ell ^2\Big(\Z_{n \ge 0},\,  L^2\Big([0,1],\frac {1}{2\pi} d \xi\Big)\Big), \, 
f \mapsto (\psi_n^u[f])_{n\ge 0} 
$$ 
is a unitary operator, and, for every $f\in H^1_+(\R )$,  
Lemma \ref{Landpsi} implies 
\begin{equation}\label{fourier}
\psi_n^u[L_u^\R f](\xi )=(\xi +\lambda_n(u))\psi_n^u[f](\xi )\ ,\ 0\le \xi <1\ .
\end{equation}
Consequently, the spectral measure $\mu_f$ of a 
function $f\in H^1_+(\R )$, acting on any given continuous function $\varphi : \R \to \R$ with compact support, 
$ \int_\R \varphi (\lambda )\, d\mu_f(\lambda )
=\langle \varphi (L_u^\R)f\vert f\rangle$, 
can be computed as
\begin{eqnarray*}
\int_{-\infty}^{\infty} \varphi (\lambda )\, d\mu_f(\lambda )
&=&\sum_{n=0}^\infty \frac{1}{2\pi}\int_0^1\varphi (\xi +\lambda_n(u))|\psi_n^u[f](\xi )|^2\, d\xi \\
&=&\frac{1}{2\pi}\sum_{n=0}^\infty \int_{\lambda_n(u)}^{\lambda_n(u)+1}\varphi (\lambda)|g_n^f(\lambda)|^2\, d\lambda \ ,
\end{eqnarray*}
where for any $\lambda_n(u)\leq \lambda <\lambda_{n}(u)+1$, $g_n^f(\lambda )$ is given by
$\psi_n^u[f](\lambda -\lambda_n(u))\ .$
This proves that the spectrum of $L_u^{\R}$ is absolutely continuous and is contained  in the union 
$$\bigcup_{n=0}^{\infty} [\lambda_n(u), \lambda_n(u) +1]\ .$$
 In order to prove that $spec(L^\R_u)$ coincides with
 the latter union it is enough to show that any measurable subset $E$ of $[0,1]$ with the property that, for some $n$, 
 \begin{equation}\label{cancellation}
 {\bf 1}_E(\xi) \, \psi_n^u[f] (\xi) = 0 \ , \quad 
 \forall \, 0 \le \xi \le 1, \,  f\in L^2(\R )\ ,\ 
 \end{equation}
 has Lebesgue measure $0$. Note that \eqref{cancellation} says that
 $$
 \int_0^{2\pi} \int_{0}^1 {\bf 1}_E(\xi) \psi [f](x,\xi )\, \overline{f_n}(x,u)\, dx\, d\xi =0\ ,
 \qquad \forall f\in L^2(\R )\ , $$
 or, since $\psi$ is unitary, 
 $$\psi ^{-1}[f_n\otimes {\bf 1}_E]=0\ .$$
 In view of \eqref{psitof}, one then concludes that
 $$f_n(x)\int_0^1  
 {\bf 1}_E(\xi) {\rm e}^{ix\xi}\, d\xi =0\ ,
 \qquad \forall x\in \R\ .$$
 Since $f_n$ is continuous and is not identically $0$, 
 the Fourier transform of ${\bf 1}_E$ then
 vanishes on some nonempty open set of $\R $. Since the Fourier transform of ${\bf 1}_E$ is an entire function, it then must vanish identically, implying that  $E$ has measure $0$.
 This completes the proof of Proposition \ref{Rspectrum}.
\hfill $\square$

\bigskip


\end{document}